\documentclass[11pt]{article}
\usepackage[a4paper,margin=1in]{geometry}
\usepackage{amsmath,amsfonts,amssymb,amsthm,mathtools}
\usepackage[T1]{fontenc}
\usepackage[utf8]{inputenc}
\usepackage{lmodern}
\usepackage{microtype}
\usepackage{mathrsfs}
\usepackage{graphicx}
\usepackage{color}
\usepackage{url}
\usepackage{enumitem}
\usepackage{cases}
\usepackage{appendix}
\usepackage[hypertexnames=false]{hyperref}
\hypersetup{colorlinks=true,linkcolor=red,citecolor=blue,urlcolor=blue}

\makeatletter
\renewcommand\section{\@startsection{section}{1}{\z@}%
  {3.5ex \@plus 1ex \@minus .2ex}%
  {2.3ex \@plus .2ex}%
  {\normalfont\Large\bfseries}}
\renewcommand\subsection{\@startsection{subsection}{2}{\z@}%
  {3.25ex \@plus 1ex \@minus .2ex}%
  {1.5ex \@plus .2ex}%
  {\normalfont\large\bfseries}}
\makeatother

\numberwithin{equation}{section}
\numberwithin{figure}{section}
\numberwithin{table}{section}

\newtheorem{theorem}{Theorem}[section]
\newtheorem{definition}[theorem]{Definition}
\newtheorem{proposition}[theorem]{Proposition}
\newtheorem{lemma}[theorem]{Lemma}
\newtheorem{remark}[theorem]{Remark}
\newtheorem{corollary}[theorem]{Corollary}

\newtheorem{theoremletter}{Theorem}

\newcommand{\R}{\mathbb{R}}

\newcommand{\norm}[1]{\left\lVert #1\right\rVert}

\newcommand{\one}{\mathbf 1}
\newcommand{\loc}{\mathrm{loc}}

\title{Morse index, Leray--Schauder degree and local uniqueness for multi-peak concentrating solutions of a fractional Schr\"odinger equation}
\author{Yinbin Deng
\thanks{School of Mathematics and Statistics and Key Laboratory of Nonlinear Analysis and Applications (Ministry of Education), Central China Normal University, Wuhan 430079, People's Republic of China. E-mail: \texttt{ybdeng@ccnu.edu.cn}.}
\and  Baiping Feng 
\thanks{School of Mathematics and Statistics, Central China Normal University, Wuhan 430079, People's Republic of China. E-mail: \texttt{bpfeng@mails.ccnu.edu.cn}.}
\and  Qing Guo
\thanks{College of Science, Minzu University of China, Beijing 100081, People's Republic of China. E-mail: \texttt{guoqing0117@163.com}.}
\and  Huafei Xie
\thanks{Department of Mathematics, Shantou University, Shantou 515063, People's Republic of China. E-mail: \texttt{hfxie@stu.edu.cn}.}
}
\date{}
\begin{document}
\maketitle

\begin{abstract}
We study positive \(k\)-peak solutions of the semiclassical fractional Schr\"odinger equation
\[
\varepsilon^{2s}(-\Delta)^s u+V(x)u=u^p
\quad\text{in }\mathbb R^N,
\]
concentrating at different nondegenerate critical points \(\xi_1^0,\ldots,\xi_k^0\) of \(V\). For every such family satisfying the natural energy quantization, we determine the complete low spectrum of the linearized operator. The first \(k\) eigenvalues remain uniformly negative, the next \(kN\) eigenvalues are of order \(\varepsilon^2\) and are governed by the Hessians \(D^2V(\xi_j^0)\), while the remaining spectrum is uniformly separated from zero. Consequently, the Morse index equals \(k\) plus the total number of negative eigenvalues of these Hessians, and every such solution is nondegenerate. Combining a unique modulation parametrization with a Leray--Schauder degree computation, we further prove that, for all sufficiently small \(\varepsilon\), the prescribed concentrating class contains exactly one positive solution. The result applies to the whole energy-quantized class, not only to a particular solution.
\end{abstract}

\par


\section{Introduction}
\subsection{Background}
In this paper, we consider the fractional nonlinear Schr\"odinger equation
\begin{equation}\label{eq:fractional-schrodinger-problem}
\begin{cases}
\varepsilon^{2s}(-\Delta)^s u + V(x)u = u^p, & \text{in } \mathbb{R}^N,\\
u>0, & \text{in } \mathbb{R}^N,\\
u\in H^{s}(\mathbb{R}^N),
\end{cases}
\end{equation}
where $\varepsilon>0$ is a small parameter, $0<s<1$, $N>2s$, and
$1<p<2_s^*-1$, with
\[
2_s^*:=\frac{2N}{N-2s}
\]
denoting the fractional critical Sobolev exponent.
Throughout the paper we assume that $V$ satisfies 
\[
V\in C^{2}_b (\mathbb{R}^N),\quad\inf_{\mathbb{R}^N}V(x)>0.
\]
We shall use the notation
\[
V_{\min}:=\inf_{\mathbb{R}^N}V(x)>0,
\qquad
V_{\max}:=\|V\|_{L^\infty(\mathbb{R}^N)}.
\]
For $u\in\mathcal{S}(\mathbb{R}^N)$, the fractional Laplacian $(-\Delta)^s$ may be defined by the singular integral
\[
(-\Delta)^s u(x)=C_{N,s}\,\mathrm{P.V.}\int_{\mathbb{R}^N}\frac{u(x)-u(y)}{|x-y|^{N+2s}}\,dy,\qquad x\in\mathbb{R}^N,
\]
where $\mathrm{P.V.}$ denotes the principal value and
$C_{N,s}=\frac{2^{2s}s\,\Gamma\!\left(\frac{N}{2}+s\right)}{\pi^{\frac{N}{2}}\Gamma(1-s)}$. This operator is well defined on $\mathcal{S}$, the Schwartz space of rapidly decreasing $C^{\infty}$ functions in $\mathbb{R}^N$, and can equivalently be defined through the Fourier transform: 
\[\mathcal{F}\left((-\Delta)^s u\right)(\eta)=\left|\eta\right|^{2s}\mathcal{F}(u)(\eta),\]
where $\mathcal{F}(u)$ denotes the Fourier transform of $u$; see, for example, \cite{cabre2014nonlinear,Caffarelli08082007}. The weak formulation of the fractional Laplacian naturally leads to the study of the fractional Sobolev spaces
\[H^s(\mathbb R^N):=\left\{u\in L^2(\mathbb R^N):\int_{\mathbb{R}^N}|\eta|^{2s}|\widehat{u}(\eta)|^2\,\mathrm{d}\eta<\infty\right\},\]
endowed with the norm
\[\|u\|_{H^{s}}^{2}:=\|u\|_{L^{2}}^{2}+\|u\|_{\dot H^{s}}^{2},\]
where $\|u\|_{\dot H^{s}}^{2}:=\int_{\mathbb{R}^{N}}|\eta|^{2s}|\widehat{u}|^{2}\mathrm{d}\eta$. By these definitions,
\[
\langle u,v\rangle_{\dot H^{s}(\mathbb{R}^N)}
=\int_{\mathbb{R}^N}(-\Delta)^{\frac{s}{2}}u\,(-\Delta)^{\frac{s}{2}}v\,dx
=\bigl\langle(-\Delta)^su,v\bigr\rangle_{H^{-s},H^s},
\qquad u,v\in H^s(\mathbb{R}^N),
\] where the last duality pairing is understood as the usual integral \(\int_{\mathbb R^N}(-\Delta)^s u\,v\,dx\). For further details on the fractional Laplacian operator, we refer to \cite{chen2020fractional,MR2944369,kwasnicki2017ten,radulescu2016variational} and the references therein.

Problem \eqref{eq:fractional-schrodinger-problem} arises as the stationary equation for semiclassical standing waves of a fractional nonlinear Schr\"odinger equation
\begin{equation}\label{eq:time-fractional-schrodinger}
i\varepsilon\frac{\partial\psi}{\partial t}
= \varepsilon^{2s}(-\Delta)^s\psi+(V(x)+E)\psi-|\psi|^{p-1}\psi,
\qquad (x,t)\in\mathbb{R}^N\times\mathbb{R}_+,
\end{equation}
where $E\in\mathbb{R}$ is the frequency.
Indeed, a solution of the form $\psi(x,t)=e^{-iEt/\varepsilon}u(x)$ satisfies \eqref{eq:time-fractional-schrodinger} if and only if $u$ solves \eqref{eq:fractional-schrodinger-problem}. The fractional Schr\"odinger equation was introduced by Laskin \cite{laskin2002fractional}. Since then, the fractional Laplacian has found important applications in many scientific fields, including biological modeling, physics and mathematical finance, and it is closely related to stable L\'evy processes; see \cite{applebaum2009levy}. Existence, multiplicity and regularity for fractional Schr\"odinger equations have been widely investigated; see, for example, \cite{barrios2012some,brandle2013concave,cabre2014nonlinear,dipierro2012existence,MR3070568,frank2016uniqueness,silvestre2007regularity} and the references therein. 

The construction of semiclassical concentrating solutions has been extensively studied. For fractional nonlinear Schr\"odinger equations, D\'avila, del~Pino and Wei \cite{MR3121716} constructed multi-peak solutions by Lyapunov--Schmidt reduction, while Deng, Peng and Yang \cite{YSX} investigated existence, nonexistence, and decay under different assumptions on the potential. These works mainly concern the existence, construction, and decay of concentrating solutions.

The present paper concerns a different, genuinely qualitative question. Once a \(k\)-peak pattern has been prescribed, how many positive solutions realize that pattern, and what is the local variational structure of each of them? A construction of one solution does not by itself determine the Morse index, exclude a kernel of the linearized operator, or rule out other solutions with the same concentration profile. These issues require a class-wide analysis rather than an expansion along one preselected Lyapunov--Schmidt solution. We consider arbitrary positive \(k\)-peak concentrating solutions \(u_\varepsilon\) satisfying Definition~\ref{def:kpeak-class} below and compute their Morse index and prove their nondegeneracy. We also prove uniqueness within the \(k\)-peak concentrating class for all sufficiently small \(\varepsilon\).

The Morse index is the number of negative eigenvalues of the linearized operator, counted with multiplicity. Morse index encodes substantial qualitative information about solutions and plays an important role in the analysis of their symmetry properties, singular behavior, and nodal sets; see, for instance, \cite{aftalion2004qualitative,bahri1992solutions,dancer2004stableII,dancer2005stable,farina2007classification,pacella2002symmetry,pacella2007symmetry,yang1998nodal}. In singular perturbation problems it also reveals how the infinite-dimensional spectrum remembers the finite-dimensional geometry selecting the peaks. Moreover, when a solution is nondegenerate, its Morse index determines the sign of its local Leray--Schauder degree. This last fact provides the bridge from spectral information to an exact count of solutions; see, for example, \cite{2002Existence,DeMarchisGrossiIanniPacella2019JMPA,ianni2025morse}.

Let us recall that the Morse index of a solution \(u_\varepsilon\) to \eqref{eq:fractional-schrodinger-problem} can be defined as follows.
\begin{definition}\label{def:morse-index}
  The Morse index of a solution $u_\varepsilon$ of problem \eqref{eq:fractional-schrodinger-problem} is the number of negative eigenvalues, counted with multiplicity, of the linearized eigenvalue problem
 \begin{equation*}
   \begin{cases}
    L_\varepsilon v:=\varepsilon^{2s}(-\Delta)^s v + V(x)v - p\,u_\varepsilon^{p-1}v = \mu v & \text{in }\mathbb{R}^N,\\
    v\in H^s(\mathbb{R}^N).
   \end{cases}
   \end{equation*}
\end{definition}

\begin{remark}
Let
\[
H_{\varepsilon,0}:=\varepsilon^{2s}(-\Delta)^s+V(x),
\qquad
L_\varepsilon=H_{\varepsilon,0}-p\,u_\varepsilon^{p-1}.
\]
Since $V, u_\varepsilon^{p-1}\in L^{\infty}(\mathbb{R}^N)$, $L_\varepsilon$ is a bounded perturbation of $\varepsilon^{2s}(-\Delta)^s$ and hence is self-adjoint on $L^2(\mathbb{R}^N)$. Moreover,
 \[
\langle H_{\varepsilon,0}v,v\rangle
=\varepsilon^{2s}\bigl\|(-\Delta)^{s/2}v\bigr\|_{L^2}^2+\int_{\mathbb{R}^N}V(x)|v|^2\,dx
\ge V_{\min}\|v\|_{L^2}^2,
\]
so $\sigma(H_{\varepsilon,0})\subset [V_{\min},\infty)$. Moreover, Lemma~\ref{lem:u-polynomial-decay} gives \(u_\varepsilon(x)\to0\) as \(|x|\to\infty\). Hence multiplication by \(u_\varepsilon^{p-1}\) is relatively compact with respect to \(H_{\varepsilon,0}\) (see \cite[Theorem XIII.14]{reed1978iv}). Therefore, by Weyl's theorem (see \cite[Theorem XIII.14, Corollary 2]{reed1978iv}),
\[
\sigma_{\mathrm{ess}}(L_\varepsilon)=\sigma_{\mathrm{ess}}(H_{\varepsilon,0})\subset [V_{\min},\infty).
\]
Thus, whenever \(\mu_\varepsilon<\inf\sigma_{\mathrm{ess}}(L_\varepsilon)\), it is a discrete eigenvalue of finite multiplicity (see \cite[Theorem VII.10]{reed1980methods}).
\end{remark}

For any $\varepsilon>0$ and $l\in\mathbb N$, we consider the linearized eigenvalue problem associated with \eqref{eq:fractional-schrodinger-problem}, namely,
\begin{equation}\label{eq:linearized-eigenvalue-problem}
\begin{cases}
\varepsilon^{2s}(-\Delta)^s v_\varepsilon^{(l)} + V(x)v_\varepsilon^{(l)} - p\,u_\varepsilon^{p-1} v_\varepsilon^{(l)}
= \mu_{\varepsilon}^{(l)} v_\varepsilon^{(l)} & \text{in }\mathbb{R}^N,\\[1mm]
\displaystyle \int_{\mathbb{R}^N}\Bigl(\varepsilon^{2s}\bigl|(-\Delta)^{s/2} v_\varepsilon^{(l)}\bigr|^2 + V(x)\bigl|v_\varepsilon^{(l)}\bigr|^2\Bigr)\,dx = \varepsilon^N.
\end{cases}
\end{equation}
   Here  \(v_\varepsilon^{(l)}\) is an eigenfunction corresponding to the \(l\)-th  eigenvalue \(\mu_{\varepsilon}^{(l)}\) of \eqref{eq:linearized-eigenvalue-problem}. For each $j\in\{1,\dots,k\}$, set $v_{\varepsilon,j}^{(l)}(x):=v_\varepsilon^{(l)}(\varepsilon x+\xi_{\varepsilon,j})$. A change of variables shows that $v_{\varepsilon,j}^{(l)}$ solves
\begin{equation*}
  \begin{cases}
(-\Delta)^s v_{\varepsilon,j}^{(l)} + V(\varepsilon x+\xi_{\varepsilon,j})v_{\varepsilon,j}^{(l)} - p u_\varepsilon^{p-1}(\varepsilon x+\xi_{\varepsilon,j})v_{\varepsilon,j}^{(l)} = \mu_{\varepsilon}^{(l)} v_{\varepsilon,j}^{(l)} &\text{ in } \mathbb{R}^N, \\
\int_{\mathbb{R}^N} (|(-\Delta)^{s/2} v_{\varepsilon,j}^{(l)}|^2 + V(\varepsilon x+\xi_{\varepsilon,j})|v_{\varepsilon,j}^{(l)}|^2) dx = 1,
\end{cases}
\end{equation*}
for any fixed $j\in\{1,\dots,k\}$. We order the eigenvalues of problem \eqref{eq:linearized-eigenvalue-problem} as
\[
\mu_{\varepsilon}^{(1)} < \mu_{\varepsilon}^{(2)} \leq \cdots \leq \mu_{\varepsilon}^{(j)} \leq \cdots.
\]
Recall that the \(l\)-th eigenvalue \(\mu_{\varepsilon}^{(l)}\) of problem \eqref{eq:linearized-eigenvalue-problem} admits the variational characterization:
\begin{equation*}
  \mu_{\varepsilon}^{(l)}
  = \inf_{\substack{W \subseteq H^s(\mathbb{R}^N) \\ \text{dim } W = l}}
  \max_{v \in W \setminus \{0\}}
  \frac{\varepsilon^{2s}\int_{\mathbb{R}^N}|(-\Delta)^{s/2} v|^2\,dx
  +\int_{\mathbb{R}^N} V(x)v^2\,dx
  -p\int_{\mathbb{R}^N}u_\varepsilon^{p-1}v^2\,dx}{\int_{\mathbb{R}^N} v^2\,dx}
\end{equation*}
for every $l \in \mathbb{N}$.

For a positive \(k\)-peak concentrating family, the expected low-spectrum structure has a natural geometric origin. The linearized operator at a single limiting ground state has one negative direction and an \(N\)-dimensional kernel generated by translations. Hence \(k\) peaks give \(k\) negative modes of order one and \(kN\) approximate translation modes near zero. The potential \(V\) breaks translation invariance, and the leading terms of these \(kN\) small eigenvalues are determined by the Hessian matrices at the concentration points. If the concentration points are nondegenerate, these small eigenvalues are nonzero for all sufficiently small \(\varepsilon\), while the remaining spectrum is uniformly separated from zero. This gives the Morse-index formula and nondegeneracy.

This spectral structure is well developed for local singularly perturbed and blow-up problems. In such problems, the small eigenvalues are often related to the Hessian of a reduced finite-dimensional functional, boundary curvature, the Robin function, or an interaction matrix associated with the blow-up configuration; see \cite{2002Existence,grossi2005eigenvalue,choi2016qualitative,gladiali2014morse,ianni2025morse}. For local nonlinear Schr\"odinger equations, Morse-index formulas for concentrating solutions were obtained in \cite{grossi2007morse,LuoPanPeng2024}. These works show that the Morse index is a precise spectral encoding of the geometry of concentration.

Nondegeneracy alone does not imply uniqueness. It shows that each solution is isolated, but it does not exclude several isolated solutions in the same concentrating class. Similarly, the construction of one family by Lyapunov--Schmidt reduction does not by itself exclude other families with the same concentration behavior. For local equations, uniqueness of single- and multi-peak solutions has been studied by degree arguments, refined asymptotic expansions, and localized Pohozaev identities; see \cite{CaoHeinz2003,CaoLiLuo2015,CaoNoussairYan1998,DLY,guo2017local}. The relation between the Morse index, degree, and uniqueness has also been used in blow-up problems; see \cite{DeMarchisGrossiIanniPacella2019JMPA,ianni2025morse}. For fractional equations, local uniqueness and nondegeneracy have been proved for several related perturbative or nearly critical problems; see \cite{CassaniWang2023,GuoLiLiuYang2026,WuZhang2024}. However, the perturbations and the classes of solutions considered in those works differ from the \(k\)-peak energy quantization condition here.

The central counting principle of this paper is therefore
\[
\text{total degree of the concentrating class}
=
\sum_u \text{local degree at }u
=
\sum_u(-1)^{m(u)},
\]
where the sums are taken over all solutions in the prescribed class. Our spectral theorem shows that all these solutions have the same Morse index and hence the same local-degree sign. If the class contains \(q\) solutions, the sum of their local degrees is therefore
\[
q(-1)^{m(u)}.
\]
The independently computed total degree is \((-1)^{m(u)}\), hence \(q=1\). This is stronger than uniqueness of a correction along a constructed solution: it gives uniqueness among all positive solutions satisfying the \(k\)-peak energy quantization.

We next recall the ground state of the limiting equation. This profile will be used to describe the peaks in our concentrating class; see \cite{MR3070568,frank2016uniqueness}.
\begin{theoremletter}\label{th:ground-state}
Consider the equation
\begin{equation}\label{eq:ground-state-equation}
(-\Delta)^s w+w=w^p,\quad w>0,\quad \text{in } \mathbb{R}^N,\quad w(0)=\max_{x\in \mathbb{R}^N}w(x),
\end{equation}
where \(N>2s\), \(s\in(0,1)\), and \(1<p<2_s^*-1\). The following statements hold.
 \begin{enumerate}
\item[(i)] (Uniqueness) The    positive ground state is unique up to translations; with the normalization in \eqref{eq:ground-state-equation}, the ground state \(w\in H^s(\mathbb R^N)\)  is unique.
\item[(ii)] (Symmetry, regularity and decay) The ground state \(w\) is smooth, radial, positive and strictly decreasing in \(|x|\).  Moreover, it satisfies the algebraic decay estimates
\begin{equation*}
\frac{C_1}{1 + |x|^{N + 2s}}\leqslant w(x)\leqslant \frac{C_2}{1 + |x|^{N + 2s}}
\qquad \text{for all }x\in\mathbb{R}^N,
\end{equation*}
with some constants \(C_2\geqslant C_1 > 0\).
\item[(iii)] (Nondegeneracy) The linearized operator \(L_0 = (-\Delta)^s + 1 - pw^{p - 1}\) is nondegenerate, i.e., its kernel is given by
\[
\ker L_0 = \operatorname{span}\left\{\frac{\partial w}{\partial x_1},\frac{\partial w}{\partial x_2},\ldots,\frac{\partial w}{\partial x_N}\right\}.
\]
Moreover, by \cite[Lemma C.2]{frank2016uniqueness}, for \(j = 1,\ldots,N\), \(\frac{\partial w}{\partial x_j}\) satisfies the decay estimate
\[
\left|\frac{\partial w}{\partial x_j}\right|\leqslant \frac{C}{1 + |x|^{N + 2s}}.
\]
\end{enumerate}
\end{theoremletter}

By scaling,
\[
w_\lambda(x):=\lambda^{\frac{1}{p-1}}w\left(\lambda^{\frac{1}{2s}}x\right)
\]
solves
\[
(-\Delta)^sw_\lambda+\lambda w_\lambda=w_\lambda^p
\quad\text{in }\mathbb{R}^N.
\]
Thus, for any \(\xi\in\mathbb R^N\), the profile
\[
u(x)=w_{V(\xi)}\left(\frac{x-\xi}{\varepsilon}\right)
\]
solves the corresponding equation
\[
\varepsilon^{2s}(-\Delta)^su+V(\xi)u=u^p
\quad\text{in }\mathbb R^N.
\]

For later use, set
\[
\|u\|_{\varepsilon,V}^2
:=\varepsilon^{2s}\int_{\mathbb R^N}|(-\Delta)^{s/2}u|^2\,dx+
\int_{\mathbb R^N}V(x)u^2\,dx,
\]
and define the semiclassical energy
\begin{equation*}
\mathcal E_\varepsilon(u)
:=\frac{1}{2}\|u\|_{\varepsilon,V}^2-\frac{1}{p+1}\int_{\R^N}|u|^{p+1}\,dx.
\end{equation*}
For the problem
\[
(-\Delta)^s v+\lambda v=v^p\quad\hbox{in }\R^N,
\]
set
\begin{equation*}
\mathcal I_\lambda(v):=\frac{1}{2}\int_{\R^N}\bigl(|(-\Delta)^{s/2}v|^2+\lambda v^2\bigr)\,dy
-\frac{1}{p+1}\int_{\R^N}|v|^{p+1}\,dy.
\end{equation*}
Finally, set
\[
a_0:=\int_{\R^N}w_1^{p+1}\,dy,
\qquad
c_*:=\frac{p-1}{2(p+1)}a_0,
\qquad
\theta:=\frac{p+1}{p-1}-\frac{N}{2s},
\]
and
\begin{equation*}
c_\lambda:=\mathcal I_\lambda(w_\lambda)
=\frac{p-1}{2(p+1)}\int_{\R^N}w_\lambda^{p+1}\,dy
=c_*\lambda^\theta.
\end{equation*}

We now specify the class of concentrating families in present paper.
\begin{definition}\label{def:kpeak-class}
Let $\xi_1^0,\ldots,\xi_k^0$ be different critical points of $V$.  A family $u_\varepsilon$ of positive solutions of \eqref{eq:fractional-schrodinger-problem} is called a positive $k$-peak concentrating family satisfying the energy quantization condition at $\xi^0=(\xi_1^0,\ldots,\xi_k^0)$ if there exists $r>0$ such that the balls $B_{2r}(\xi_i^0)$ are pairwise disjoint and, for all sufficiently small \(\varepsilon\), each $B_r(\xi_j^0)$ contains a unique local maximum point $\xi_{\varepsilon,j}$ of $u_\varepsilon$ with
\[
\xi_{\varepsilon,j}\to \xi_j^0,
\]
while $u_\varepsilon\to0$ uniformly on $\mathbb R^N\setminus\bigcup_{j=1}^k B_r(\xi_j^0)$.
Moreover, the peaks are nonvanishing,
\begin{equation}\label{eq:peak-nonvanishing}
\liminf_{\varepsilon\to0}u_\varepsilon(\xi_{\varepsilon,j})>0,
\qquad j=1,\ldots,k,
\end{equation}
and the family satisfies the $k$-peak energy quantization condition
\begin{equation}\label{eq:sharp-energy-quantization}
\varepsilon^{-N}\mathcal E_\varepsilon(u_\varepsilon)
\longrightarrow
\sum_{j=1}^k c_{V(\xi_j^0)}.
\end{equation}
\end{definition}
In the sequel we write
\[
\xi_\varepsilon:=(\xi_{\varepsilon,1},\ldots,\xi_{\varepsilon,k}),
\qquad
q_{\varepsilon,j}:=\frac{\xi_{\varepsilon,j}}{\varepsilon},
\qquad
q_\varepsilon:=(q_{\varepsilon,1},\ldots,q_{\varepsilon,k}).
\]

\subsection{Main results}
We now state the main results. The first theorem computes the Morse index for every positive \(k\)-peak family in the sense of Definition~\ref{def:kpeak-class}.

\begin{theorem}\label{th:morse-index}
Let $\xi_1^0,\dots,\xi_k^0$ be different nondegenerate critical points of $V$, namely
\begin{equation}\label{eq:nondegenerate-critical-points}
  \nabla V\left(\xi_i^0\right)=0,\quad D^2V\left(\xi_i^0\right)\text{~is~invertible~for~all~}i=1,\ldots,k.
\end{equation}
Let $u_\varepsilon$ be a positive $k$-peak family of solutions to \eqref{eq:fractional-schrodinger-problem} in the sense of Definition~\ref{def:kpeak-class}, concentrating at $\xi_1^0,\dots,\xi_k^0$. Then, for all sufficiently small $\varepsilon$, the Morse index of $u_\varepsilon$ is
\[
m(u_\varepsilon)
=k+\sum_{j=1}^k m(\xi_j^0,V)
=k+\#\bigl\{\Lambda _{j,i}<0:\ \Lambda_{j,i}\text{ is an eigenvalue of }D^2V(\xi_j^0)\bigr\}.
\]
\end{theorem}

The spectral estimates also imply the nondegeneracy of $u_\varepsilon$ for problem \eqref{eq:fractional-schrodinger-problem}.
\begin{corollary}\label{th:nondegeneracy}
Under the assumptions of Theorem~\ref{th:morse-index}, the solution $u_\varepsilon$ is nondegenerate. More precisely, for the linearized operator at $u_\varepsilon$,
\[
L_\varepsilon \phi
:=\varepsilon^{2s}(-\Delta)^s\phi + V(x)\phi - p\,u_\varepsilon^{p-1}\phi,
\qquad \phi\in H^s(\mathbb{R}^N),
\]
we have 
\[
\ker(L_\varepsilon)=\{0\}\quad\text{in }H^s(\mathbb{R}^N).
\]
\end{corollary}

To state the degree results, we introduce the approximate multi-peak profiles and the orthogonality conditions for the remainder.
For \(\xi=(\xi_1,\ldots,\xi_k)\) and \(\alpha=(\alpha_1,\ldots,\alpha_k)\), set
\[
W_{\varepsilon,j}(x;\xi_j):=w_{V(\xi_j)}\!\left(\frac{x-\xi_j}{\varepsilon}\right),
\qquad
W_{\varepsilon,\alpha,\xi}(x):=
\sum_{j=1}^k\alpha_j W_{\varepsilon,j}(x;\xi_j),
\]
and
\[
Y_{\varepsilon,j,h}(x;\xi_j):=\partial_{\xi_{j,h}}
w_{V(\xi_j)}\left(\frac{x-\xi_j}{\varepsilon}\right),
\qquad j=1,\ldots,k,\quad h=1,\ldots,N.
\]
Define the orthogonal complement:
\[
\begin{aligned}
E_{\varepsilon,\xi}:=
\{\omega\in H^s(\mathbb R^N):&
\ \langle\omega,W_{\varepsilon,j}(x;\xi_j)\rangle_{\varepsilon,V}=0,\\
&\ \langle\omega,Y_{\varepsilon,j,h}(x;\xi_j)\rangle_{\varepsilon,V}=0,
\quad j=1,\ldots,k,\ h=1,\ldots,N\}.
\end{aligned}
\]

We use the following weighted norm to control the solutions of \eqref{eq:fractional-schrodinger-problem}.
For \(0<\sigma <s\), set
\[
\beta_\sigma :=N+2s-\sigma.
\]
For \(q=(q_1,\ldots,q_k)\in(\mathbb R^N)^k\), define the weighted norm by
\begin{equation}\label{eq:rescaled-weight}
\rho_{\sigma ,q}(y)
:=\sum_{j=1}^k(1+|y-q_j|)^{-\beta_\sigma },
\qquad
\|f\|_{\sigma ,q}
:=\|\rho_{\sigma ,q}^{-1}f\|_{L^\infty(\mathbb R^N)},
\end{equation}
and
\begin{equation}\label{eq:physical-weight}
\rho_{\sigma ,\varepsilon,q}(x)
:=\rho_{\sigma ,q/\varepsilon}(x/\varepsilon)
=\sum_{j=1}^k
\left(1+\frac{|x-q_j|}{\varepsilon}\right)^{-\beta_\sigma },
\qquad
\|\varphi\|_{\sigma ,\varepsilon,q}
:=\|\rho_{\sigma ,\varepsilon,q}^{-1}\varphi\|_{L^\infty(\mathbb R^N)}.
\end{equation}
Choose \(0<\delta_0<\sigma <\tau<s \) and a small constant \(\tau_0>0\), and define
\begin{equation}\label{eq:parameter-domain}
\mathcal D_\varepsilon=\{(\alpha,\xi): |\alpha_j-1|<\varepsilon^{\delta_0},\; |\xi_j-\xi_j^0|<\tau_0,\ j=1,\ldots,k\}.
\end{equation}
Set
\[
\mathcal W_{\varepsilon,\xi}:=
E_{\varepsilon,\xi}\cap
\{\omega:\|\omega\|_{\tau,\varepsilon,\xi}<\infty\},
\]
and
\begin{equation}\label{eq:local-degree-set}
\begin{aligned}
S_\varepsilon:=\{(\alpha,\xi,\omega):{}&(\alpha,\xi)\in \mathcal D_\varepsilon,\quad
\omega\in \mathcal W_{\varepsilon,\xi},\\
&\|\omega\|_{\varepsilon,V}<\varepsilon^{N/2+\delta_0},\quad
\|\omega\|_{\tau,\varepsilon,\xi}<\varepsilon^{\delta_0}\}.
\end{aligned}
\end{equation}

\begin{theorem}\label{th:parametrization-class}
Let \(u_\varepsilon\) be any positive \(k\)-peak concentrating family in the sense of Definition~\ref{def:kpeak-class}.  Then, for all sufficiently small \(\varepsilon\), there exist
\((\alpha_\varepsilon,\widehat{\xi}_\varepsilon,\omega_\varepsilon)\in S_\varepsilon\) such that
\[
u_\varepsilon=W_{\varepsilon,\alpha_\varepsilon,\widehat{\xi}_\varepsilon}+\omega_\varepsilon.
\]
This representation is unique among triples in \(S_\varepsilon\).
\end{theorem}

The next theorem gives the Leray--Schauder degree formula corresponding to the Morse index formula.  Let
\begin{equation}\label{eq:reduced-potential}
\mathcal V(\xi):=\sum_{j=1}^k V(\xi_j)^\theta,
\qquad
\theta:=\frac{p+1}{p-1}-\frac{N}{2s}>0,
\end{equation}
where the positivity of \(\theta\) follows from the subcritical assumption.
\begin{theorem}\label{th:degree-counting}
Let \(k\in\mathbb N\) and let \(\xi^0:=(\xi_1^0,\dots,\xi_k^0)\in(\mathbb R^N)^k\), with \(\xi_i^0\ne\xi_j^0\) for \(i\ne j\), be a critical point of the reduced potential \(\mathcal V\).  If \(\xi^0\) is nondegenerate, then the total Leray--Schauder degree of all positive \(k\)-peak concentrating solutions of problem \eqref{eq:fractional-schrodinger-problem} satisfying Definition~\ref{def:kpeak-class} and concentrating at \(\xi^0\) is given by
\[
(-1)^{k+m(\xi^0,\mathcal V)}.
\]
Here \(m(\xi^0,\mathcal V)\) denotes the Morse index of the reduced potential
\(\mathcal V\) at \(\xi^0\); equivalently,
\(m(\xi^0,\mathcal V)=\sum_{j=1}^k m(\xi_j^0,V)\).
\end{theorem}

Combining the total degree with the common local-degree sign supplied by the Morse-index theorem gives exact uniqueness.
\begin{theorem}\label{th:local-uniqueness}
Let \(\xi_1^0,\ldots,\xi_k^0\) be different nondegenerate critical points of \(V\), that is \eqref{eq:nondegenerate-critical-points} holds.  Then the $k$-peak solution of \eqref{eq:fractional-schrodinger-problem} concentrating at \(\xi_1^0,\ldots,\xi_k^0\) is unique for sufficiently small \(\varepsilon\).
\end{theorem}

\subsection{Main difficulties and proof strategy}

The fractional problem cannot be handled by a direct repetition of the local arguments. The first difficulty comes from localization. A cut-off function does not commute with \((-\Delta)^s\), so localizing the equation or an eigenfunction produces commutator terms that couple a neighborhood of a peak with its complement. The second difficulty comes from the algebraic decay of the fractional ground state,
\[
w(x)\asymp |x|^{-N-2s}
\qquad\text{as }|x|\to\infty.
\]
Unlike the exponential decay in the local Schr\"odinger equation, the tails of the peaks and the interactions between different peaks decay only algebraically and must be estimated in weighted norms. The third difficulty is that a family satisfying Definition~\ref{def:kpeak-class} is given only through its concentration properties and energy quantization condition. The amplitudes, modulation centers, and orthogonal remainder needed in the reduction are not given in advance. Finally, the orthogonal space for the remainder depends on the center parameters, whereas the Leray--Schauder degree must be computed after the problem has been written in a fixed Banach space.

We overcome these difficulties in four steps. First, the concentration assumptions and energy quantization show that the limiting profile at each peak is the positive ground state of the corresponding limiting equation. They also yield global weighted estimates for the solution and for the relevant eigenfunctions.

Second, we use the Caffarelli--Silvestre extension and localized Pohozaev identities to determine the negative and near-zero spectrum. The first \(k\) eigenvalues remain uniformly negative. The next \(kN\) eigenvalues are of order \(\varepsilon^2\), and, after division by \(\varepsilon^2\), they converge to the eigenvalues of a block diagonal matrix whose \(j\)-th block is a positive multiple of \(D^2V(\xi_j^0)\). The remaining spectrum is uniformly separated from zero. Under the nondegeneracy assumption on the limiting critical points, this spectral description gives the Morse index formula and the nondegeneracy of every solution in Definition~\ref{def:kpeak-class}.

Third, for every solution in this class, we obtain a modulation decomposition
\[
u_\varepsilon
=
W_{\varepsilon,\alpha_\varepsilon,\widehat{\xi}_\varepsilon}
+\omega_\varepsilon,
\]
where \(\omega_\varepsilon\) satisfies the required orthogonality conditions and is small in both the energy norm and a weighted \(L^\infty\) norm. We then represent the orthogonal spaces, which depend on the center parameters, in a fixed Banach space. This allows us to define and compute the Leray--Schauder degree for the whole concentrating class.

Fourth, we compute the total degree by a homotopy to the leading equations for the amplitudes, the centers, and the remainder. Every solution in Definition~\ref{def:kpeak-class} has the same Morse index and hence the same sign of the local degree. Comparing the total degree with the sum of the local degrees gives uniqueness within the prescribed concentrating class.

The main point of the argument is that both the spectral analysis and the degree computation are carried out for the whole class of positive \(k\)-peak families satisfying Definition~\ref{def:kpeak-class}. The argument is therefore not restricted to one family constructed by Lyapunov--Schmidt reduction.

\subsection{Organization of the paper}

Section~\ref{sec:preliminaries} develops the weighted estimates, extension estimates, and localized Pohozaev identities used throughout the paper. Sections~\ref{sec:eig-estimates} and~\ref{sec:morse-computation} identify the spectrum and prove the Morse-index and nondegeneracy results. Section~\ref{sec:improved-expansion} establishes the modulation decomposition, while Section~\ref{sec:approx-manifold-estimates} proves coercivity on the orthogonal complement and derives the leading amplitude and center equations. Section~\ref{sec:degree-reduction} computes the total Leray--Schauder degree. Finally, Section~\ref{sec:local-uniqueness} compares the total degree with the local Morse signs and proves Theorem~\ref{th:local-uniqueness}.

\subsection{Notations}\label{subsec:notation}

\begin{itemize}[leftmargin=1.1cm,label=--]
\item The \(O\) and \(o\) notations are used to describe the
limit behavior of a quantity as \(\varepsilon\to0\).
\item \(C>0\) denotes a generic constant independent of \(\varepsilon\), whose
value may change from line to line, while constants with subscripts have fixed
positive values.
\item \(B_r(z)\) denotes the open ball in \(\R^N\) with center \(z\) and radius
\(r\), and \(\partial B_r(z)\) denotes its boundary.
\item \(\mathbb S^{N-1}\) denotes the unit sphere in \(\R^N\), and
\(|\mathbb S^{N-1}|\) denotes its surface measure.
\item We write \(\R^{N+1}_+:=\R^N\times(0,\infty)\).  For \(z\in\R^N\),
\[
\mathcal B_r(z):=\{X\in\R^{N+1}:|X-(z,0)|<r\},\qquad
\mathcal B_r^+(z):=\mathcal B_r(z)\cap\R^{N+1}_+ .
\]
We also write
\[
\partial'\mathcal B_r^+(z):=\partial\mathcal B_r^+(z)\cap\R^{N+1}_+ .
\]
\item The variables in \(\R^{N+1}_+\) are denoted by \(X=(x,t)\), where
\(x\in\R^N\) and \(t>0\). We use \(dX=dx\,dt\).
\item The measure \(dS\) denotes the surface measure on boundaries in \(\R^N\),
while \(d\sigma\) denotes the surface measure on boundaries in the extension
space.
\item For a real number \(a\), \(a_+:=\max\{a,0\}\) and
\(a_-:=\max\{-a,0\}\).
\item The notation \(\partial_h\) means differentiation with respect to the
\(h\)-th spatial variable, and \(D^2V\) denotes the Hessian matrix of \(V\).
\end{itemize}

\section{Preliminaries}\label{sec:preliminaries}

\subsection{Estimates}

We first give some consequences of the profile of concentrating solutions.

\begin{lemma}\label{lem:single-bubble-input}
For $\lambda\in[V_{\min},V_{\max}]$ set
\[
L_\lambda\phi:=(-\Delta)^s\phi+\lambda\phi-pw_\lambda^{p-1}\phi.
\]
Then the following facts hold.
\begin{enumerate}[label=(\roman*)]
\item For every $\lambda\in[V_{\min},V_{\max}]$, the operator $L_\lambda$ has exactly one negative eigenvalue. This eigenvalue is simple and its eigenfunction may be chosen positive and radial. Moreover,
\[
\ker L_\lambda=\operatorname{span}\{\partial_1w_\lambda,\ldots,\partial_Nw_\lambda\}.
\]
\item The map $\lambda\mapsto w_\lambda$ is $C^1$ from $[V_{\min},V_{\max}]$ into $H^s(\R^N)$, and
\begin{equation}\label{eq:uniform-decay}
|w_\lambda(y)|+|\nabla w_\lambda(y)|+|\partial_\lambda w_\lambda(y)|
\le C(1+|y|)^{-N-2s}.
\end{equation}
In addition, the second derivatives satisfy the uniform decay estimate
\[
|D^2w_\lambda(y)|\le C(1+|y|)^{-N-2s},
\qquad \lambda\in[V_{\min},V_{\max}],\quad y\in\R^N.
\]
Moreover, $w_\lambda$ is strictly radially decreasing and has a nondegenerate maximum at the origin:
\begin{equation}\label{eq:hessian-nondeg-max}
D^2w_\lambda(0)=-\kappa_\lambda I_N,
\qquad \kappa_\lambda>0,
\end{equation}
uniformly for $\lambda\in[V_{\min},V_{\max}]$.
\end{enumerate}
\end{lemma}
\begin{lemma}\label{lem:true-peak-profile-convergence}
Let $u_\varepsilon$ be a $k$-peak family in the sense of Definition~\ref{def:kpeak-class}. Define
\begin{equation*}
u_{\varepsilon,j}(y):=u_\varepsilon(\xi_{\varepsilon,j}+\varepsilon y).
\end{equation*}
Then
\begin{equation}\label{eq:true-peak-profile-convergence}
u_{\varepsilon,j}\to w_{V(\xi_j^0)}
\quad\hbox{in }C^1_{\loc}(\R^N),
\qquad j=1,\ldots,k.
\end{equation}
\end{lemma}
\noindent The proofs of Lemmas ~\ref{lem:single-bubble-input} and \ref{lem:true-peak-profile-convergence} are given in Appendix~\ref{app:proof-profile-lemmas}.

We now establish the global weighted estimates for concentrating solutions and for the relevant linearized eigenfunctions.
\begin{lemma}\label{lem:u-polynomial-decay}
Let $u_\varepsilon$ be a $k$-peak family in the sense of Definition~\ref{def:kpeak-class} and set
\[
\tilde u_\varepsilon(y):=u_\varepsilon(\varepsilon y).
\]
Then the following estimates hold.
\begin{enumerate}[label=(\roman*)]
\item For all sufficiently small $\varepsilon$,
\begin{equation}\label{eq:u-weighted-bound}
\|\tilde u_\varepsilon\|_{\sigma,q_\varepsilon}\le C.
\end{equation}

Equivalently,
\begin{equation}\label{eq:u-polynomial-decay}
u_\varepsilon(x)\le C\sum_{j=1}^k
\left(\frac{\varepsilon}{\varepsilon+|x-\xi_{\varepsilon,j}|}\right)^{\beta_\sigma}
\qquad\text{for all }x\in\mathbb R^N.
\end{equation}

\item Let $(\mu_\varepsilon^{(l)},v_\varepsilon^{(l)})$ solve the linearized eigenvalue problem \eqref{eq:linearized-eigenvalue-problem}, for $-C\le\mu_\varepsilon^{(l)}\le V_{\min}/4$.  Set
\[
\tilde v_\varepsilon^{(l)}(y):=v_\varepsilon^{(l)}(\varepsilon y).
\]
Then
\begin{equation}\label{eq:rescaled-eigenfunction-decay}
\|\tilde v_\varepsilon^{(l)}\|_{\sigma,q_\varepsilon}\le C.
\end{equation}

Equivalently,
\begin{equation}\label{eq:eigenfunction-decay}
|v_\varepsilon^{(l)}(x)|\le C\sum_{j=1}^k
\left(\frac{\varepsilon}{\varepsilon+|x-\xi_{\varepsilon,j}|}\right)^{\beta_\sigma}.
\end{equation}

\end{enumerate}

Moreover, for some \(\alpha\in(0,1)\),
\[
\tilde u_\varepsilon(y),\quad
\tilde v_\varepsilon^{(l)}(y) 
\]
are uniformly bounded in $C^{1,\alpha}(\mathbb R^N)$.

\end{lemma}

\begin{proof}

We prove (i).  The rescaled function $\tilde u_\varepsilon$ satisfies
\begin{equation}\label{eq:scaled-u-equation}
(-\Delta)^s\tilde u_\varepsilon+V(\varepsilon y)\tilde u_\varepsilon
=\tilde u_\varepsilon^p\qquad\text{in }\mathbb R^N.
\end{equation}
Testing the equation with \(\tilde u_\varepsilon\) and using \eqref{eq:sharp-energy-quantization} give a uniform \(H^s\) bound for \(\tilde u_\varepsilon\).

Since \(1<p<2_s^*-1\), the function \(\tilde u_\varepsilon^{p-1}\) belongs to \(L^{q}(\mathbb R^N)\), where \(q=\frac{2_s^*}{p-1}>\frac{N}{2s}\). Hence, Brezis--Kato estimate in \cite[Proposition 4.5]{DuarteSouto2019} and Sobolev embedding give
\[
\|\tilde u_\varepsilon\|_{L^\infty(\mathbb R^N)}\le C.
\]
Then, the regularity estimates in \cite{cabre2014nonlinear,silvestre2007regularity} give uniform \(C^{1,\alpha}\) bounds for \(\tilde u_\varepsilon\) on \(\mathbb R^N\).

We claim that 
\begin{equation}\label{eq:no-additional-concentration}
\lim_{R\to\infty}\limsup_{\varepsilon\to0}
\sup_{x\notin\cup_jB_{R\varepsilon}(\xi_{\varepsilon,j})}u_\varepsilon(x)=0.
\end{equation}
Indeed, if \eqref{eq:no-additional-concentration} failed, then for some $\delta>0$ there would be $\varepsilon_n\to0$, $R_n\to\infty$, and $x_n$ such that
\[
\operatorname{dist}(x_n,\{\xi_{\varepsilon_n,j}\}_{j=1}^k)\ge R_n\varepsilon_n,
\qquad u_{\varepsilon_n}(x_n)\ge\delta.
\]
The uniform H\"older estimate then gives $r_\delta,c_\delta>0$ with
\[
\varepsilon_n^{-N}\int_{B_{r_\delta\varepsilon_n}(x_n)}u_{\varepsilon_n}^{p+1}\,dx\ge c_\delta.
\]
For every fixed $R$, this ball is disjoint from all $B_{R\varepsilon_n}(\xi_{\varepsilon_n,j})$ for large $n$. Lemma~\ref{lem:true-peak-profile-convergence} gives the limiting mass in those peak balls. Choosing \(R\) sufficiently large so that $\sum\limits_{j=1}^k\int_{\mathbb R^N\setminus B_R(\xi_j^0).}w_{V(\xi_j^0)}^{p+1}\,dy<c_\delta/2$, we obtain
\[
\liminf_{n\to\infty}\varepsilon_n^{-N}\int_{\mathbb R^N}u_{\varepsilon_n}^{p+1}\,dx
\ge\sum_{j=1}^k\int_{\mathbb R^N}w_{V(\xi_j^0)}^{p+1}\,dy+\frac{c_\delta}{2},
\]
contrary to \eqref{eq:sharp-energy-quantization}.

We next prove the weighted estimate.  By \eqref{eq:no-additional-concentration}, we can choose $R>1$ sufficiently large and set
\[
B:=\bigcup_{j=1}^k B_R(q_{\varepsilon,j}).
\]
Then
\[
p\tilde u_\varepsilon(y)^{p-1}\le \frac{V_{\min}}{2}
\qquad\text{for }y\in\mathbb R^N\setminus B
\]
for all sufficiently small $\varepsilon$. Accordingly, \eqref{eq:scaled-u-equation} can be written as
\begin{equation}\label{eq:u-weighted-equation}
(-\Delta)^s\tilde u_\varepsilon+W_\varepsilon(y)\tilde u_\varepsilon=0,
\qquad
W_\varepsilon(y):=V(\varepsilon y)-\tilde u_\varepsilon(y)^{p-1},
\end{equation}
with $W_\varepsilon\ge V_{\min}/2$ in $\mathbb R^N\setminus B$. Applying \cite[Lemma~2.5]{MR3121716} to \eqref{eq:u-weighted-equation} gives
\[
\|\rho_{\sigma,q_\varepsilon}^{-1}\tilde u_\varepsilon\|_{L^\infty(\mathbb R^N)}
\le C\|\tilde u_\varepsilon\|_{L^\infty(B)}
\le C.
\]
This proves \eqref{eq:u-weighted-bound} and \eqref{eq:u-polynomial-decay}.

For (ii), the normalization in \eqref{eq:linearized-eigenvalue-problem} and the fractional Brezis--Kato estimate give a uniform $L^\infty$ bound.  Outside the same union of rescaled balls,
\[
V(\varepsilon y)-p\tilde u_\varepsilon^{p-1}-\mu_\varepsilon^{(l)}\ge V_{\min}/4.
\]
The fractional Kato inequality and the preceding barrier argument yield \eqref{eq:rescaled-eigenfunction-decay} and \eqref{eq:eigenfunction-decay}.  

By \eqref{eq:u-weighted-bound} and \eqref{eq:rescaled-eigenfunction-decay}, the stated uniform $C^{1,\alpha}$ bounds then follow from the regularity estimates in \cite{cabre2014nonlinear,silvestre2007regularity}.

\end{proof}

\subsection{Extension}
In this subsection, we recall the local extension introduced by Caffarelli and Silvestre \cite{Caffarelli08082007}. Denote $X=(x,t)\in\mathbb{R}^{N+1}$. More precisely, for any $u\in H^s(\mathbb{R}^N)$, set 
\[\bar{u}(x,t):=E(u)=\int_{\mathbb{R}^{N}}P_{s}(x-z,t)u(z)dz,\quad(x,t)\in\mathbb{R}_{+}^{N+1},\]
where
\[P_{s}(x,t)=\beta(N,s)\frac{t^{2s}}{(|x|^2+t^2)^{\frac{N+2s}{2}}}\]
with a constant $\beta(N,s)$ such that $\int_{\mathbb{R}^N}P_{s}(x,t)\,dx=1$. Then $\bar{u}\in L^2(t^{1-2s},K)$ for any compact set $K$ in $\overline{\mathbb{R}_{+}^{N+1}}$, $\nabla\bar{u}\in L^{2}(t^{1-2s},\mathbb{R}_{+}^{N+1})$ and $\bar{u}\in C^{\infty}(\mathbb{R}_{+}^{N+1})$. Moreover, by the work of Caffarelli and Silvestre \cite{Caffarelli08082007}, $\bar{u}$ satisfies
\[
\begin{cases}
\operatorname{div}(t^{1-2s}\nabla\bar{u})=0,\quad  &\text{in }\mathbb{R}_+^{N+1},\\
\bar{u}(x,0)=u(x),\quad &\text{in }\mathbb{R}^N,\\
\lim\limits_{t\to0}t^{1-2s}\frac{\partial \bar{u}}{\partial \nu}(x,t)=\omega_s(-\Delta)^su(x), &\text{in }\mathbb{R}^N
\end{cases}
\]
in the sense of distributions, where $\omega_{s}=2^{1-2s}\frac{\Gamma(1-s)}{\Gamma(s)}$. Moreover, the following identity holds
\[\int_{\mathbb{R}^{N+1}_+}t^{1-2s}\,|\nabla \bar u|^2\,dx\,dt=\omega_s\,\|u\|_{\dot H^{s}(\mathbb{R}^N)}^{2}.\]

For notational simplicity, from now on we divide the weighted energy and the conormal derivative by the fixed positive constant $\omega_s$ and omit this normalization factor in the notation.  With this convention, the nonlocal problem \eqref{eq:fractional-schrodinger-problem} can be reformulated as the following local problem:

\begin{equation}\label{eq:extension-u}
 \left\{
 \begin{aligned}
 \operatorname{div}\!\left(t^{1-2s}\nabla \bar{u}_\varepsilon\right) &= 0
 &&\text{in }\mathbb{R}^{N+1}_+,\\
 \lim_{t\to 0^+}t^{1-2s}\frac{\partial \bar{u}_\varepsilon}{\partial \nu}
 &=\varepsilon^{-2s}\bigl(u_\varepsilon^{\,p}- V(x)\,u_\varepsilon\bigr)
 &&\text{on }\mathbb{R}^N.
 \end{aligned}
 \right.\nonumber
\end{equation}

We next record a useful property of extension functions.
\begin{lemma}\label{lem:minimality-extension}
Let $0<s<1$ and define the energy space
\[
\mathcal D_s^{1,2}(\mathbb R^{N+1}_+)
:=
\left\{
\Psi\in H^1_{\rm loc}(\mathbb R^{N+1}_+,t^{1-2s}):
\int_{\mathbb R^{N+1}_+}t^{1-2s}|\nabla\Psi|^2<\infty
\right\}.
\]
For $\phi\in H^s(\mathbb{R}^N)$, let
$\Phi=E(\phi)\in \mathcal D_s^{1,2}$ denote its $s$-harmonic extension, i.e.,
\begin{equation}\label{eq:extension-problem}
\begin{cases}
\! \operatorname{div}\bigl(t^{1-2s}\nabla \Phi\bigr)=0 &\text{in }\mathbb{R}^{N+1}_+,\\
\Phi(x,0)=\phi(x) &\text{on }\mathbb{R}^N,
\end{cases}
\end{equation}
in the weak sense. Then $\Phi$ minimizes the weighted Dirichlet energy among all
extensions of $\phi$:
\begin{equation}\label{eq:minimality-property}
\int_{\mathbb{R}^{N+1}_+} t^{1-2s}|\nabla \Phi|^2\,dx\,dt
=\min\Bigl\{\int_{\mathbb{R}^{N+1}_+} t^{1-2s}|\nabla \Psi|^2\,dx\,dt \;:\;
\Psi\in \mathcal D_s^{1,2},\ \Psi(x,0)=\phi(x)\Bigr\}.
\end{equation}
Moreover, for every $\Psi\in \mathcal D_s^{1,2}$ with $\Psi(x,0)=\phi(x)$ one has the orthogonal
 decomposition
\begin{equation}\label{eq:pythagoras}
\int_{\mathbb{R}^{N+1}_+} t^{1-2s}|\nabla \Psi|^2\,dx\,dt
=
\int_{\mathbb{R}^{N+1}_+} t^{1-2s}|\nabla \Phi|^2\,dx\,dt
+
\int_{\mathbb{R}^{N+1}_+} t^{1-2s}|\nabla(\Psi-\Phi)|^2\,dx\,dt.
\end{equation}
In particular,
\[
\int_{\mathbb{R}^{N+1}_+} t^{1-2s}|\nabla \Psi|^2\,dx\,dt
\ge
\int_{\mathbb{R}^{N+1}_+} t^{1-2s}|\nabla \Phi|^2\,dx\,dt.
\]
\end{lemma}

\begin{proof}
Fix $\phi\in H^s(\mathbb{R}^N)$ and let $\Phi=E(\phi)\in \mathcal D_s^{1,2}$.
Take any $\Psi\in \mathcal D_s^{1,2}$ such that $\Psi(x,0)=\phi(x)$ and set $W:=\Psi-\Phi$.
Then $W\in \mathcal D_s^{1,2}$ and $W(x,0)=0$. Expanding the weighted Dirichlet energy gives
\[
\int_{\mathbb{R}^{N+1}_+} t^{1-2s}|\nabla \Psi|^2
=
\int_{\mathbb{R}^{N+1}_+} t^{1-2s}|\nabla \Phi|^2
+
2\int_{\mathbb{R}^{N+1}_+} t^{1-2s}\nabla \Phi\cdot\nabla W
+
\int_{\mathbb{R}^{N+1}_+} t^{1-2s}|\nabla W|^2.
\]
Since $\Phi$ solves
\eqref{eq:extension-problem} in the weak sense, it satisfies
\[
\int_{\mathbb{R}^{N+1}_+} t^{1-2s}\nabla \Phi\cdot\nabla \eta\,dx\,dt=0
\qquad\text{for all }\eta\in\mathcal D_s^{1,2}(\mathbb R^{N+1}_+)
\text{ with }\eta(x,0)=0.
\]
Applying this with $\eta=W$ yields
\(
\int_{\mathbb{R}^{N+1}_+} t^{1-2s}\nabla \Phi\cdot\nabla W=0,
\)
and hence \eqref{eq:pythagoras}. The minimization property \eqref{eq:minimality-property}
follows immediately. The characterization of equality is also immediate from
\eqref{eq:pythagoras}.
\end{proof}

\subsection{Estimates on the extension}\label{subsec:extension-boundary-estimates}
With the notation of Section~\ref{subsec:notation}, set
\[
A_{a,b}^{+,j}:=\mathcal B_b^+(\xi_{\varepsilon,j})\setminus \overline{\mathcal B_a^+(\xi_{\varepsilon,j})},
\qquad
A_{a,b}^{j}:=B_b(\xi_{\varepsilon,j})\setminus \overline{B_a(\xi_{\varepsilon,j})}.
\]

\begin{lemma}\label{lem:extension-u-bound}
Let $u_\varepsilon$ be a $k$-peak family in the sense of Definition~\ref{def:kpeak-class}, and let
\[
\bar u_\varepsilon=E(u_\varepsilon).
\]
Then the following estimates hold.
\begin{enumerate}[label=(\roman*)]
\item For every fixed sufficiently small \(r>0\) and every \(j=1,\ldots,k\), on
\(A_{r/2,3r}^{+,j}\) one has
\begin{equation}\label{eq:extension-u-annulus-pointwise}
|\bar u_{\varepsilon}(x,t)|
\le
C\varepsilon^N\sum_{i=1}^k\frac{1}{(1+|x-\xi_{\varepsilon,i}|)^{\beta_\sigma}}.
\end{equation}
Moreover,
\begin{equation}\label{eq:annular-u-energy}
\int_{A_{r,2r}^{+,j}}t^{1-2s}|\nabla \bar u_\varepsilon|^2dX
\le C\varepsilon^{2N}.
\end{equation}

\item For every \(a=1,\ldots,N\), there exists \(C>0\), independent of
\(\varepsilon\), such that
\begin{equation}\label{eq:rescaled-u-derivative-decay}
  \left\|\rho_{\sigma,q_\varepsilon}^{-1}\partial_{y_a}\widetilde u_\varepsilon
\right\|_{L^\infty(\mathbb R^N)}
\le C.
\end{equation}
Moreover, set
\[
\overline{\partial_{x_a}u_\varepsilon}
:=
E(\partial_{x_a}u_\varepsilon).
\]
Then for every fixed small \(r>0\) and every
\(j=1,\ldots,k\),
\begin{equation}\label{eq:u-extension-derivative-decay}
  \left|
\overline{\partial_{x_a}u_\varepsilon}(x,t)
\right|
\le
C\varepsilon^{N-1}
\sum_{i=1}^k
\frac{1}{(1+|x-\xi_{\varepsilon,i}|)^{\beta_\sigma}}
\end{equation}
for all
\[
(x,t)\in A_{r/2,3r}^{+,j}.
\]

\item Let $v_\varepsilon^{(l)}$ be an eigenfunction satisfying the assumptions of Lemma~\ref{lem:u-polynomial-decay}(ii), and let
\[
\bar v_\varepsilon^{(l)}=E(v_\varepsilon^{(l)}).
\]
For every fixed $r>0$ and every $j=1,\ldots,k$, on $A_{r/2,3r}^{+,j}$ one has
\begin{equation}\label{eq:extension-v-annulus-pointwise}
|\bar v_{\varepsilon}^{(l)}(x,t)|
\le
C\varepsilon^N\sum_{i=1}^k\frac{1}{(1+|x-\xi_{\varepsilon,i}|)^{\beta_\sigma}}.
\end{equation}
Moreover,
\begin{equation}\label{eq:annular-v-energy}
\int_{A_{r,2r}^{+,j}}t^{1-2s}|\nabla \bar v_\varepsilon^{(l)}|^2dX
\le C\varepsilon^{2N},
\end{equation}
\end{enumerate}
\end{lemma}

\begin{proof}
We first prove (i). By Lemma~\ref{lem:u-polynomial-decay},
\[
|u_\varepsilon(z)|\le C\sum_{i=1}^k
\left(1+\frac{|z-\xi_{\varepsilon,i}|}{\varepsilon}\right)^{-\beta_\sigma},
\qquad \beta_\sigma>N.
\]
Using the Poisson representation of the extension and the bound
\[
P_s(x-z,t)\le C\frac{t^{2s}}{(|x-z|+t)^{N+2s}},
\]
after the change of variables $\widetilde z=x-z$ and an application of
Lemma~\ref{lem:weighted-convolution-estimate}, we obtain
\[
\begin{aligned}
|\bar u_\varepsilon(x,t)|
&\le
C\sum_{i=1}^k\int_{\mathbb R^N}
\frac{t^{2s}}{(|x-z|+t)^{N+2s}}
\left(1+\frac{|z-\xi_{\varepsilon,i}|}{\varepsilon}\right)^{-\beta_\sigma}\,dz \\
&=
C\sum_{i=1}^k t^{2s}\int_{\mathbb R^N}
\frac{1}{(t+|\widetilde z|)^{N+2s}}
\left(1+\frac{|x-\xi_{\varepsilon,i}-\widetilde z|}{\varepsilon}\right)^{-\beta_\sigma}
\,d\widetilde z \\
&\le
C\varepsilon^N\sum_{i=1}^k
\frac{1}{(1+|x-\xi_{\varepsilon,i}|)^{\beta_\sigma}}.
\end{aligned}
\]
This proves \eqref{eq:extension-u-annulus-pointwise}.  In particular, \eqref{eq:extension-u-annulus-pointwise} implies
\begin{equation}\label{eq:ubar-annulus-pointwise}
|\bar u_\varepsilon(x,t)|\le C\varepsilon^N,
\qquad (x,t)\in A_{r/2,3r}^{+,j}.
\end{equation}

Choose $\eta\in C_c^\infty(\overline{\mathbb R^{N+1}_+})$ such that
\[
0\le\eta\le1,
\qquad
\eta\equiv1\quad\text{on }A_{r,2r}^{+,j},
\qquad
\operatorname{supp}\eta\subset \overline{\mathcal B_{3r}^+(\xi_{\varepsilon,j})}\setminus \mathcal B_{r/2}^+(\xi_{\varepsilon,j}),
\qquad
|\nabla\eta|\le C.
\]
Testing \eqref{eq:extension-u} with $\eta^2\bar u_\varepsilon$ and applying Young's inequality, we obtain
\begin{align*}
\int_{\mathbb R^{N+1}_+}t^{1-2s}\eta^2|\nabla\bar u_\varepsilon|^2dX
&\le
C\int_{\mathbb R^{N+1}_+}t^{1-2s}|\nabla\eta|^2\bar u_\varepsilon^2dX \\
&\quad +C\varepsilon^{-2s}\int_{A_{r/2,3r}^{j}}
\bigl(u_\varepsilon^{p+1}+u_\varepsilon^2\bigr)dx.
\end{align*}
By \eqref{eq:ubar-annulus-pointwise},
\[
\int_{\mathbb R^{N+1}_+}t^{1-2s}|\nabla\eta|^2\bar u_\varepsilon^2dX
\le C\varepsilon^{2N}\int_{A_{r/2,3r}^{+,j}}t^{1-2s}dX
\le C\varepsilon^{2N}.
\]
On $A_{r/2,3r}^{j}$, by Lemma~\ref{lem:u-polynomial-decay}, we have
\[
\varepsilon^{-2s}\int_{A_{r/2,3r}^{j}}
\bigl(u_\varepsilon^{p+1}+u_\varepsilon^2\bigr)dx
\le C\varepsilon^{-2s}\bigl(\varepsilon^{(p+1)\beta_\sigma}+\varepsilon^{2\beta_\sigma}\bigr)
\le C\varepsilon^{2\beta_\sigma-2s}.
\]
Thus, \eqref{eq:annular-u-energy} follows.

The proofs of (ii) and (iii) follow from arguments similar to those used in Lemma~\ref{lem:u-polynomial-decay} and part (i), so we omit the details.
\end{proof}

We next estimate the peak locations \(\xi_{\varepsilon,j}\).
\begin{lemma}\label{lem:pohozaev-center-estimate}
  Let $u_\varepsilon$ be a $k$-peak family in the sense of Definition~\ref{def:kpeak-class}. The local maximum points satisfy
\begin{equation}\label{eq:pohozaev-center-estimate}
|\nabla V(\xi_{\varepsilon,j})|\le C\varepsilon,
\qquad j=1,\ldots,k.
\end{equation}
\end{lemma}

\begin{proof}
Fix a small number $r_0>0$ such that the balls $B_{4r_0}(\xi_j^0)$ are pairwise disjoint. By \eqref{eq:annular-u-energy}, for every $j$ there exists
\[
\rho_{\varepsilon,j}\in(r_0,2r_0)
\]
such that
\begin{equation}\label{eq:good-radius-u-energy}
  \int_{\partial'\mathcal B_{\rho_{\varepsilon,j}}^+(\xi_{\varepsilon,j})}
 t^{1-2s}|\nabla \bar u_\varepsilon|^2\,d\sigma
\le C\varepsilon^{2N}.
\end{equation}
For fixed $i\in\{1,\ldots,N\}$, the following local Pohozaev identity holds
\begin{equation}\label{eq:local-Pohozaev-identity-v}
  \begin{aligned}
\int_{B_{\rho_{\varepsilon,j}}(\xi_{\varepsilon,j})}
\partial_iV(x)u_\varepsilon^2\,dx
&=
\varepsilon^{2s}
\int_{\partial'\mathcal B_{\rho_{\varepsilon,j}}^+(\xi_{\varepsilon,j})}
t^{1-2s}|\nabla\bar u_\varepsilon|^2\nu_i\,d\sigma \\
&\quad
-2\varepsilon^{2s}
\int_{\partial'\mathcal B_{\rho_{\varepsilon,j}}^+(\xi_{\varepsilon,j})}
t^{1-2s}
\frac{\partial\bar u_\varepsilon}{\partial\nu}
\frac{\partial\bar u_\varepsilon}{\partial x_i}\,d\sigma \\
&\quad
+\int_{\partial B_{\rho_{\varepsilon,j}}(\xi_{\varepsilon,j})}
\left(V(x)u_\varepsilon^2-\frac{2}{p+1}u_\varepsilon^{p+1}\right)\nu_i\,dS .
\end{aligned}
\end{equation}
By \eqref{eq:good-radius-u-energy} and Lemma~\ref{lem:u-polynomial-decay}, we have
\begin{equation}\label{eq:RHS-local}
  \text{RHS of \eqref{eq:local-Pohozaev-identity-v}}=O(\varepsilon^{N+1}).
\end{equation}
Moreover, by Lemma~\ref{lem:true-peak-profile-convergence} and Lemma~\ref{lem:u-polynomial-decay}, we have
\begin{equation}\label{eq:pohozaev-center-estimate-1}
\begin{aligned}
&\int_{B_{\rho_{\varepsilon,j}}(\xi_{\varepsilon,j})}
\partial_iV(x)u_\varepsilon^2\,dx        \\
&\quad=
\varepsilon^N\partial_iV(\xi_{\varepsilon,j})
\int_{B_{{\rho_{\varepsilon,j}}/ {\varepsilon}}(0)}u_{\varepsilon,j}^2(y)\,dy
+O(\varepsilon^{N+1})\\
&\quad=\varepsilon^N\partial_iV(\xi_{\varepsilon,j})\left(\int_{\R^N}w_{V(\xi_j^0)}^2(y)+o(1)\right)+O(\varepsilon^{N+1}).
\end{aligned}
\end{equation}
Combining \eqref{eq:RHS-local} and \eqref{eq:pohozaev-center-estimate-1}, we obtain
\[
\varepsilon^N\partial_iV(\xi_{\varepsilon,j})\left(\int_{\R^N}w_{V(\xi_j^0)}^2(y)+o(1)\right)=O(\varepsilon^{N+1}).
\]
Hence
\[
|\partial_iV(\xi_{\varepsilon,j})|\le C\varepsilon,
\qquad i=1,\ldots,N.
\]
Summing over $i$ gives \eqref{eq:pohozaev-center-estimate}.
\end{proof}

\section{Spectral analysis}\label{sec:eig-estimates}
To compute the Morse index of the solution $u_\varepsilon$ of problem \eqref{eq:fractional-schrodinger-problem}, we first establish several key estimates for the eigenvalues and eigenfunctions of the linearized problem \eqref{eq:linearized-eigenvalue-problem}. For the concentration points $\xi_{1}^{0},\ldots,\xi_{k}^{0}$, we choose a fixed small constant $r>0$ such that 
\[B_{4r}(\xi_i^0)\cap B_{4r}(\xi_j^0)=\emptyset,\quad\text{for any }1\le i\ne j\le k.\] 

Let $\phi \in C_c^\infty(\mathbb{R}^N)$ be a cut-off function such that
\[
0\le \phi \le 1,\qquad
\phi \equiv 1 \ \text{in }B_r(0),\qquad
\operatorname{supp}\,\phi \subset B_{2r}(0),
\]
and set $\phi_l=\phi(x-\xi_{\varepsilon,l})$, so that $\operatorname{supp}\,\phi_l \cap \operatorname{supp}\,\phi_m=\varnothing$ for $l \neq m$.
Define
\[
U_{\varepsilon,l}:=\phi_l u_\varepsilon,
\qquad
\psi_{\varepsilon,l,j}:=\phi_l\,\frac{\partial u_\varepsilon}{\partial x_j},
\qquad j=1,\dots,N.
\]

In the extension space, let us also define a cut-off function $\Phi \in C_c^\infty(\overline{\mathbb{R}^{N+1}_+})$ such that
\[0\le \Phi\le 1,
\quad\Phi \equiv 1 \ \text{in } \mathcal B_{r}^+(0),\quad
\operatorname{supp}\,\Phi \subset \mathcal B_{2r}^+(0),
\]
and set $\Phi_l(x,t)=\Phi(x-\xi_{\varepsilon,l},t)$, so that $\Phi_l(x,0)=\phi_l(x)$ on $\mathbb{R}^N$.

Then we have the following linear independence result.
\begin{lemma}\label{lem:fractional-linear-independence}
Let $u_\varepsilon$ be a $k$-peak family in the sense of Definition~\ref{def:kpeak-class}. Then, for sufficiently small $\varepsilon$, the following set of functions 
\[
\Big\{U_{\varepsilon,l},\ \psi_{\varepsilon,l,1},\dots,\psi_{\varepsilon,l,N}\ :\  l =1,\dots,k\Big\},
\]
is linearly independent.
\end{lemma}

\begin{proof}
Fix $l\in\{1,\dots,k\}$. Assume that there exist $a\in\mathbb{R}$ and $b=(b_1,\dots,b_N)\in\mathbb{R}^N$ such that
\begin{equation}\label{eq:local-relation}
a\,U_{\varepsilon,l} + \sum_{j=1}^N b_j\,\psi_{\varepsilon,l,j} \equiv 0
\qquad\text{in }\mathbb{R}^N.
\end{equation}
Since $\phi_l\equiv 1$ in $B_r(\xi_{\varepsilon,l})$, the identity \eqref{eq:local-relation} implies
\begin{equation*}
a\,u_\varepsilon(x) + b\cdot \nabla u_\varepsilon(x) = 0
\qquad \text{for all }x\in B_r(\xi_{\varepsilon,l}).
\end{equation*}
Evaluating this identity at the peak point $\xi_{\varepsilon,l}$ gives $a=0$, because
\[
\nabla u_\varepsilon(\xi_{\varepsilon,l})=0,
\qquad u_\varepsilon(\xi_{\varepsilon,l})>0.
\]
Thus
\begin{equation}\label{eq:local-gradient-relation}
b\cdot \nabla u_\varepsilon(x)=0
\qquad\text{for all }x\in B_r(\xi_{\varepsilon,l}).
\end{equation}
If \(b\ne0\), let \(e=b/|b|\). Equation~\eqref{eq:local-gradient-relation} shows that
\(t\mapsto u_\varepsilon(\xi_{\varepsilon,l}+te)\) is constant for all sufficiently small \(|t|\). Since \(\xi_{\varepsilon,l}\) is a local maximum, every point of this short segment is then a local maximum, contradicting the uniqueness of the local maximum in \(B_r(\xi_l^0)\). Hence \(b=0\). This proves that \(\{U_{\varepsilon,l},\psi_{\varepsilon,l,1},\dots,\psi_{\varepsilon,l,N}\}\)
is linearly independent.

Assume that for some coefficients $\alpha_ l \in\mathbb{R}$ and $\beta_{l,j}\in\mathbb{R}$ we have
\begin{equation}\label{eq:global-relation}
\sum_{l=1}^k \alpha_ l \,U_{\varepsilon,l}
+\sum_{l=1}^k\sum_{j=1}^N \beta_{l,j}\,\psi_{\varepsilon,l,j}\equiv 0
\qquad\text{in }\mathbb{R}^N.
\end{equation}
Fix $ l _0\in\{1,\dots,k\}$. Restricting \eqref{eq:global-relation} to $B_r(\xi_{\varepsilon,l_0})$,
we use $\phi_{l_0}\equiv 1$ there and $\phi_ l \equiv 0$ for $ l \neq  l _0$ to obtain
\[
\alpha_{l_0}u_\varepsilon + \sum_{j=1}^N \beta_{l_0,j}\,\frac{\partial u_\varepsilon}{\partial x_j} \equiv 0
\qquad \text{in }B_r(\xi_{\varepsilon,l_0}).
\]

The preceding local argument gives $\alpha_{l_0}=0$ and $\beta_{l_0,j}=0$ for every $j$.

Since $ l _0$ is arbitrary, all coefficients in \eqref{eq:global-relation} vanish, and the full family is linearly independent.
\end{proof}

Arguing as in \cite[Lemma 3.2]{LuoPanPeng2024}, we obtain the following result.
\begin{lemma}\label{lem:local-identities}
  Let $u_\varepsilon$ be a $k$-peak family in the sense of Definition~\ref{def:kpeak-class}. For $(\alpha_1,\ldots,\alpha_N)\in\mathbb{R}^N$ set $f:=\sum_{i=1}^N\alpha_i \frac{\partial u_\varepsilon}{\partial x_i}$. Then the following identities hold:
\begin{equation}\label{eq:localized-u-energy-identity}
 \int_{\mathbb{R}^N} u_\varepsilon^{p+1}\phi_l^2\,dx+\varepsilon^{2s}\int_{\mathbb{R}^{N+1}_+} t^{1-2s}|\nabla\Phi_l|^2\bar{u}_\varepsilon^{\,2}dX=\varepsilon^{2s}\int_{\mathbb{R}^{N+1}_+} t^{1-2s}|\nabla(\Phi_l \bar{u}_\varepsilon)|^2 \,dX+\int_{\mathbb{R}^N} V(x)\phi_l^2 u_\varepsilon^2\,dx,
\end{equation}
  \begin{equation}\label{eq:localized-f-energy-identity}
\begin{aligned}
&\varepsilon^{2s}\int_{\mathbb{R}^{N+1}_+} t^{1-2s}
|\nabla(\Phi_l \bar{f})|^2\,dX+\int_{\mathbb{R}^N}V(x)\phi_l^2 f^2\,dx\\
&\quad= \varepsilon^{2s}\int_{\mathbb{R}^{N+1}_+} t^{1-2s}|\nabla \Phi_l|^2 \bar{f}^{\,2}\,dX+ p\int_{\mathbb{R}^N}u_\varepsilon^{p-1}\phi_l^2 f^2\,dx
-\sum_{i=1}^N\alpha_i\int_{\mathbb{R}^N}\frac{\partial V(x)}{\partial x_i}\,u_\varepsilon\phi_l^2 f\,dx,
\end{aligned}
\end{equation}
  and 
\begin{equation}\label{eq:localized-u-f-cross-identity}
  \begin{aligned}
    &2\, \varepsilon^{2s}\int_{\mathbb{R}^{N+1}_+}t^{1-2s}\nabla (\Phi_l \bar{u}_\varepsilon)\cdot\nabla (\Phi_l \bar{f})dX+2 \int_{\mathbb{R}^N}V(x)f u_\varepsilon \phi_l^2 \,dx\\
    &\quad =2\, \varepsilon^{2s}\int_{\mathbb{R}^{N+1}_+}t^{1-2s}|\nabla \Phi_l|^2 \bar{u}_\varepsilon \bar{f}\,dX+(p+1)\int_{\mathbb{R}^{N}} u_\varepsilon^{p}\phi_{l}^{2} f\,dx-\sum_{i=1}^{N}\alpha_{i}\int_{\mathbb{R}^{N}}\frac{\partial V(x)}{\partial x_{i}}u_\varepsilon^{2}\phi_{l}^{2}\,dx
  \end{aligned}
\end{equation}
\end{lemma}

\begin{proof}
  Note that $\bar{u}_\varepsilon $ and $\bar{f}$ satisfy
\begin{equation}\label{eq:u-extension-equation}
 \left\{
 \begin{aligned}
 \operatorname{div}\!\left(t^{1-2s}\nabla \bar{u}_\varepsilon\right) &= 0
 &&\text{in }\mathbb{R}^{N+1}_+,\\
 \lim_{t\to 0^+}t^{1-2s}\frac{\partial \bar{u}_\varepsilon}{\partial \nu}
 &=\varepsilon^{-2s}\bigl(u_\varepsilon^{\,p}- V(x)\,u_\varepsilon\bigr)
 &&\text{on }\mathbb{R}^N.
 \end{aligned}
 \right.
\end{equation}
and 
\begin{equation}\label{eq:f-extension-equation}
\left\{
\begin{aligned}
\operatorname{div}\!\left(t^{1-2s}\nabla \bar{f}\right) &= 0
&&\text{in }\mathbb{R}^{N+1}_+,\\
\lim_{t\to 0^+}t^{1-2s}\frac{\partial \bar{f}}{\partial \nu}
&=\varepsilon^{-2s}\Bigl(p u_\varepsilon^{\,p-1}f- V(x)\,f
-\sum_{i=1}^N\alpha_i\,\frac{\partial V(x)}{\partial x_i}\,u_\varepsilon\Bigr)
&&\text{on }\mathbb{R}^N.
\end{aligned}
\right.
\end{equation}
Multiplying \eqref{eq:u-extension-equation} by $\Phi_l^2 \bar{u}_\varepsilon $ yields \eqref{eq:localized-u-energy-identity}. Similarly, multiplying \eqref{eq:f-extension-equation} by $\Phi_l^2 \bar{f}$ yields \eqref{eq:localized-f-energy-identity}. Then, multiplying \eqref{eq:u-extension-equation} and \eqref{eq:f-extension-equation} by $\Phi_l^2 \bar{f}$ and $\Phi_l^2 \bar{u}_\varepsilon $ respectively, and adding the resulting identities, we obtain \eqref{eq:localized-u-f-cross-identity}.
\end{proof}
In the next two propositions, we establish estimates for the eigenvalues.

\begin{proposition}\label{prop:mu-zero}
Let $u_\varepsilon$ be a $k$-peak family in the sense of Definition~\ref{def:kpeak-class}. Then there exists $c_0>0$, independent of $\varepsilon$, such that, for all sufficiently small $\varepsilon$,
\[
\mu_{\varepsilon}^{(l)}\le \mu_{\varepsilon}^{(k)}\le -c_0<0,
\qquad l=1,\dots,k.
\]
\end{proposition}

\begin{proof}
For $v\in H^s(\mathbb{R}^N)\setminus\{0\}$ define
\begin{equation*}
J_\varepsilon(v):=
\frac{Q_\varepsilon(v)}{\displaystyle\int_{\mathbb R^N}v^2\,dx},
\end{equation*}
where
\begin{equation}\label{eq:linearized-quadratic-form}
Q_\varepsilon(v)
:=\varepsilon^{2s}\int_{\mathbb R^N}|(-\Delta)^{s/2}v|^2\,dx
+\int_{\mathbb R^N}V(x)v^2\,dx
-p\int_{\mathbb R^N}u_\varepsilon^{p-1}v^2\,dx .
\end{equation}
By the min--max principle, we have
\[
\mu_\varepsilon^{(k)}
=\inf_{\substack{W\subset H^s(\mathbb R^N)\\ \dim W=k}}
\sup_{0\ne v\in W}J_\varepsilon(v).
\]
We use the test space
\[
W_k:=\operatorname{span}\{U_{\varepsilon,1},\ldots,U_{\varepsilon,k}\},
\qquad U_{\varepsilon,j}=\phi_j u_\varepsilon .
\]
By Lemma~\ref{lem:fractional-linear-independence}, $\dim W_k=k$ for sufficiently small $\varepsilon$.

Since $v\in W_k\setminus\{0\}$, we may write
\[
v=\sum_{j=1}^k a_j U_{\varepsilon,j}
=\sum_{j=1}^k a_j \phi_j u_\varepsilon ,
\] 
for some $(a_1,\ldots,a_k) \neq (0,\ldots,0)$. Let $\bar v=E(v)$ be the extension of $v$, and set
\[
\Psi:=\sum_{j=1}^k a_j\Phi_j\bar u_\varepsilon .
\]
Since $\Psi(x,0)=v(x)$, Lemma~\ref{lem:minimality-extension} gives
\begin{equation}\label{eq:mu-zero-extension-minimality}
\int_{\mathbb R^{N+1}_+}t^{1-2s}|\nabla\bar v|^2\,dX
\le
\int_{\mathbb R^{N+1}_+}t^{1-2s}|\nabla\Psi|^2\,dX .
\end{equation}
The supports of $\Phi_1,\ldots,\Phi_k$ are pairwise disjoint.  Hence
\[
\int_{\mathbb R^{N+1}_+}t^{1-2s}|\nabla\Psi|^2\,dX
=
\sum_{j=1}^k a_j^2
\int_{\mathbb R^{N+1}_+}t^{1-2s}|\nabla(\Phi_j\bar u_\varepsilon)|^2\,dX .
\]
Thus, using \eqref{eq:mu-zero-extension-minimality}, we obtain
\begin{equation*}
\begin{aligned}
Q_\varepsilon(v)
&\le \sum_{j=1}^k a_j^2\Bigg[
\varepsilon^{2s}\int_{\mathbb R^{N+1}_+}t^{1-2s}|\nabla(\Phi_j\bar u_\varepsilon)|^2\,dX
+\int_{\mathbb R^N}V(x)\phi_j^2u_\varepsilon^2\,dx \\
&\hspace{36mm}
-p\int_{\mathbb R^N}u_\varepsilon^{p+1}\phi_j^2\,dx\Bigg].
\end{aligned}
\end{equation*}
By the identity \eqref{eq:localized-u-energy-identity}, we have
\begin{equation}\label{eq:mu-zero-quadratic-upper-bound}
Q_\varepsilon(v)
\le
\sum_{j=1}^k a_j^2\left[(1-p)\int_{\mathbb R^N}u_\varepsilon^{p+1}\phi_j^2\,dx
+
\varepsilon^{2s}\int_{\mathbb R^{N+1}_+}t^{1-2s}|\nabla\Phi_j|^2\bar u_\varepsilon^2\,dX\right].
\end{equation}
On the support of $\nabla\Phi_j$, by Lemma~\ref{lem:extension-u-bound}(i),
\[
|\bar u_\varepsilon(x,t)|\le C\varepsilon^N
\quad\text{on }\operatorname{supp}\nabla\Phi_j.
\]
Thus
\begin{equation*}
\varepsilon^{2s}\int_{\mathbb R^{N+1}_+}t^{1-2s}|\nabla\Phi_j|^2\bar u_\varepsilon^2\,dX
\le C\varepsilon^{2s+2N}.
\end{equation*}
Moreover, after the change of variables $x=\xi_{\varepsilon,j}+\varepsilon y$, Lemma~\ref{lem:true-peak-profile-convergence} gives
\begin{align}
\int_{\mathbb R^N}u_\varepsilon^{p+1}\phi_j^2\,dx
&=\varepsilon^N\int_{\mathbb R^N}w_{V(\xi_j^0)}^{p+1}(y)\,dy+o(\varepsilon^N),
\label{eq:localized-u-p-plus-one-mass}\\
\int_{\mathbb R^N}u_\varepsilon^{2}\phi_j^2\,dx
&=\varepsilon^N\int_{\mathbb R^N}w_{V(\xi_j^0)}^{2}(y)\,dy+o(\varepsilon^N).
\label{eq:localized-u-l2-mass}
\end{align}
Since $p>1$ and $\int_{\mathbb R^N}w_{V(\xi_j^0)}^{p+1}>0$, \eqref{eq:mu-zero-quadratic-upper-bound}--\eqref{eq:localized-u-l2-mass} imply
\[
J_\varepsilon(v)
\le
\max_{1\le j\le k}
(1-p)\frac{\int_{\mathbb R^N}w_{V(\xi_j^0)}^{p+1}\,dy}{\int_{\mathbb R^N}w_{V(\xi_j^0)}^{2}\,dy}
+o(1).
\]
Hence there exists $c_0>0$ such that, for all sufficiently small $\varepsilon$,
\[
\sup_{0\ne v\in W_k}J_\varepsilon(v)\le -c_0.
\]
This gives $\mu_\varepsilon^{(k)}\le -c_0$.  Since the eigenvalues are ordered increasingly, we deduce that $\mu_{\varepsilon}^{(l)}\le \mu_{\varepsilon}^{(k)}\le -c_0<0$ for $l=1,\dots,k.$
\end{proof}

\begin{proposition}\label{prop:mu-small}
Let $u_\varepsilon$ be a $k$-peak family in the sense of Definition~\ref{def:kpeak-class}.  Then
\[
\lim_{\varepsilon\to0}\mu_{\varepsilon}^{(l)}=0,
\qquad l=k+1,\ldots,(N+1)k.
\]
\end{proposition}

\begin{proof}
First, the eigenvalues under consideration are uniformly bounded from below.  Indeed, by \eqref{eq:linearized-quadratic-form} and Lemma~\ref{lem:u-polynomial-decay},
\begin{equation*}
   \begin{aligned}
\mu_{\varepsilon}^{(1)}
&=\inf_{v\in H^s(\mathbb{R}^N)\setminus\{0\}}
\frac{
\varepsilon^{2s}\int_{\mathbb{R}^N}|(-\Delta)^{s/2}v|^2\,dx
+\int_{\mathbb{R}^N}V(x)v^2\,dx
-p\int_{\mathbb{R}^N}u_\varepsilon^{p-1}v^2\,dx}{\int_{\mathbb{R}^N}v^2\,dx}  \\
&\ge
\inf_{v\in H^s(\mathbb{R}^{N})\setminus\{0\}}
\frac{-p\int_{\mathbb{R}^{N}}u_{\varepsilon}^{p-1}v^{2}\,dx}{\int_{\mathbb{R}^{N}}v^{2}\,dx}
\ge -p\|u_{\varepsilon}\|_{L^{\infty}(\mathbb{R}^{N})}^{p-1}
\ge -C .
\end{aligned} 
\end{equation*}
Thus all eigenvalues are bounded from below uniformly in $\varepsilon$.

We next prove
\begin{equation*}
\limsup_{\varepsilon\to0}\mu_\varepsilon^{((N+1)k)}\le0.
\end{equation*}
Fix integers $j\in\{1,\dots,N\}$ and $i\in\{1,\dots,k\}$ and set $l=jk+i$.
Define
\[
W_l:=\mathrm{span}\Big\{U_{\varepsilon,1},\dots,U_{\varepsilon,k},\ 
\psi_{\varepsilon,1,1},\dots,\psi_{\varepsilon,k,1},\dots,\psi_{\varepsilon,i,j}\Big\},
\]
where
\[
U_{\varepsilon,m}:=\phi_m u_\varepsilon,
\qquad
\psi_{\varepsilon,m,q}:=\phi_m\frac{\partial u_\varepsilon}{\partial x_q} .
\]
Let $v\in W_l\setminus\{0\}$ be written as
\[
v=f_\varepsilon+g_\varepsilon+h_\varepsilon,
\]
with
\[
f_\varepsilon:=\sum_{j=1}^k\alpha_{\varepsilon,j}\,U_{\varepsilon,j}
=\sum_{j=1}^k\alpha_{\varepsilon,j}\,\phi_j u_\varepsilon,
\]
\[
g_\varepsilon:=\sum_{m=1}^k\sum_{q=1}^{j-1}\beta_{\varepsilon,m,q}\,\psi_{\varepsilon,m,q}
      =\sum_{m=1}^k\phi_m z_{\varepsilon,m},
\qquad
z_{\varepsilon,m}:=\sum_{q=1}^{j-1}\beta_{\varepsilon,m,q}\,\frac{\partial u_\varepsilon}{\partial x_q},
\]
\[
h_\varepsilon:=\sum_{r=1}^i\beta_{\varepsilon,r,j}\,\psi_{\varepsilon,r,j}
      =\sum_{r=1}^i\phi_r \hat z_{\varepsilon,r},
\qquad
\hat z_{\varepsilon,r}:=\beta_{\varepsilon,r,j}\,\frac{\partial u_\varepsilon}{\partial x_j},
\]
where $(\alpha_{\varepsilon,1},\ldots,\alpha_{\varepsilon,k},\beta_{\varepsilon,1,1},\ldots,\beta_{\varepsilon,k,1},\cdots,\beta_{\varepsilon,i,j})\neq(0,\ldots,0)\in\mathbb{R}^{l}.$

By the variational characterization of $\mu_{\varepsilon}^{(l)}$, we have 
\begin{equation}\label{eq:minmax-small-eigenvalue-bound}
  \begin{aligned}
    \mu_{\varepsilon}^{(l)} =& \inf_{\substack{W \subseteq H^s(\mathbb{R}^{N}) \\ \dim W = l}} \max_{v \in W \setminus \{0\}} J_\varepsilon (v) \leq \max_{v \in W_{l} \setminus \{0\}} J_\varepsilon (v) = \max_{v \in W_{l} \setminus \{0\}} J_\varepsilon (f_{\varepsilon} + g_{\varepsilon} + h_{\varepsilon})\\
    =& \max_{v \in W_{l} \setminus \{0\}} \frac{\varepsilon^{2s} \langle f_{\varepsilon} + g_{\varepsilon} + h_{\varepsilon},f_{\varepsilon} + g_{\varepsilon} + h_{\varepsilon}\rangle_{\dot H^s} +\int_{\mathbb{R}^{N}}   (V(x) - p u_{\varepsilon}^{p-1}) (f_{\varepsilon} + g_{\varepsilon} + h_{\varepsilon})^{2}  dx}{\int_{\mathbb{R}^{N}} (f_{\varepsilon} + g_{\varepsilon} + h_{\varepsilon})^{2} dx}
  \end{aligned}
\end{equation}

Let $\bar u_\varepsilon$, $\bar z_{\varepsilon,m}$ ,  and $\overline{\hat z}_{\varepsilon,r}$ denote the $s$-harmonic extensions of $u_\varepsilon$ ,  $z_{\varepsilon,m}$ ,  and $\hat z_{\varepsilon,r}$, respectively, and define
\[
\bar f_\varepsilon:=\sum_{j=1}^k\alpha_{\varepsilon,j}\,\Phi_j\,\bar u_\varepsilon,
\qquad
\bar g_\varepsilon:=\sum_{m=1}^k\Phi_m\,\bar z_{\varepsilon,m},
\qquad
\bar h_\varepsilon:=\sum_{r=1}^i\Phi_r\,\overline{\hat z}_{\varepsilon,r}.
\]
By Lemma~\ref{lem:minimality-extension}, we have 
\begin{equation}\label{eq:extension-minimality-test-space}
  \langle f_{\varepsilon} + g_{\varepsilon} + h_{\varepsilon},f_{\varepsilon} + g_{\varepsilon} + h_{\varepsilon}\rangle_{\dot H^s}\leq \int_{\mathbb{R}^{N+1}_+}\! t^{1-2s} \Big(|\nabla(\bar f_\varepsilon+ \bar g_\varepsilon+ \bar h_\varepsilon)|^2\Big)\,dX.
\end{equation}
Moreover, direct computation gives
\begin{equation}\label{eq:nonlinear-test-space-expansion}
  \begin{aligned}[b]
    \int_{\mathbb{R}^N}&u_\varepsilon^{p-1}(f_\varepsilon+g_\varepsilon+h_\varepsilon)^2dx\\
    =&\sum_{m=1}^k\alpha_{\varepsilon,m}^2\int_{\mathbb{R}^N}u_\varepsilon^{p+1}\phi_m^2dx
    +2\sum_{m=1}^k\alpha_{\varepsilon,m}\int_{\mathbb{R}^N}u_\varepsilon^p\phi_m^2z_{\varepsilon,m}\,dx
    +2\sum_{r=1}^i\alpha_{\varepsilon,r}\int_{\mathbb{R}^N}u_\varepsilon^p\phi_r^2\hat z_{\varepsilon,r}\,dx\\
    &+\sum_{m=1}^k\int_{\mathbb{R}^N}u_\varepsilon^{p-1}z_{\varepsilon,m}^2\phi_m^2dx
    +\sum_{r=1}^i\int_{\mathbb{R}^N}u_\varepsilon^{p-1}\hat z_{\varepsilon,r}^2\phi_r^2dx
    +2\sum_{r=1}^i\int_{\mathbb{R}^N}u_\varepsilon^{p-1}z_{\varepsilon,r}\hat z_{\varepsilon,r}\phi_r^2dx.
  \end{aligned}
\end{equation}
By \eqref{eq:extension-minimality-test-space}, \eqref{eq:nonlinear-test-space-expansion} and Lemma~\ref{lem:appendix-energy-expansions}, we obtain
\begin{equation}\label{eq:test-space-quadratic-form-bound}
   \begin{aligned}[b]
Q_\varepsilon(v)
&\le
\varepsilon^{2s}\int_{\mathbb{R}^{N+1}_+} t^{1-2s}
|\nabla(\bar f_\varepsilon+\bar g_\varepsilon+\bar h_\varepsilon)|^2\,dX\\
&\qquad\quad +\int_{\mathbb{R}^N}V(x)(f_\varepsilon+g_\varepsilon+h_\varepsilon)^2\,dx
-p\int_{\mathbb{R}^N}u_\varepsilon^{p-1}v^2\,dx;\\
&\qquad\quad =J_{1,\varepsilon }+J_{2,\varepsilon }
\end{aligned} 
\end{equation}
where
\begin{align*}
J_{1,\varepsilon}(v)
&:= (1-p)\sum_{m=1}^k\alpha_{\varepsilon,m}^2\int_{\mathbb{R}^N}u_\varepsilon^{p+1}\phi_m^2\,dx
+(1-p)\sum_{m=1}^k\alpha_{\varepsilon,m}\int_{\mathbb{R}^N}u_\varepsilon^{p}\phi_m^2 z_{\varepsilon,m}\,dx \notag\\
&\quad +(1-p)\sum_{r=1}^i\alpha_{\varepsilon,r}
\int_{\mathbb{R}^N}u_\varepsilon^{p}\phi_r^2 \hat z_{\varepsilon,r}\,dx \notag\\
&\quad +\varepsilon^{2s}\Bigg\{
\sum_{m=1}^k \alpha_{\varepsilon,m}^2
\int_{\mathbb{R}^{N+1}_+} t^{1-2s}|\nabla\Phi_m|^2\bar u_\varepsilon^{\,2}\,dX \notag\\
&\qquad
+\sum_{m=1}^k\int_{\mathbb{R}^{N+1}_+} t^{1-2s}|\nabla\Phi_m|^2\bar z_{\varepsilon,m}^{\,2}\,dX
+\sum_{r=1}^i\int_{\mathbb{R}^{N+1}_+} t^{1-2s}|\nabla\Phi_r|^2\bar{\hat z}_{\varepsilon,r}^{\,2}\,dX \notag\\
&\qquad
+2\sum_{m=1}^k\alpha_{\varepsilon,m}
\int_{\mathbb{R}^{N+1}_+} t^{1-2s}|\nabla\Phi_m|^2\bar u_\varepsilon\,\bar z_{\varepsilon,m}\,dX \notag\\
&\qquad
+2\sum_{r=1}^i\alpha_{\varepsilon,r}
\int_{\mathbb{R}^{N+1}_+} t^{1-2s}|\nabla\Phi_r|^2\bar u_\varepsilon\,\bar{\hat z}_{\varepsilon,r}\,dX \notag\\
&\qquad
+2\sum_{r=1}^i\int_{\mathbb{R}^{N+1}_+} t^{1-2s}|\nabla\Phi_r|^2\bar z_{\varepsilon,r}\,\bar{\hat z}_{\varepsilon,r}\,dX
\Bigg\},
\end{align*}
and 
\begin{align*}
J_{2,\varepsilon}(v)
&:=
-\sum_{m=1}^k\sum_{q=1}^{j-1}\beta_{\varepsilon,m,q}
\int_{\mathbb{R}^N}\frac{\partial V}{\partial x_q}(x)\,u_\varepsilon\,\phi_m^2 z_{\varepsilon,m}\,dx
-\sum_{r=1}^i\beta_{\varepsilon,r,j}
\int_{\mathbb{R}^N}\frac{\partial V}{\partial x_j}(x)\,u_\varepsilon\,\phi_r^2 \hat z_{\varepsilon,r}\,dx
\nonumber\\
&\quad
-\sum_{m=1}^k\sum_{q=1}^{j-1}\alpha_{\varepsilon,m}\beta_{\varepsilon,m,q}
\int_{\mathbb{R}^N}\frac{\partial V}{\partial x_q}(x)\,u_\varepsilon^{2}\phi_m^2\,dx
-\sum_{r=1}^i\alpha_{\varepsilon,r}\beta_{\varepsilon,r,j}
\int_{\mathbb{R}^N}\frac{\partial V}{\partial x_j}(x)\,u_\varepsilon^{2}\phi_r^2\,dx
\nonumber\\
&\quad
-\sum_{r=1}^i\sum_{q=1}^{j-1}\beta_{\varepsilon,r,q}
\int_{\mathbb{R}^N}\frac{\partial V}{\partial x_q}(x)\,u_\varepsilon\,\phi_r^2 \hat z_{\varepsilon,r}\,dx
-\sum_{r=1}^i\beta_{\varepsilon,r,j}
\int_{\mathbb{R}^N}\frac{\partial V}{\partial x_j}(x)\,u_\varepsilon\,\phi_r^2 z_{\varepsilon,r}\,dx,
\end{align*}

Set
\[
J_{3,\varepsilon}(v):=\int_{\mathbb{R}^N}v^2\,dx.
\]

As in \eqref{eq:nonlinear-test-space-expansion},
\begin{align}\label{eq:test-space-l2-expansion}
J_{3,\varepsilon}(v)
&=
\sum_{j=1}^k\alpha_{\varepsilon,j}^2\int_{\mathbb{R}^N}\phi_j^2 u_\varepsilon^2\,dx
+2\sum_{j=1}^k\alpha_{\varepsilon,j}\int_{\mathbb{R}^N}\phi_j^2 u_\varepsilon z_{\varepsilon,j}\,dx
+2\sum_{r=1}^i\alpha_{\varepsilon,r}\int_{\mathbb{R}^N}\phi_r^2 u_\varepsilon \hat z_{\varepsilon,r}\,dx
\nonumber\\
&\quad
+\sum_{m=1}^k\int_{\mathbb{R}^N}\phi_m^2 z_{\varepsilon,m}^2\,dx
+\sum_{r=1}^i\int_{\mathbb{R}^N}\phi_r^2 \hat z_{\varepsilon,r}^2\,dx
+2\sum_{r=1}^i\int_{\mathbb{R}^N}\phi_r^2 z_{\varepsilon,r}\hat z_{\varepsilon,r}\,dx.
\end{align}
Then, \eqref{eq:test-space-quadratic-form-bound} and \eqref{eq:test-space-l2-expansion} imply
\begin{equation*}
  J_\varepsilon(v):=\frac{Q_\varepsilon(v)}{\int_{\mathbb{R}^N}v^2\,dx}\leq\frac{J_{1,\varepsilon}+J_{2,\varepsilon}}{J_{3,\varepsilon}}.
\end{equation*}
By \eqref{eq:minmax-small-eigenvalue-bound} and Lemma~\ref{lem:fractional-test-space-estimate}, we have
\begin{equation}\label{eq:small-eigenvalue-upper-bound}
  \mu_{\varepsilon}^{(l)} \leq
  \max_{v\in W_{l}\setminus\{0\}}
  \frac{J_{1,\varepsilon}+J_{2,\varepsilon}}{J_{3,\varepsilon}}
  =o(\varepsilon ).
\end{equation}
Letting \(\varepsilon\to0\) in \eqref{eq:small-eigenvalue-upper-bound}, we obtain
\begin{equation}\label{eq:limsup-all-small}
\limsup_{\varepsilon\to0}\mu_\varepsilon^{(l)}\le0.
\end{equation}

It remains to prove the lower bound
\begin{equation}\label{eq:liminf-k-plus-one}
\liminf_{\varepsilon\to0}\mu_\varepsilon^{(k+1)}\ge0.
\end{equation}
Suppose for contradiction that $\liminf\limits_{\varepsilon\to0}\mu_\varepsilon^{(k+1)}<0$. Then, along a subsequence still denoted by $\varepsilon$, there exists a constant $\delta>0$ such that
\begin{equation*}
\mu_\varepsilon^{(k+1)}\le -\delta,
\end{equation*}
for sufficiently small $\varepsilon$.

Let \(e_\varepsilon^{(1)},\ldots,e_\varepsilon^{(k+1)}\) be \(L^2\)-orthonormal eigenfunctions corresponding to the first $k+1$ eigenvalues, and define
\[
E_\varepsilon^-:=\operatorname{span}\{e_\varepsilon^{(1)},\ldots,e_\varepsilon^{(k+1)}\}.
\]
Then every $\varphi\in E_\varepsilon^-$ satisfies
\begin{equation}\label{eq:negative-space-bound}
Q_\varepsilon(\varphi)
\le
-\delta\int_{\mathbb R^N}\varphi^2\,dx.
\end{equation}
Indeed, writing $\varphi=\sum\limits_{m=1}^{k+1}c_m e_\varepsilon^{(m)}$ and using \eqref{eq:linearized-quadratic-form}, we obtain
\[
Q_\varepsilon(\varphi)=\sum_{m=1}^{k+1}\mu_\varepsilon^{(m)}c_m^2
\le -\delta\sum_{m=1}^{k+1}c_m^2
=-\delta\int\varphi^2.
\]

Fix $R>1$. Let $\phi \in C_c^\infty(\mathbb{R}^N)$ be a radial cut-off function such that
\[
0\le \phi \le 1,\qquad
\phi \equiv 1 \ \text{in }B_1(0),\qquad
\operatorname{supp}\,\phi \subset B_{2}(0), \qquad |\nabla\phi|\le C,
\]
and such that $(1-\phi^2)^{1/2}$ is smooth. For fixed $R>1$, define
\[\chi_{\varepsilon,j}(x):=\phi\left(\frac{x-\xi_{\varepsilon,j}}{R\varepsilon}\right),\quad j=1,\ldots,k.\]
Finally set
\[\chi_{\varepsilon,0}(x):=\left(1-\sum_{j=1}^k\chi_{\varepsilon,j}(x)^2\right)^{1/2}.\]
Then
\[
\sum_{j=0}^k\chi_{\varepsilon,j}^2=1,
\]
and $\chi_{\varepsilon,j}$ is supported in $B_{2R\varepsilon}(\xi_{\varepsilon,j})$ and equals $1$ in $B_{R\varepsilon}(\xi_{\varepsilon,j})$ for $j=1,\ldots,k$, while $\chi_{\varepsilon,0}$ is supported in $\mathbb R^N\setminus\bigcup_{j=1}^kB_{R\varepsilon}(\xi_{\varepsilon,j})$.  The supports of $\chi_{\varepsilon,i}$ and $\chi_{\varepsilon,j}$ are disjoint for $i\ne j$ and $i,j\ge1$, for sufficiently small $\varepsilon$.

Since $\sum_{j=0}^k\chi_{\varepsilon,j}^2\equiv1$ and
$\|\nabla\chi_{\varepsilon,j}\|_{L^\infty}\le C(R\varepsilon)^{-1}$ for
$j=0,\ldots,k$, we have
\[
\sum_{j=0}^k(\chi_{\varepsilon,j}(x)-\chi_{\varepsilon,j}(y))^2
\le C\frac{|x-y|^2}{(R\varepsilon)^2},
\]
and also
\[
\sum_{j=0}^k(\chi_{\varepsilon,j}(x)-\chi_{\varepsilon,j}(y))^2
\le
2\sum_{j=0}^k\chi_{\varepsilon,j}(x)^2
+
2\sum_{j=0}^k\chi_{\varepsilon,j}(y)^2
=4.
\]
Therefore
\begin{equation}\label{eq:cutoff-difference-ims}
\sum_{j=0}^k(\chi_{\varepsilon,j}(x)-\chi_{\varepsilon,j}(y))^2
\le
C\min\left\{1,\frac{|x-y|^2}{(R\varepsilon)^2}\right\}.
\end{equation}
We claim that, for every $\varphi\in H^s(\mathbb R^N)$,
\begin{equation}\label{eq:ims-localized-form}
\varepsilon^{2s}\int_{\mathbb R^N}|(-\Delta)^{s/2}\varphi|^2\,dx
\ge
\sum_{j=0}^k\varepsilon^{2s}\int_{\mathbb R^N}|(-\Delta)^{s/2}(\chi_{\varepsilon,j}\varphi)|^2\,dx
-CR^{-2s}\int_{\mathbb R^N}\varphi^2\,dx.
\end{equation}
We recall the proof of \eqref{eq:ims-localized-form}.  By the identity $\sum\limits_{j=0}^k\chi_{\varepsilon,j}^2=1$ and Lemma~3.5 of \cite{frank2008hardy}, one has
\begin{equation}\label{eq:ims-identity}
  \sum_{j=0}^k\|(-\Delta)^{s/2}(\chi_{\varepsilon,j}\varphi)\|_2^2
=
\|(-\Delta)^{s/2}\varphi\|_2^2
+c_{N,s}\iint\frac{\varphi(x)\varphi(y)\Gamma_\varepsilon(x,y)}{|x-y|^{N+2s}}\,dxdy,
\end{equation}
where
\[
\Gamma_\varepsilon(x,y)=\sum_{j=0}^k(\chi_{\varepsilon,j}(x)-\chi_{\varepsilon,j}(y))^2.
\]
Using \eqref{eq:cutoff-difference-ims}, we have
\begin{equation}\label{eq:cross}
\begin{aligned}
\iint\frac{\varphi(x)\varphi(y)\Gamma_\varepsilon(x,y)}{|x-y|^{N+2s}}\,dxdy
&\le
\frac{1}{2}
\iint
\frac{(\varphi(x)^2+\varphi(y)^2)\Gamma_\varepsilon(x,y)}{|x-y|^{N+2s}}\,dx\,dy\\
&\le
C\int_{\mathbb R^N}\varphi(x)^2
\left(
\int_{\mathbb R^N}
\frac{
\min\left\{1,\frac{|x-y|^2}{(R\varepsilon)^2}\right\}
}{|x-y|^{N+2s}}\,dy
\right)dx\\
&=
C\int_{\mathbb R^N}\varphi(x)^2\,dx
\int_{\mathbb R^N}
\frac{\min\left\{1,\frac{|h|^2}{a^2}\right\}}{|h|^{N+2s}}\,dh
\qquad
\bigl(h=y-x,\ a=R\varepsilon\bigr)\\
&\le
C a^{-2s}
\int_{\mathbb R^N}\varphi(x)^2\,dx
=
C(R\varepsilon)^{-2s}
\int_{\mathbb R^N}\varphi^2\,dx .
\end{aligned}
\end{equation}
Combining \eqref{eq:ims-identity} and \eqref{eq:cross} gives \eqref{eq:ims-localized-form}. Since $\sum_j\chi_{\varepsilon,j}^2=1$, we obtain
\begin{equation}\label{eq:ims-quadratic-form-bound}
Q_\varepsilon(\varphi)
\ge
\sum_{j=0}^k Q_\varepsilon(\chi_{\varepsilon,j}\varphi)
-CR^{-2s}\int_{\mathbb R^N}\varphi^2\,dx.
\end{equation}

On the support of $\chi_{\varepsilon,0}$ we have $|x-\xi_{\varepsilon,j}|\ge R\varepsilon$ for every $j$.  Lemma~\ref{lem:u-polynomial-decay} gives
\[
u_\varepsilon(x)\le CR^{-\beta_\sigma}\quad\text{on }\operatorname{supp}\chi_{\varepsilon,0}.
\]
Choosing $R$ large enough, independently of $\varepsilon$, we have
\[
p u_\varepsilon^{p-1}(x)\le \frac{1}{2}V_{\min}
\quad\text{on }\operatorname{supp}\chi_{\varepsilon,0}.
\]
Thus
\begin{equation}\label{eq:outer-region-nonnegative}
  \begin{aligned}Q_{\varepsilon}(\chi_{\varepsilon,0}\phi)&=\varepsilon^{2s}\int|(-\Delta)^{s/2}(\chi_{\varepsilon,0}\phi)|^2+\int V(\chi_{\varepsilon,0}\phi)^2-p\int u_\varepsilon^{p-1}(\chi_{\varepsilon,0}\phi)^2\\&\geq\varepsilon^{2s}\int|(-\Delta)^{s/2}(\chi_{\varepsilon,0}\phi)|^2+\frac{1}{2}V_{\min}\int(\chi_{\varepsilon,0}\phi)^2\geq0.\end{aligned}
\end{equation}

For each $j$, let $Z_j$ be the $L^2$-normalized positive eigenfunction corresponding to the unique negative eigenvalue of the operator
\[
\mathcal L_j=(-\Delta)^s+V(\xi_j^0)-pw_{V(\xi_j^0)}^{p-1}.
\]
Let
\[
\mathfrak q_j(\psi):=\int_{\mathbb R^N}|(-\Delta)^{s/2}\psi|^2\,dy
+V(\xi_j^0)\int_{\mathbb R^N}\psi^2\,dy
-p\int_{\mathbb R^N}w_{V(\xi_j^0)}^{p-1}\psi^2\,dy.
\]
Since Lemma~\ref{lem:single-bubble-input} states that $\mathcal L_j$ has exactly one negative eigenvalue, we have
\begin{equation}\label{eq:single-negative-orth}
\mathfrak q_j(\psi)\ge0
\quad\text{whenever}\quad
\int_{\mathbb R^N}\psi Z_j=0.
\end{equation}

For any \(\varphi\in E_\varepsilon^-\),
\[
\varphi=\sum_{m=1}^{k+1}c_m e_\varepsilon^{(m)} .
\]
The $k$ conditions
\begin{equation*}
\int_{\mathbb R^N}\chi_{\varepsilon,j}(x)\varphi(x)
Z_j\left(\frac{x-\xi_{\varepsilon,j}}{\varepsilon}\right)dx=0,
\qquad j=1,\ldots,k,
\end{equation*} are then equivalent to
\[
\sum_{m=1}^{k+1} a_{jm}^\varepsilon c_m=0,
\qquad j=1,\ldots,k,
\]
where
\[
a_{jm}^\varepsilon
:=
\int_{\mathbb R^N}
\chi_{\varepsilon,j}(x)e_\varepsilon^{(m)}(x)
Z_j\left(\frac{x-\xi_{\varepsilon,j}}{\varepsilon}\right)\,dx .
\]
This is a homogeneous linear system with \(k\) equations and \(k+1\) unknowns.
Therefore it admits a nontrivial solution
\[
(c_1^\varepsilon,\ldots,c_{k+1}^\varepsilon)\neq 0 .
\]
Consequently,
\[
\varphi_\varepsilon
:=
\sum_{m=1}^{k+1}c_m^\varepsilon e_\varepsilon^{(m)}
\]
is a nonzero element of \(E_\varepsilon^-\) satisfying
\begin{equation}\label{eq:k-linear-conditions}
\int_{\mathbb R^N}
\chi_{\varepsilon,j}(x)\varphi_\varepsilon(x)
Z_j\left(\frac{x-\xi_{\varepsilon,j}}{\varepsilon}\right)\,dx=0,
\qquad j=1,\ldots,k .
\end{equation}
After rescaling, we may assume 
\begin{equation}\label{eq:l2-normalize-varphi}
  \int_{\mathbb R^N}\varphi_\varepsilon^2\,dx=1 .
\end{equation}

From \eqref{eq:negative-space-bound},
\begin{equation}\label{eq:negative-varphi}
Q_\varepsilon(\varphi_\varepsilon)\le -\delta.
\end{equation}
For $j=1,\ldots,k$, set
\[
\psi_{\varepsilon,j}(y):=
\chi_{\varepsilon,j}(\xi_{\varepsilon,j}+\varepsilon y)
\varphi_\varepsilon(\xi_{\varepsilon,j}+\varepsilon y).
\]
Then \eqref{eq:k-linear-conditions} gives
\[
\int_{\mathbb R^N}\psi_{\varepsilon,j}(y)Z_j(y)\,dy=0,
\]
and therefore, by \eqref{eq:single-negative-orth},
\begin{equation}\label{eq:localized-quadratic-form-nonnegative}
\mathfrak q_j(\psi_{\varepsilon,j})\ge0.
\end{equation}
On the fixed ball $B_{2R}(0)$, by Lemma~\ref{lem:true-peak-profile-convergence}, we have
\[
V(\xi_{\varepsilon,j}+\varepsilon y)\to V(\xi_j^0),
\qquad
u_\varepsilon(\xi_{\varepsilon,j}+\varepsilon y)\to w_{V(\xi_j^0)}(y)
\]
uniformly. Then
\[
\begin{aligned}
Q_\varepsilon(\chi_{\varepsilon,j}\varphi_\varepsilon)
&=
\varepsilon^N
\bigg[
\int_{\mathbb R^N}
\left|(-\Delta)^{s/2}\psi_{\varepsilon,j}\right|^2\,dy
+
\int_{\mathbb R^N}
V(\xi_{\varepsilon,j}+\varepsilon y)\psi_{\varepsilon,j}^2\,dy \\
&\qquad\qquad
-
p\int_{\mathbb R^N}
u_\varepsilon(\xi_{\varepsilon,j}+\varepsilon y)^{p-1}
\psi_{\varepsilon,j}^2\,dy
\bigg] \\
&=
\varepsilon^N \mathfrak q_j(\psi_{\varepsilon,j}) \\
&\quad
+
\varepsilon^N
\int_{\mathbb R^N}
\bigl[
V(\xi_{\varepsilon,j}+\varepsilon y)-V(\xi_j^0)
\bigr]\psi_{\varepsilon,j}^2\,dy \\
&\quad
-
p\varepsilon^N
\int_{\mathbb R^N}
\bigl[
u_\varepsilon(\xi_{\varepsilon,j}+\varepsilon y)^{p-1}
-
w_{V(\xi_j^0)}(y)^{p-1}
\bigr]\psi_{\varepsilon,j}^2\,dy .
\end{aligned}
\]
Since \(k\) is finite, for sufficiently small \(\varepsilon\), the following bound holds for every \(j=1,\ldots,k\),
\[
\sup_{|y|\le 2R}
\left(
\left|V(\xi_{\varepsilon,j}+\varepsilon y)-V(\xi_j^0)\right|
+
p\left|
u_\varepsilon(\xi_{\varepsilon,j}+\varepsilon y)^{p-1}
-
w_{V(\xi_j^0)}(y)^{p-1}
\right|
\right)
\le \frac{\delta}{4}.
\]
Therefore, using \eqref{eq:l2-normalize-varphi} and \eqref{eq:localized-quadratic-form-nonnegative}, we obtain
\begin{equation}\label{eq:local-quadratic-form-nonnegative}
  \begin{aligned}
Q_\varepsilon(\chi_{\varepsilon,j}\varphi_\varepsilon)
&\ge
-\frac{\delta}{4}
\varepsilon^N
\int_{\mathbb R^N}\psi_{\varepsilon,j}^2\,dy  \\
&=
-\frac{\delta}{4}
\int_{\mathbb R^N}
(\chi_{\varepsilon,j}\varphi_\varepsilon)^2\,dx.
\end{aligned}
\end{equation}

Combining \eqref{eq:ims-quadratic-form-bound}, \eqref{eq:outer-region-nonnegative} and \eqref{eq:local-quadratic-form-nonnegative}, we obtain
\[
Q_\varepsilon(\varphi_\varepsilon)
\ge
-\frac{\delta}{4}-CR^{-2s}.
\]
Choose $R$ so large that $CR^{-2s}<\delta/4$. Then
\[
Q_\varepsilon(\varphi_\varepsilon)\ge-\frac{\delta}{2},
\]
contradicting \eqref{eq:negative-varphi}.  Thus \eqref{eq:liminf-k-plus-one} holds.

Finally, by monotonicity of the eigenvalues,
\[
\mu_\varepsilon^{(k+1)}\le\mu_\varepsilon^{(l)}\le\mu_\varepsilon^{((N+1)k)},
\qquad l=k+1,\ldots,(N+1)k.
\]
Together with \eqref{eq:limsup-all-small} and \eqref{eq:liminf-k-plus-one}, this proves
$\mu_\varepsilon^{(l)}\to0$ for $l=k+1,\ldots,(N+1)k$.
\end{proof}

We next prove the asymptotic behavior of the eigenfunctions $v_\varepsilon ^{(l)}$, for $l\in\{k+1,\dots,(N+1)k\}$.

\begin{lemma}\label{lem:small-eigenfunction-orthogonality}
Let $u_\varepsilon$ be a $k$-peak family in the sense of Definition~\ref{def:kpeak-class}, and let
$\{(\mu_{\varepsilon}^{(l)},v_\varepsilon^{(l)})\}_{l\ge1}$ be the eigenpairs of \eqref{eq:linearized-eigenvalue-problem}. For
$l\in\{k+1,\dots,(N+1)k\}$ and $j=1,\dots,k$, the eigenfunctions $v_{\varepsilon,j}^{(l)}$ satisfy 
\[
v_{\varepsilon,j}^{(l)}\to v_j^{(l)}
\quad\text{in }C^1_{\rm loc}(\mathbb R^N),
\]
where
\begin{equation}\label{eq:limit-translation-combination}
v_j^{(l)}(x)=\sum_{i=1}^N a_{j,i}^{(l)}\,\partial_i w_{V(\xi_j^0)}(x).
\end{equation}
The full coefficient vector
\[
a^{(l)}:=\bigl(a_{1,1}^{(l)},\ldots,a_{1,N}^{(l)},\ldots,a_{k,1}^{(l)},\ldots,a_{k,N}^{(l)}\bigr)\in\mathbb R^{kN}
\]
is nonzero.  Moreover, for $k+1\le l\ne l'\le (N+1)k$,
\begin{equation}\label{eq:orth-full-vector}
\sum_{j=1}^k C_j\,a_j^{(l)}\cdot a_j^{(l')}=0,
\qquad
C_j:=\frac{1}N\int_{\mathbb R^N}|\nabla w_{V(\xi_j^0)}|^2>0,
\end{equation}
where $a_j^{(l)}=(a_{j,1}^{(l)},\ldots,a_{j,N}^{(l)})$.
\end{lemma}

\begin{proof}
For $l\in\{k+1,\dots,(N+1)k\}$, by \eqref{eq:linearized-eigenvalue-problem}, $v_{\varepsilon,j}^{(l)}$ satisfies
\begin{equation}\label{eq:scaled-small-eigenfunction-equation}
(-\Delta)^s v_{\varepsilon,j}^{(l)} + V(\varepsilon x+ \xi_{\varepsilon,j})\,v_{\varepsilon,j}^{(l)}
- p\,u_\varepsilon(\varepsilon x+ \xi_{\varepsilon,j})^{p-1}\,v_{\varepsilon,j}^{(l)}
= \mu_{\varepsilon}^{(l)}\, v_{\varepsilon,j}^{(l)}
\quad\text{in } \mathbb{R}^N .
\end{equation}
By Proposition~\ref{prop:mu-small}, $\mu_{\varepsilon}^{(l)}\to0$. The uniform bound in Lemma~\ref{lem:u-polynomial-decay} gives $v_{\varepsilon,j}^{(l)}\to v_j^{(l)}$ in $C^1_{\rm loc}(\mathbb{R}^N)$.  Passing to the limit in \eqref{eq:scaled-small-eigenfunction-equation}, we obtain
\begin{equation*}
(-\Delta)^s v_j^{(l)} + V(\xi_j^0) v_j^{(l)} - p\,w_{V(\xi_j^0)}^{p-1} v_j^{(l)}=0
\quad\text{in } \mathbb{R}^N.
\end{equation*}
By Lemma~\ref{lem:single-bubble-input}, we have
\[
\ker\bigl(( -\Delta)^s + V(\xi_j^0) - p\,w_{V(\xi_j^0)}^{p-1}\bigr)
=\operatorname{span}\{\partial_1w_{V(\xi_j^0)},\ldots,\partial_Nw_{V(\xi_j^0)}\},
\]
which gives \eqref{eq:limit-translation-combination}.

We now prove that the full vector $a^{(l)}$ is nonzero. Suppose, for contradiction, that
\[
a_j^{(l)}=0\quad\text{for every }j=1,\ldots,k.
\]
Then $v_j^{(l)}\equiv0$ for every $j$.  From the eigenvalue equation and normalization in \eqref{eq:linearized-eigenvalue-problem}, we have
\begin{equation}\label{eq:norm-eigen-identity-nonzero}
1
=p\varepsilon^{-N}\int_{\mathbb R^N}u_\varepsilon^{p-1}|v_\varepsilon^{(l)}|^2\,dx
+\mu_\varepsilon^{(l)}\varepsilon^{-N}\int_{\mathbb R^N}|v_\varepsilon^{(l)}|^2\,dx .
\end{equation}
Since $\mu_\varepsilon^{(l)}\to0$ and
\[V_{\min}\int_{\mathbb{R}^N}|v_\varepsilon^{(l)}|^2dx\leq\int_{\mathbb{R}^N}V(x)|v_\varepsilon^{(l)}|^2dx\leq\int_{\mathbb{R}^N}\left(\varepsilon^{2s}\left|(-\Delta)^{s/2}v_\varepsilon^{(l)}\right|^2+V(x)|v_\varepsilon^{(l)}|^2\right)dx=\varepsilon ^N,\]
we have 
\begin{equation}\label{eq:eigenvalue-l2-term-small}
|\mu_\varepsilon^{(l)}|\varepsilon^{-N}\int_{\mathbb R^N}|v_\varepsilon^{(l)}|^2\,dx<\frac{1}{2},
\end{equation}
for sufficiently small $\varepsilon$.

Set
\[
I_\varepsilon
:=
\varepsilon^{-N}
\int_{\mathbb R^N}
u_\varepsilon^{p-1}
\left|v_\varepsilon^{(l)}\right|^2\,dx .
\]
Fix \(R>1\). For sufficiently small \(\varepsilon\), we have
\[
\begin{aligned}
I_\varepsilon
&=
\sum_{j=1}^k
\varepsilon^{-N}
\int_{B_{R\varepsilon}(\xi_{\varepsilon,j})}
u_\varepsilon^{p-1}
\left|v_\varepsilon^{(l)}\right|^2\,dx
+
\varepsilon^{-N}
\int_{\mathbb R^N\setminus
\bigcup_{j=1}^kB_{R\varepsilon}(\xi_{\varepsilon,j})}
u_\varepsilon^{p-1}
\left|v_\varepsilon^{(l)}\right|^2\,dx\\
&=
\sum_{j=1}^k
\int_{B_R(0)}
u_\varepsilon(\xi_{\varepsilon,j}+\varepsilon y)^{p-1}
\left|v_{\varepsilon,j}^{(l)}(y)\right|^2\,dy
+
\varepsilon^{-N}
\int_{\mathbb R^N\setminus
\bigcup_{j=1}^kB_{R\varepsilon}(\xi_{\varepsilon,j})}
u_\varepsilon^{p-1}
\left|v_\varepsilon^{(l)}\right|^2\,dx .\\
&\le
C\sum_{j=1}^k
\int_{B_R(0)}
\left|v_{\varepsilon,j}^{(l)}(y)\right|^2\,dy
+
C R^{-\beta_\sigma(p-1)}
\varepsilon^{-N}
\int_{\mathbb R^N}
\left|v_\varepsilon^{(l)}\right|^2\,dx\\
&\le
C\sum_{j=1}^k
\left\|v_{\varepsilon,j}^{(l)}\right\|_{L^2(B_R)}^2
+
C R^{-\beta_\sigma(p-1)},
\end{aligned}
\]
where we used Lemma~\ref{lem:u-polynomial-decay}. Since \(a_j^{(l)}=0\) for every \(j\), we have
\[
v_{\varepsilon,j}^{(l)}\to0
\quad\text{in }L^2(B_R)
\qquad
\text{for every fixed }R>1,
\]
and hence, for sufficiently small $\varepsilon$,
\[
I_\varepsilon \le C R^{-\beta_\sigma(p-1)} .
\]
   Choosing  \(R\) sufficiently large, we obtain
 \begin{equation}\label{eq:weighted-eigenfunction-term-small}
  p\varepsilon^{-N}\int_{\mathbb R^N}u_\varepsilon^{p-1}|v_\varepsilon^{(l)}|^2\,dx<\frac{1}{2}.
\end{equation}

Using \eqref{eq:eigenvalue-l2-term-small} and \eqref{eq:weighted-eigenfunction-term-small}, we obtain a contradiction to \eqref{eq:norm-eigen-identity-nonzero}.  Hence $a^{(l)}\ne0$.

Fix $R > 1$. Since the eigenfunctions are chosen \(L^2\)-orthogonal, Lemma~\ref{lem:u-polynomial-decay} implies that, for \(k+1\le l\neq l'\le (N+1)k\),
\begin{equation*}
\begin{aligned}[b]
0
&=
\varepsilon^{-N}
\int_{\mathbb R^N}v_\varepsilon^{(l)}v_\varepsilon^{(l')}\,dx \\
&=
\sum_{j=1}^k
\int_{B_R(0)}
v_{\varepsilon,j}^{(l)}(y)v_{\varepsilon,j}^{(l')}(y)\,dy
+
\varepsilon^{-N}
\int_{\mathbb R^N\setminus\bigcup_{j=1}^kB_{R\varepsilon}(\xi_{\varepsilon,j})}
v_\varepsilon^{(l)}v_\varepsilon^{(l')}\,dx \\
&=
\lim_{R\to\infty}\lim_{\varepsilon\to0}
\left[\sum_{j=1}^k
\int_{B_R(0)}
v_{\varepsilon,j}^{(l)}(y)v_{\varepsilon,j}^{(l')}(y)\,dy+O(R^{-(2\beta_\sigma-N)})\right]\\
&=
\sum_{j=1}^k a_j^{(l)}\cdot a_j^{(l')}\frac{1}N\int_{\mathbb R^N}|\nabla w_{V(\xi_j^0)}|^2,
\end{aligned}
\end{equation*}
where $a_j^{(l)}=(a_{j,1}^{(l)},\ldots,a_{j,N}^{(l)})$. In the last equality, we used   \eqref{eq:limit-translation-combination}. The lemma is proved.

\end{proof}

A direct computation gives the following local Pohozaev identity.

\begin{lemma}\label{lem:pohozaev-halfball}
Let $u_\varepsilon$ be a $k$-peak family in the sense of Definition~\ref{def:kpeak-class}, and let
$(\mu_\varepsilon^{(l)},v_\varepsilon^{(l)})$ be an eigenpair of the linearized problem \eqref{eq:linearized-eigenvalue-problem}. Then, for every $i\in\{1,\dots,N\}$, the following identity holds:
\begin{align}\label{eq:local-pohozaev-identity}
&-\int_{\partial'\mathcal{B}_{r}^{+}(\xi)} t^{1-2s}\Big(
\frac{\partial \bar u_\varepsilon}{\partial \nu}\frac{\partial \bar v_\varepsilon^{(l)}}{\partial x_i}
+\frac{\partial \bar v_\varepsilon^{(l)}}{\partial \nu}\frac{\partial \bar u_\varepsilon}{\partial x_i}
\Big)\,d\sigma
+\int_{\partial'\mathcal{B}_{r}^{+}(\xi)} t^{1-2s}\langle\nabla \bar u_\varepsilon,\nabla \bar v_\varepsilon^{(l)}\rangle\,\nu_i\,d\sigma \nonumber\\
&\qquad
+\varepsilon^{-2s}\int_{\partial B_{r}(\xi)}\big(V(x)u_\varepsilon v_\varepsilon^{(l)}-u_\varepsilon^p v_\varepsilon^{(l)}\big)\nu_i\,dS\nonumber \\
&
=\varepsilon^{-2s}\int_{B_r(\xi)}\frac{\partial V(x)}{\partial x_i}\,u_\varepsilon v_\varepsilon^{(l)}\,dx
+\varepsilon^{-2s}\mu_\varepsilon^{(l)}\int_{B_r(\xi)} v_\varepsilon^{(l)}\frac{\partial u_\varepsilon}{\partial x_i}\,dx.
\end{align}
Here $dX=dx\,dt$, $d\sigma$ is the surface measure on $\partial'\mathcal{B}_{r}^{+}(\xi)$, $dS$ is the
surface measure on $\partial B_{r}(\xi)$, and $\nu=(\nu_1,\dots,\nu_{N+1})$ is the outer unit normal to $\partial'\mathcal{B}_{r}^{+}(\xi)$.
\end{lemma}

Before stating the next proposition, we introduce the $kN\times kN$ matrices that yield precise estimates for the eigenvalues. For \(j=1,\ldots,k\), recall
\[
C_j:=\frac{1}N\int_{\mathbb R^N}|\nabla w_{V(\xi_j^0)}|^2\,dy>0,
\]
and write
\[
M_j:=\int_{\mathbb R^N}w_{V(\xi_j^0)}^2\,dy>0 .
\]
Define
\[
\mathcal C
:=
\operatorname{diag}
\left(
C_1I_N,\ldots,C_kI_N
\right)
\in\mathbb R^{kN\times kN},
\]
and
\[
\mathcal H
:=
\operatorname{diag}
\left(
\frac{M_1}{2}D^2V(\xi_1^0),
\ldots,
\frac{M_k}{2}D^2V(\xi_k^0)
\right)
\in\mathbb R^{kN\times kN}.
\]
Equivalently, since $\mathcal C$ is positive definite and each block of
$\mathcal C$ is a scalar multiple of the identity, we may write
\[
\mathcal A
:=
\mathcal C^{-1}\mathcal H
=
\operatorname{diag}
\left(
\frac{M_1}{2C_1}D^2V(\xi_1^0),
\ldots,
\frac{M_k}{2C_k}D^2V(\xi_k^0)
\right).
\]
Thus the generalized eigenvalue problem
\[
\mathcal H a=\eta \mathcal C a
\]
is equivalent to the ordinary eigenvalue problem
\[
\mathcal A a=\eta a .
\]

We next derive the precise asymptotic behavior of the eigenvalues $\mu_\varepsilon^{(l)}$ with
$l\in\{k+1,\dots,(N+1)k\}$.

\begin{proposition}\label{prop:hessian-eigenvalue-asymptotics}

Let $u_\varepsilon$ be a $k$--peak family in the sense of Definition~\ref{def:kpeak-class}, and assume \eqref{eq:nondegenerate-critical-points}.  Then, for
$l=k+1,\ldots,(N+1)k$, 
\begin{equation}\label{eq:mu-matrix-asymptotics}
  \frac{\mu_\varepsilon^{(l)}}{\varepsilon^2}
a^{(l)}
=
\mathcal A a^{(l)}+o(1)
\quad
\text{in }\mathbb R^{kN},
\end{equation}
where $a^{(l)}=(a_1^{(l)},\ldots,a_k^{(l)})
\in\mathbb R^{kN}$, as in Lemma~\ref{lem:small-eigenfunction-orthogonality}.
Let $\eta_1\le\cdots\le\eta_{kN}$ denote the ordered eigenvalues of $\mathcal A$, counted with multiplicity. Then
\[
\frac{\mu_\varepsilon^{(k+\ell)}}{\varepsilon^2}=\eta_\ell+o(1),
\qquad \ell=1,\ldots,kN.
\]
Since $M_j/(2C_j)>0$, the signs of these $kN$ eigenvalues are exactly the signs of the eigenvalues of $D^2V(\xi_j^0)$, $j=1,\ldots,k$.

\end{proposition}

\begin{proof}
Fix $j\in\{1,\ldots,k\}$.
By Lemma~\ref{lem:extension-u-bound} and Proposition~\ref{prop:mu-small}, we obtain
\[
\int_{A_{r,2r}^{+,j}}t^{1-2s}
\left(|\nabla\bar u_\varepsilon|^2+
\sum_{\ell=k+1}^{(N+1)k}|\nabla\bar v_\varepsilon^{(\ell)}|^2\right)dX
\le C\varepsilon^{2N}.
\]
Hence, there exists
\[
\rho_{\varepsilon}\in(r,2r)
\]
such that
\begin{equation}\label{eq:good-radius-energy}
\int_{\partial'\mathcal B_{\rho_{\varepsilon}}^+(\xi_{\varepsilon,j})}
t^{1-2s}
\left(
|\nabla \bar u_\varepsilon|^2+
\sum_{\ell=k+1}^{k+kN}
|\nabla \bar v_\varepsilon^{(\ell)}|^2
\right)d\sigma
\le C\varepsilon^{2N},
\end{equation}
uniformly in $j$.

Apply the local Pohozaev identity \eqref{eq:local-pohozaev-identity} with $\xi=\xi_{\varepsilon,j}$ and $r=\rho_{\varepsilon}$, and multiply it by $\varepsilon^{2s}$.  We obtain
\begin{align}
&-\varepsilon^{2s}\int_{\partial'\mathcal B_{\rho_{\varepsilon}}^+(\xi_{\varepsilon,j})}
t^{1-2s}
\left(
\frac{\partial \bar u_\varepsilon}{\partial\nu}
\frac{\partial \bar v_\varepsilon^{(l)}}{\partial x_i}
+
\frac{\partial \bar v_\varepsilon^{(l)}}{\partial\nu}
\frac{\partial \bar u_\varepsilon}{\partial x_i}
\right)d\sigma 
\notag\\
&\quad+
\varepsilon^{2s}\int_{\partial'\mathcal B_{\rho_{\varepsilon}}^+(\xi_{\varepsilon,j})}
t^{1-2s}
\langle\nabla\bar u_\varepsilon,\nabla\bar v_\varepsilon^{(l)}\rangle\nu_i\,d\sigma
\notag\\
&\quad+
\int_{\partial B_{\rho_{\varepsilon}}(\xi_{\varepsilon,j})}
\bigl(V(x)u_\varepsilon v_\varepsilon^{(l)}-u_\varepsilon^p v_\varepsilon^{(l)}\bigr)\nu_i\,dS
\notag\\
&=
\int_{B_{\rho_{\varepsilon}}(\xi_{\varepsilon,j})}
\frac{\partial V}{\partial x_i}(x)u_\varepsilon v_\varepsilon^{(l)}\,dx
+\mu_\varepsilon^{(l)}
\int_{B_{\rho_{\varepsilon}}(\xi_{\varepsilon,j})}
v_\varepsilon^{(l)}\frac{\partial u_\varepsilon}{\partial x_i}\,dx.
\label{eq:pohozaev-good-radius}
\end{align}

By \eqref{eq:good-radius-energy} and Lemma~\ref{lem:u-polynomial-decay}, we have
\begin{equation}\label{eq:boundary-pohozaev-good-radius}
  \begin{aligned}
     &\left|\text{\color{blue}LHS of \eqref{eq:pohozaev-good-radius}}\right|\\
     \le&C\varepsilon^{2s}
\left(
\int_{\partial'\mathcal B_{\rho_{\varepsilon}}^+(\xi_{\varepsilon,j})}t^{1-2s}|\nabla\bar u_\varepsilon|^2d\sigma
\right)^{1/2}
\left(
\int_{\partial'\mathcal B_{\rho_{\varepsilon}}^+(\xi_{\varepsilon,j})}t^{1-2s}|\nabla\bar v_\varepsilon^{(l)}|^2d\sigma
\right)^{1/2}\\
&+O(\varepsilon^{2N+2s})\\
=&O(\varepsilon^{2N+2s}).
  \end{aligned}
\end{equation}

Combining \eqref{eq:pohozaev-good-radius} and \eqref{eq:boundary-pohozaev-good-radius}, we obtain
\begin{equation}\label{eq:pohozaev-balance}
\int_{B_{\rho_{\varepsilon}}(\xi_{\varepsilon,j})}
\frac{\partial V}{\partial x_i}(x)u_\varepsilon v_\varepsilon^{(l)}\,dx
+\mu_\varepsilon^{(l)}
\int_{B_{\rho_{\varepsilon}}(\xi_{\varepsilon,j})}
v_\varepsilon^{(l)}\frac{\partial u_\varepsilon}{\partial x_i}\,dx
=O(\varepsilon^{2N+2s})=o(\varepsilon^{N+1}).
\end{equation}

We first estimate
\[
I_{1,\varepsilon}
:=
\int_{B_{\rho_\varepsilon}(\xi_{\varepsilon,j})}
\partial_i V(x)u_\varepsilon v_\varepsilon^{(l)}\,dx .
\]
After the change of variables
\(
x=\xi_{\varepsilon,j}+\varepsilon y,
\) Lemmas~\ref{lem:true-peak-profile-convergence}, \ref{lem:pohozaev-center-estimate} and \ref{lem:small-eigenfunction-orthogonality}, together with a Taylor expansion of $\partial_i V$ about $\xi_{\varepsilon,j}$, give
 \begin{equation}\label{eq:potential-gradient-expansion}
  \begin{aligned}[b]
I_{1,\varepsilon}
&=
\varepsilon^N\partial_iV(\xi_{\varepsilon,j})
\int_{B_{\rho_\varepsilon /\varepsilon}(0)}
u_{\varepsilon,j}(y)v_{\varepsilon,j}^{(l)}(y)\,dy
\\
&\quad
+
\varepsilon^{N+1}
\sum_{m=1}^N
\partial_{im}^2V(\xi_j^0)
\int_{B_{\rho_\varepsilon /\varepsilon}(0)}
y_m u_{\varepsilon,j}(y)v_{\varepsilon,j}^{(l)}(y)\,dy
+
o(\varepsilon^{N+1})
\\
&=
\varepsilon^N\partial_iV(\xi_{\varepsilon,j})
\left(
\sum_{n=1}^N
a_{j,n}^{(l)}
\int_{\mathbb R^N}w_{V(\xi_j^0)}(y)\partial_n w_{V(\xi_j^0)}(y)\,dy
+
o(1)
\right)
\\
&\quad
+
\varepsilon^{N+1}
\sum_{m,n=1}^N
\partial_{im}^2V(\xi_j^0)a_{j,n}^{(l)}
\int_{\mathbb R^N}
y_m w_{V(\xi_j^0)}(y)\partial_n w_{V(\xi_j^0)}(y)\,dy
+
o(\varepsilon^{N+1})
\\
&=
-\frac{1}{2}\varepsilon^{N+1}
\sum_{m,n=1}^N
\partial_{im}^2V(\xi_j^0)a_{j,n}^{(l)}
\delta_{mn}
\int_{\mathbb R^N}w_{V(\xi_j^0)}^2(y)\,dy
+
o(\varepsilon^{N+1})
\\
&=
-\frac{1}{2}\varepsilon^{N+1}M_j
\sum_{m=1}^N
\partial_{im}^2V(\xi_j^0)a_{j,m}^{(l)}
+
o(\varepsilon^{N+1}).
\end{aligned}
\end{equation}

Next set
\[
I_{2,\varepsilon}:=
\int_{B_{\rho_{\varepsilon}}(\xi_{\varepsilon,j})}
v_\varepsilon^{(l)}\partial_i u_\varepsilon\,dx.
\]
Then Lemma~\ref{lem:true-peak-profile-convergence} and \ref{lem:small-eigenfunction-orthogonality} give
\begin{equation}\label{eq:derivative-overlap-expansion}
  \begin{aligned}[b]
    I_{2,\varepsilon}&=\varepsilon^{N-1}
\int_{B_{\rho_\varepsilon /\varepsilon}(0)}
v_{\varepsilon,j}^{(l)}(y)\partial_{y_i}u_{\varepsilon,j}(y)\,dy\\
&=\varepsilon^{N-1}
\sum_{n=1}^Na_{j,n}^{(l)}
\int_{\mathbb R^N}\partial_n w_{V(\xi_j^0)}\partial_i w_{V(\xi_j^0)}\,dy
+o(\varepsilon^{N-1})\\
&=\varepsilon^{N-1}C_ja_{j,i}^{(l)}+o(\varepsilon^{N-1}).
  \end{aligned}
\end{equation}

Substituting \eqref{eq:potential-gradient-expansion} and \eqref{eq:derivative-overlap-expansion} into \eqref{eq:pohozaev-balance}, we obtain
\begin{align}
&-\frac{1}{2}\varepsilon^{N+1}M_j
\sum_{m=1}^N\partial_{im}^2V(\xi_j^0)a_{j,m}^{(l)}
+\mu_\varepsilon^{(l)}\varepsilon^{N-1}C_ja_{j,i}^{(l)}
\notag\\
&\quad+
\mu_\varepsilon^{(l)}o(\varepsilon^{N-1})+o(\varepsilon^{N+1})
=o(\varepsilon^{N+1}).
\label{eq:mu-system-before-division}
\end{align}
By Lemma~\ref{lem:small-eigenfunction-orthogonality}, there exists at least one component $a_{j_0,i_0}^{(l)}\ne0$. Applying \eqref{eq:mu-system-before-division} to that component gives
\[
\mu_\varepsilon^{(l)}=O(\varepsilon^2).
\]
Dividing \eqref{eq:mu-system-before-division} by $\varepsilon^{N+1}$ yields, for every $j=1,\ldots,k$ and $i=1,\ldots,N$,
\begin{equation}\label{eq:hessian-eigenvalue-system}
  \frac{\mu_\varepsilon^{(l)}}{\varepsilon^2}C_j a_{j,i}^{(l)}
=
\frac{1}{2}M_j\sum_{m=1}^N\partial_{im}^2V(\xi_j^0)a_{j,m}^{(l)}+o(1).
\end{equation}

In block form, \eqref{eq:hessian-eigenvalue-system} is exactly
\[
\frac{\mu_\varepsilon^{(l)}}{\varepsilon^2}
\mathcal C a^{(l)}
=
\mathcal H a^{(l)}
+
o(1).
\]
Multiplying by $\mathcal C^{-1}$ gives
\[
\frac{\mu_\varepsilon^{(l)}}{\varepsilon^2}
a^{(l)}
=
\mathcal C^{-1}\mathcal H a^{(l)}
+
o(1)
=
\mathcal A a^{(l)}
+
o(1).
\]
This proves \eqref{eq:mu-matrix-asymptotics}. 

It remains to identify the limiting finite-dimensional spectrum. By Lemma~\ref{lem:small-eigenfunction-orthogonality}, the $kN$ nonzero vectors
\[
a^{(k+1)},\ldots,a^{(k+kN)}
\]
are mutually orthogonal for the positive definite inner product $\langle a,b\rangle_{\mathcal C}=a^T\mathcal Cb$ and hence form a basis of $\mathbb R^{kN}$.  Passing to the limit in \eqref{eq:mu-matrix-asymptotics}, the corresponding limits exhaust the spectrum of $\mathcal A$, counted with multiplicity.

Finally,
\[
\mathcal A=\operatorname{diag}\left(
\frac{M_1}{2C_1}D^2V(\xi_1^0),\ldots,
\frac{M_k}{2C_k}D^2V(\xi_k^0)\right),
\]
and all factors $M_j/(2C_j)$ are positive.  This proves the sign assertion and completes the proof.
\end{proof}

The following proposition gives the spectral gap after the first $k+kN$ eigenvalues.

\begin{proposition}\label{prop:kn-plus-one}
Let $u_{\varepsilon}$ be a $k$--peak family in the sense of Definition~\ref{def:kpeak-class} and assume \eqref{eq:nondegenerate-critical-points}. Then there exist $C>0$ and $\varepsilon_0>0$ such that
\[
\mu_{\varepsilon}^{(k+kN+1)} \ge C \qquad \text{for every }\varepsilon\in(0,\varepsilon_0].
\]
\end{proposition}

\begin{proof}
By Proposition~\ref{prop:mu-small},
\[
\mu_\varepsilon^{(k+kN)}\to0,
\]
and the ordering of the eigenvalues gives
\[
\mu_{\varepsilon}^{(k+kN+1)}\ge \mu_{\varepsilon}^{(k+kN)}.
\]
Hence
\[
\liminf_{\varepsilon\to0}\mu_{\varepsilon}^{(k+kN+1)}\ge0.
\]
It remains to prove that this lower limit is strictly positive. Suppose for contradiction that no positive lower bound exists. Then, along a subsequence still denoted by \(\varepsilon\),
 \begin{equation}\label{eq:gap-contr-zero}
\mu_\varepsilon^{(k+kN+1)}\to0.
\end{equation}
Let $v_\varepsilon^{(k+kN+1)}$ be an associated eigenfunction, chosen $L^2$-orthogonal to the previous eigenfunctions. The proof of Lemma~\ref{lem:small-eigenfunction-orthogonality} applies to this eigenfunction.  Thus, after passing to a subsequence, for every $j=1,\ldots,k$,
\[
v_{\varepsilon,j}^{(k+kN+1)}(x):=v_\varepsilon^{(k+kN+1)}(\varepsilon x+\xi_{\varepsilon,j})
\longrightarrow
v_j^{(*)}(x)
=\sum_{i=1}^N a_{j,i}^{(*)}\partial_iw_{V(\xi_j^0)}(x)
\quad\hbox{in }C^1_{\rm loc}(\mathbb R^N),
\]
and the full vector
\[
a^{(*)}:=\bigl(a_{1,1}^{(*)},\ldots,a_{1,N}^{(*)},\ldots,a_{k,1}^{(*)},\ldots,a_{k,N}^{(*)}\bigr)
\in\mathbb R^{kN}
\]
is nonzero. 
Moreover, since the eigenfunctions are chosen $L^2$-orthogonal, for every
$l=k+1,\ldots,k+kN$ we have
\[
\int_{\mathbb R^N}v_\varepsilon^{(k+kN+1)}v_\varepsilon^{(l)}dx=0.
\]
Repeating the orthogonality argument in Lemma~\ref{lem:small-eigenfunction-orthogonality} yields
\begin{equation}\label{eq:new-vector-orthogonal}
\sum_{j=1}^k C_j a_j^{(*)}\cdot a_j^{(l)}=0,
\qquad l=k+1,\ldots,k+kN.
\end{equation}
On the other hand, 
\[
\sum_{j=1}^k C_j a_j^{(l)}\cdot a_j^{(l')}=0,
\qquad k+1\le l\ne l'\le k+kN,
\]
and each $a^{(l)}$ is nonzero.  Thus the $kN$ vectors
\[
a^{(k+1)},\ldots,a^{(k+kN)}
\]
are nonzero and mutually orthogonal with respect to the positive definite inner product
\[
\langle a,b\rangle_*:=\sum_{j=1}^k C_j a_j\cdot b_j
\quad \hbox{on }\mathbb R^{kN}.
\]
Consequently they form a basis of $\mathbb R^{kN}$.  But \eqref{eq:new-vector-orthogonal} says that the nonzero vector $a^{(*)}$ is orthogonal to this basis with respect to the same positive definite inner product, which is impossible.

Therefore the case $\liminf_{\varepsilon\to0}\mu_{\varepsilon}^{(k+kN+1)}=0$
is impossible.  Since
$\liminf_{\varepsilon\to0}\mu_{\varepsilon}^{(k+kN+1)}\ge0$, we obtain
\[
\liminf_{\varepsilon\to0}\mu_{\varepsilon}^{(k+kN+1)}>0.
\]
Thus there exist $C>0$ and $\varepsilon_0>0$ such that
\[
\mu_{\varepsilon}^{(k+kN+1)}\ge C
\qquad\hbox{for all }\varepsilon\in(0,\varepsilon_0].
\]

\end{proof}

\section{Morse index and nondegeneracy}\label{sec:morse-computation}

We now combine the spectral estimates obtained in Section~\ref{sec:eig-estimates} to compute the Morse index and prove nondegeneracy.

\begin{proof}[Proof of Theorem~\ref{th:morse-index}]
By Proposition~\ref{prop:mu-zero}, the first $k$ eigenvalues of the linearized operator are negative:
\[
\mu_\varepsilon^{(1)},\ldots,\mu_\varepsilon^{(k)}<0.
\]
By Proposition~\ref{prop:hessian-eigenvalue-asymptotics}, the next $kN$ eigenvalues satisfy
\[
\frac{\mu_\varepsilon^{(k+\ell)}}{\varepsilon^2}=\eta_\ell+o(1),
\qquad \ell=1,\ldots,kN,
\]
where the signs of $\eta_1,\ldots,\eta_{kN}$ coincide with the signs of the eigenvalues of the Hessian matrices
$D^2V(\xi_j^0)$, $j=1,\ldots,k$.  Finally, Proposition~\ref{prop:kn-plus-one} gives
\[
\mu_\varepsilon^{(k+kN+1)}\ge C>0.
\]
Hence, by monotonicity of the ordered spectrum, all further eigenvalues are positive.  Therefore the number of negative eigenvalues is exactly
\[
k+\#\{\text{negative eigenvalues of }D^2V(\xi_j^0),\ j=1,\ldots,k\},
\]
which is the desired formula.
\end{proof}

\begin{proof}[Proof of Corollary~\ref{th:nondegeneracy}]
  The spectral analysis in Propositions~\ref{prop:mu-zero}, \ref{prop:hessian-eigenvalue-asymptotics} and \ref{prop:kn-plus-one} directly implies that 
\[
\ker L_\varepsilon=\{0\}.
\]
\end{proof}

\section{Parametrization}\label{sec:improved-expansion}
In this section, we establish a quantitative modulation decomposition for every sharply quantized \(k\)-peak family in the sense of Definition~\ref{def:kpeak-class}. Recall that, for \(\xi=(\xi_1,\ldots,\xi_k)\) and
\(\alpha=(\alpha_1,\ldots,\alpha_k)\), we have 
\[
W_{\varepsilon,j}(x;\xi_j):=w_{V(\xi_j)}\!\left(\frac{x-\xi_j}{\varepsilon}\right),
\qquad
W_{\varepsilon,\alpha,\xi}(x):=
\sum_{j=1}^k\alpha_j W_{\varepsilon,j}(x;\xi_j).
\]
The reference profile attached to the peak centers is
\[
W_{\varepsilon,j}^*(x):=W_{\varepsilon,j}(x;\xi_{\varepsilon,j}),
\qquad
W_\varepsilon^*(x):=W_{\varepsilon,\one,\xi_\varepsilon}(x)=
\sum_{j=1}^k W_{\varepsilon,j}^*(x),
\]
where \(\xi_\varepsilon=(\xi_{\varepsilon,1},\ldots,\xi_{\varepsilon,k})\).

We work with exponents
\[
0<\sigma<\tau<s.
\]
Thus \(\beta_\tau=N+2s-\tau\) and \(N<\beta_\tau<\beta_\sigma\).

\subsection{Expansion at the peak centers}\label{subsec:global-weighted-expansion}

Set
\[
\Phi_\varepsilon(y):=\widetilde u_\varepsilon(y)-W_\varepsilon^*(\varepsilon y),
\qquad
A_\varepsilon:=\|\Phi_\varepsilon\|_{\tau,q_\varepsilon}.
\]
For \(j=1,\ldots,k\) and \(h=1,\ldots,N\), define
\(
Z_{j,h}(y):=\partial_h w_{V(\xi_{\varepsilon,j})}(y-q_{\varepsilon,j}).
\)
The linearized operator around the reference profile in the rescaled variables is
\[
\mathcal L_\varepsilon^*
:=(-\Delta)^s+V(\varepsilon y)-p\bigl(W_\varepsilon^*(\varepsilon y)\bigr)^{p-1}.
\]
The rescaled centers satisfy
\[
|q_{\varepsilon,i}-q_{\varepsilon,j}|\ge \frac{c}{\varepsilon},\qquad i\ne j,
\qquad
\sum_{j=1}^k|q_{\varepsilon,j}|\le C\varepsilon^{-1},
\]
where \(q_{\varepsilon,j}=\frac{\xi_{\varepsilon,j}}{\varepsilon}\).
We first recall the projected inverse estimate for \(\mathcal L_\varepsilon^*\).

\begin{lemma}\label{lem:projected-linear-inverse}
If \(\psi_\varepsilon\) satisfies
\[
\mathcal L_\varepsilon^*\psi_\varepsilon
=g_\varepsilon+
\sum_{j=1}^k\sum_{h=1}^N c_{j,h}Z_{j,h}
\quad\hbox{in }\mathbb R^N,
\qquad
\int_{\mathbb R^N}\psi_\varepsilon Z_{j,h}=0,
\]
then, for sufficiently small \(\varepsilon\),
\begin{equation}\label{eq:projected-linear-inverse}
\|\psi_\varepsilon\|_{\tau,q_\varepsilon}
\le C\|g_\varepsilon\|_{\tau,q_\varepsilon}.
\end{equation}
\end{lemma}

\begin{proof}
This is a priori estimate for the projected operator in \cite[Proposition~4.1 and Lemma~4.2]{MR3121716}.
\end{proof}

\begin{proposition}\label{prop:local-expansion}
There exists \(C>0\) such that
\begin{equation}\label{eq:global-projected-estimate}
\|\widetilde u_\varepsilon(y)-W_\varepsilon^*(\varepsilon y)\|_{\tau,q_\varepsilon}
\le C\varepsilon^\tau .
\end{equation}
\end{proposition}

\begin{proof}
We divide the proof into several steps.

\smallskip
\noindent\emph{Step 1: preliminary convergence.}
We first prove
\begin{equation}\label{eq:global-rough-convergence-new}
A_\varepsilon=o(1).
\end{equation}
For fixed \(R>1\), set
\[
O_{\varepsilon,R}:=\mathbb R^N\setminus\bigcup_{j=1}^kB_R(q_{\varepsilon,j}).
\]
Since \(\tau>\sigma\), by Lemma~\ref{lem:single-bubble-input} and Lemma~\ref{lem:u-polynomial-decay},
\[
\sup_{y\in O_{\varepsilon,R}}\rho_{\tau,q_\varepsilon}(y)^{-1}|\Phi_\varepsilon(y)|
\le CR^{\sigma-\tau}+CR^{-\tau}.
\]
On \(B_R(q_{\varepsilon,j})\), Lemma~\ref{lem:single-bubble-input} and Lemma~\ref{lem:true-peak-profile-convergence} give
\[
\sup_{y\in\cup_jB_R(q_{\varepsilon,j})}
\rho_{\tau,q_\varepsilon}(y)^{-1}|\Phi_\varepsilon(y)|=o(1).
\]
Thus
\[
A_\varepsilon\le o(1)+CR^{\sigma-\tau}+CR^{-\tau}.
\]
Letting first \(\varepsilon\to0\) and then \(R\to\infty\) proves \eqref{eq:global-rough-convergence-new}.

\smallskip
\noindent\emph{Step 2: finite-dimensional decomposition.}
The standard Gram estimate for the family \(\{Z_{j,h}\}\) in \cite{MR3121716} gives
\[
\int_{\mathbb R^N}Z_{j,h}Z_{i,m}
=G_{\varepsilon,j}\delta_{ij}\delta_{hm}+O\left(\min_{\ell\ne r}|q_{\varepsilon,\ell}-q_{\varepsilon,r}|^{-N}\right),
\qquad
G_{\varepsilon,j}:=\int_{\mathbb R^N}(\partial_h w_{V(\xi_{\varepsilon,j})})^2>0.
\]
Since \(\min_{i\ne j}|q_{\varepsilon,i}-q_{\varepsilon,j}|\ge c\varepsilon^{-1}\), this Gram matrix is uniformly invertible.  Therefore there exists a unique decomposition
\begin{equation}\label{eq:global-orthogonal-decomposition-new}
  \Phi_\varepsilon
=
\Phi_\varepsilon^\perp
+
\sum_{j=1}^k\sum_{h=1}^N a_{j,h}Z_{j,h},
\qquad
\int_{\mathbb R^N}\Phi_\varepsilon^\perp Z_{j,h}=0.
\end{equation}
Taking the \(L^2\)-inner product with \(Z_{i,m}\), we find that the coefficients satisfy
\[
\sum_{j,h}a_{j,h}\int_{\mathbb R^N}Z_{j,h}Z_{i,m}
=
\int_{\mathbb R^N}\Phi_\varepsilon Z_{i,m}.
\]
Using the uniform invertibility of the Gram matrix and the weighted bound for
\(\Phi_\varepsilon\), we obtain
\[
|a|\le C A_\varepsilon.
\]
Moreover, by Lemma~\ref{lem:single-bubble-input}, we obtain
\[
\|\rho_{\tau,q_\varepsilon}^{-1}Z_{j,h}\|_{L^\infty}\le C.
\]
Then, the decomposition gives
\begin{equation}\label{eq:gram-decomposition-bounds-new}
  A_\varepsilon
\le
C\|\Phi_\varepsilon^\perp\|_{\tau,q_\varepsilon}
+
C|a|.
\end{equation}

\smallskip

\noindent\emph{Step 3: residual estimates.}
Set
\[
E_\varepsilon:=
\sum_{j=1}^k(V(\xi_{\varepsilon,j})-V(\varepsilon y))W_{\varepsilon,j}^*(\varepsilon y)
+
\bigl(W_\varepsilon^*(\varepsilon y)\bigr)^p
-
\sum_{j=1}^k\bigl(W_{\varepsilon,j}^*(\varepsilon y)\bigr)^p
\]
and
\[
N_\varepsilon(\Phi):=
(W_\varepsilon^*(\varepsilon y)+\Phi)^p-
\bigl(W_\varepsilon^*(\varepsilon y)\bigr)^p
-p\bigl(W_\varepsilon^*(\varepsilon y)\bigr)^{p-1}\Phi.
\]
Thus
\begin{equation}\label{eq:phi-global-equation-new}
\mathcal L_\varepsilon^*\Phi_\varepsilon
=E_\varepsilon+N_\varepsilon(\Phi_\varepsilon).
\end{equation}
Choose \(\gamma\in(\max\{0,1-2s\},1)\) such that
\[
\gamma+2s\notin\mathbb N.
\]
Then \(\gamma+2s>1\).  We shall use the following estimates.  With
\(\kappa:=\min\{2,p\}>1\),
\begin{equation}\label{eq:residual-basic-claim-1}
\|E_\varepsilon\|_{\tau,q_\varepsilon}\le C\varepsilon^\tau,
\qquad
\max_j\|E_\varepsilon\|_{C^\gamma(B_2(q_{\varepsilon,j}))}\le C\varepsilon^\tau,
\end{equation}
\begin{equation}\label{eq:residual-basic-claim-2}
\max_{j,h}\|\mathcal L_\varepsilon^*Z_{j,h}\|_{\tau,q_\varepsilon}
+
\max_{\ell,j,h}\|\mathcal L_\varepsilon^*Z_{j,h}\|_{C^\gamma(B_2(q_{\varepsilon,\ell}))}
\le C\varepsilon^\tau,
\end{equation}
and
\begin{equation}\label{eq:global-nonlinear-estimate-new}
\|N_\varepsilon(\Phi_\varepsilon)\|_{\tau,q_\varepsilon}\le CA_\varepsilon^\kappa.
\end{equation}
Consequently, if
\begin{equation}\label{eq:H-definition-new}
H_\varepsilon:=E_\varepsilon+N_\varepsilon(\Phi_\varepsilon)
-
\sum_{j=1}^k\sum_{h=1}^Na_{j,h}\mathcal L_\varepsilon^*Z_{j,h},
\end{equation}
then
\begin{equation}\label{eq:source-global-bound-new}
\|H_\varepsilon\|_{\tau,q_\varepsilon}
\le C\varepsilon^\tau+CA_\varepsilon^\kappa+C\varepsilon^\tau|a|,
\end{equation}
and
\begin{equation}\label{eq:source-local-bound-new}
\max_j\|H_\varepsilon\|_{C^\gamma(B_2(q_{\varepsilon,j}))}
\le C\varepsilon^\tau+CA_\varepsilon^\kappa+C\varepsilon^\tau|a|.
\end{equation}

We first prove \eqref{eq:residual-basic-claim-1}.  The weighted estimate follows directly from
Lemma~\ref{lem:weighted-tail}, applied with \(\alpha=\one\), \(\xi=\xi_\varepsilon\), and
\(x=\varepsilon y\).  It remains to prove the local \(C^\gamma\) estimate for
\(E_\varepsilon\). 

Fix \(j\in\{1,\ldots,k\}\). Using Lemma~\ref{lem:pohozaev-center-estimate}, we obtain
\begin{equation}\label{eq:E-main-potential-local-new}
\bigl\|(
V(\xi_{\varepsilon,j})-V(\varepsilon y))W_{\varepsilon,j}^*(\varepsilon y)
\bigr\|_{C^\gamma(B_3(q_{\varepsilon,j}))}
\le C\varepsilon^2.
\end{equation}

For \(i\ne j\), Lemma~\ref{lem:single-bubble-input} gives on \(B_3(q_{\varepsilon,j})\)
\begin{equation}\label{eq:E-tail-potential-local-new}
\sum_{i\ne j}
\bigl\|(
V(\xi_{\varepsilon,i})-V(\varepsilon y))W_{\varepsilon,i}^*(\varepsilon y)
\bigr\|_{C^\gamma(B_3(q_{\varepsilon,j}))}
\le C\varepsilon^{N+2s}.
\end{equation}

We now estimate the interaction term.  On \(B_3(q_{\varepsilon,j})\),
\[
W_{\varepsilon,j}^*(\varepsilon y)\ge c_0>0,
\qquad
\|W_{\varepsilon,j}^*(\varepsilon y)\|_{C^1(B_3(q_{\varepsilon,j}))}\le C,
\]
and by Lemma~\ref{lem:single-bubble-input}, we obtain
\[
\left\|\sum_{i\ne j}W_{\varepsilon,i}^*(\varepsilon y)\right\|_{C^\gamma(B_3(q_{\varepsilon,j}))}
+
\sum_{i\ne j}
\left\|\bigl(W_{\varepsilon,i}^*(\varepsilon y)\bigr)^p\right\|_{C^\gamma(B_3(q_{\varepsilon,j}))}
\le C\varepsilon^{N+2s}.
\]
Then,
\[
\begin{aligned}
&\bigl(W_\varepsilon^*(\varepsilon y)\bigr)^p
-
\sum_{i=1}^k\bigl(W_{\varepsilon,i}^*(\varepsilon y)\bigr)^p \\
&\quad=
\left(\int_0^1
p\left(W_{\varepsilon,j}^*(\varepsilon y)
+t\sum_{i\ne j}W_{\varepsilon,i}^*(\varepsilon y)\right)^{p-1}dt\right)
\sum_{i\ne j}W_{\varepsilon,i}^*(\varepsilon y)
-
\sum_{i\ne j}\bigl(W_{\varepsilon,i}^*(\varepsilon y)\bigr)^p,
\end{aligned}
\]
hence we obtain
\begin{equation}\label{eq:E-interaction-local-new}
\left\|
\bigl(W_\varepsilon^*(\varepsilon y)\bigr)^p
-
\sum_{i=1}^k\bigl(W_{\varepsilon,i}^*(\varepsilon y)\bigr)^p
\right\|_{C^\gamma(B_3(q_{\varepsilon,j}))}
\le C\varepsilon^{N+2s}.
\end{equation}
Combining \eqref{eq:E-main-potential-local-new}, \eqref{eq:E-tail-potential-local-new}, and
\eqref{eq:E-interaction-local-new}, we have
\[
\|E_\varepsilon\|_{C^\gamma(B_3(q_{\varepsilon,j}))}
\le C\varepsilon^2+C\varepsilon^{N+2s}
\le C\varepsilon^\tau,
\]
because \(\tau<s<1\).  This proves \eqref{eq:residual-basic-claim-1}.

We next prove \eqref{eq:residual-basic-claim-2}. We have
\[
\mathcal L_\varepsilon^*Z_{j,h}
=
\bigl(V(\varepsilon y)-V(\xi_{\varepsilon,j})\bigr)Z_{j,h}
+
p\left[\bigl(W_{\varepsilon,j}^*(\varepsilon y)\bigr)^{p-1}
-
\bigl(W_\varepsilon^*(\varepsilon y)\bigr)^{p-1}\right]Z_{j,h}.
\]
By the same argument as in the proof of \eqref{eq:residual-basic-claim-1}, we obtain
\eqref{eq:residual-basic-claim-2}.

For the nonlinear remainder, if \(1<p<2\), then
\[
|N_\varepsilon(\Phi_\varepsilon)|\le C|\Phi_\varepsilon|^p,
\]
whereas if \(p\ge2\), then
\[
|N_\varepsilon(\Phi_\varepsilon)|
\le C\bigl(W_\varepsilon^*(\varepsilon y)\bigr)^{p-2}\Phi_\varepsilon^2+C|\Phi_\varepsilon|^p.
\]
Using \(|\Phi_\varepsilon|\le A_\varepsilon\rho_{\tau,q_\varepsilon}\) gives
\eqref{eq:global-nonlinear-estimate-new}.

It remains to obtain the local H\"older estimate needed for \(H_\varepsilon\). Notice that \(\Phi_\varepsilon\) satisfies
\[
(-\Delta)^s\Phi_\varepsilon
=E_\varepsilon+N_\varepsilon(\Phi_\varepsilon)
-\bigl(V(\varepsilon y)-p(W_\varepsilon^*(\varepsilon y))^{p-1}\bigr)\Phi_\varepsilon.
\]
By \eqref{eq:residual-basic-claim-1}, \eqref{eq:global-nonlinear-estimate-new}, and the
interior estimates of \cite[Theorem~1.1 and Corollary~3.5]{RosOtonSerra2016Stable}, a
finite local Schauder bootstrap yields
\[
\|\Phi_\varepsilon\|_{C^\gamma(B_2(q_{\varepsilon,j}))}
\le C(A_\varepsilon+\varepsilon^\tau).
\]
Thus, we obtain
\begin{equation}\label{eq:nonlinear-local-estimate}
  \max_j\|N_\varepsilon(\Phi_\varepsilon)\|_{C^\gamma(B_2(q_{\varepsilon,j}))}
\le CA_\varepsilon^\kappa+C\varepsilon^\tau.
\end{equation}
Combining this estimate with \eqref{eq:residual-basic-claim-1},
\eqref{eq:residual-basic-claim-2}, and \eqref{eq:nonlinear-local-estimate} gives
\eqref{eq:source-global-bound-new} and \eqref{eq:source-local-bound-new}.

\smallskip
\noindent\emph{Step 4: projected estimate.}
From \eqref{eq:phi-global-equation-new}, \eqref{eq:global-orthogonal-decomposition-new}, and \eqref{eq:H-definition-new},
\[
\mathcal L_\varepsilon^*\Phi_\varepsilon^\perp=H_\varepsilon,
\qquad
\int_{\mathbb R^N}\Phi_\varepsilon^\perp Z_{j,h}=0.
\]
Lemma~\ref{lem:projected-linear-inverse} and \eqref{eq:source-global-bound-new} give
\begin{equation}\label{eq:perp-global-estimate-new}
\|\Phi_\varepsilon^\perp\|_{\tau,q_\varepsilon}
\le C\varepsilon^\tau+CA_\varepsilon^\kappa+C\varepsilon^\tau|a|.
\end{equation}

We next estimate \(a\).  Since \(\xi_{\varepsilon,j}\) is the local maximum point,
\[
\nabla_y\widetilde u_\varepsilon(q_{\varepsilon,j})=0.
\]
Moreover,
\[
\nabla_y W_\varepsilon^*(\varepsilon y)\big|_{y=q_{\varepsilon,j}}
=
\sum_{i\ne j}\nabla w_{V(\xi_{\varepsilon,i})}(q_{\varepsilon,j}-q_{\varepsilon,i})
=O(\varepsilon^{N+2s}),
\]
and hence
\begin{equation}\label{eq:gradient-phi-at-center-new}
\nabla\Phi_\varepsilon(q_{\varepsilon,j})=O(\varepsilon^{N+2s}).
\end{equation}
Differentiating \eqref{eq:global-orthogonal-decomposition-new} at \(q_{\varepsilon,j}\), we obtain
\[
\nabla\Phi_\varepsilon(q_{\varepsilon,j})
=\nabla\Phi_\varepsilon^\perp(q_{\varepsilon,j})
+\sum_{h=1}^Na_{j,h}\nabla Z_{j,h}(q_{\varepsilon,j})
+
\sum_{i\ne j}\sum_{h=1}^Na_{i,h}\nabla Z_{i,h}(q_{\varepsilon,j}).
\]
The principal term is
\[
\sum_{h=1}^Na_{j,h}\nabla Z_{j,h}(q_{\varepsilon,j})
=D^2w_{V(\xi_{\varepsilon,j})}(0)a_j
=-\kappa_{V(\xi_{\varepsilon,j})}a_j,
\qquad
\kappa_{V(\xi_{\varepsilon,j})}\ge c>0,
\]
while by Lemma~\ref{lem:single-bubble-input}, the remaining peak terms are \(O(\varepsilon^{N+2s}|a|)\). Since \(N+2s>\tau\), \eqref{eq:gradient-phi-at-center-new} gives
\begin{equation}\label{eq:a-by-gradient-new}
|a_j|
\le C|\nabla\Phi_\varepsilon^\perp(q_{\varepsilon,j})|+C\varepsilon^\tau+C\varepsilon^\tau|a|.
\end{equation}
On \(B_2(q_{\varepsilon,j})\), \(\Phi_\varepsilon^\perp\) solves
\[
(-\Delta)^s\Phi_\varepsilon^\perp+\left(V(\varepsilon y)-p(W_\varepsilon^*(\varepsilon y))^{p-1}\right)\Phi_\varepsilon^\perp=H_\varepsilon.
\]
Since \(\gamma+2s>1\), the local Schauder estimate in \cite[Theorem~1.1 and Corollary~3.5]{RosOtonSerra2016Stable}  yields
\[
|\nabla\Phi_\varepsilon^\perp(q_{\varepsilon,j})|
\le C\|\Phi_\varepsilon^\perp\|_{L^\infty(\mathbb R^N)}
+C\|H_\varepsilon\|_{C^\gamma(B_2(q_{\varepsilon,j}))}.
\]
Since \(\|\Phi_\varepsilon^\perp\|_{L^\infty}\le C\|\Phi_\varepsilon^\perp\|_{\tau,q_\varepsilon}\), \eqref{eq:perp-global-estimate-new} and \eqref{eq:source-local-bound-new} imply
\[
|\nabla\Phi_\varepsilon^\perp(q_{\varepsilon,j})|
\le C\varepsilon^\tau+CA_\varepsilon^\kappa+C\varepsilon^\tau|a|.
\]
Substitution into \eqref{eq:a-by-gradient-new} gives
\[
|a|\le C\varepsilon^\tau+CA_\varepsilon^\kappa+C\varepsilon^\tau|a|.
\]
Absorbing the last term, we obtain
\begin{equation}\label{eq:a-final-estimate-new}
|a|\le C\varepsilon^\tau+CA_\varepsilon^\kappa.
\end{equation}
Combining \eqref{eq:perp-global-estimate-new} and \eqref{eq:a-final-estimate-new} also gives
\[
\|\Phi_\varepsilon^\perp\|_{\tau,q_\varepsilon}
\le C\varepsilon^\tau+CA_\varepsilon^\kappa.
\]

\smallskip
\noindent\emph{Step 5: absorption.}
Using \eqref{eq:gram-decomposition-bounds-new},
\[
A_\varepsilon
\le C\|\Phi_\varepsilon^\perp\|_{\tau,q_\varepsilon}+C|a|
\le C\varepsilon^\tau+CA_\varepsilon^\kappa.
\]
By \eqref{eq:global-rough-convergence-new}, \(A_\varepsilon=o(1)\).  Since \(\kappa>1\),
\[
CA_\varepsilon^\kappa=CA_\varepsilon^{\kappa-1}A_\varepsilon\le \frac12A_\varepsilon
\]
for sufficiently small \(\varepsilon\).  Therefore \(A_\varepsilon\le C\varepsilon^\tau\), which is exactly \eqref{eq:global-projected-estimate}.
\end{proof}

We next give the corresponding energy estimate.

\begin{lemma}\label{lem:energy-approx}
There exists \(C>0\) such that
\begin{equation}\label{eq:energy-approx-estimate}
\|u_\varepsilon-W_\varepsilon^*\|_{\varepsilon,V}
\le C\varepsilon^{N/2+\tau}.
\end{equation}
\end{lemma}

\begin{proof}
Set \(R_\varepsilon:=u_\varepsilon-W_\varepsilon^*\). We obtain 
\[
\varepsilon^{2s}(-\Delta)^sR_\varepsilon+V(x)R_\varepsilon=F_\varepsilon
\quad\hbox{in }\mathbb R^N,
\]
where
\[
F_\varepsilon:=u_\varepsilon^p-
\sum_{j=1}^k(W_{\varepsilon,j}^*)^p
-
\sum_{j=1}^k(V(x)-V(\xi_{\varepsilon,j}))W_{\varepsilon,j}^*.
\]
Testing the equation with \(R_\varepsilon\) gives
\[
\|R_\varepsilon\|_{\varepsilon,V}^2
=\int_{\mathbb R^N}F_\varepsilon R_\varepsilon
\le \|F_\varepsilon\|_{L^1(\mathbb R^N)}\|R_\varepsilon\|_{L^\infty(\mathbb R^N)}.
\]
By the computation in Proposition~\ref{prop:local-expansion}, we have
\[
\|F_\varepsilon\|_{L^1(\mathbb R^N)}
\le C\varepsilon^{N+\tau}.
\]
Proposition~\ref{prop:local-expansion} also gives \(\|R_\varepsilon\|_{L^\infty}\le C\varepsilon^\tau\).  Therefore
\[
\|R_\varepsilon\|_{\varepsilon,V}^2
\le C\varepsilon^{N+2\tau},
\]
which proves \eqref{eq:energy-approx-estimate}.
\end{proof}

\subsection{Modulation by minimization}\label{subsec:modulation-minimization}

Recall that, for \(\xi=(\xi_1,\ldots,\xi_k)\), we have 
\[
Y_{\varepsilon,j,h}(x;\xi_j):=\partial_{\xi_{j,h}}W_{\varepsilon,j}(x;\xi_j).
\qquad h=1,\ldots,N.
\]
Explicitly, taking \(y=(x-\xi_j)/\varepsilon\), we obtain
\begin{equation}\label{eq:tangent-vector-explicit}
Y_{\varepsilon,j,h}(x;\xi_j)
=-\varepsilon^{-1}\partial_hw_{V(\xi_j)}(y)+
\left.\partial_\lambda w_\lambda(y)\right|_{\lambda=V(\xi_j)}\partial_hV(\xi_j).
\end{equation}
And the orthogonality space is
\[
\begin{aligned}
E_{\varepsilon,\xi}:=
\{\omega\in H^s(\mathbb R^N):&
\ \langle\omega,W_{\varepsilon,j}(x;\xi_j)\rangle_{\varepsilon,V}=0,\\
&\ \langle\omega,Y_{\varepsilon,j,h}(x;\xi_j)\rangle_{\varepsilon,V}=0,
\quad j=1,\ldots,k,\ h=1,\ldots,N\}.
\end{aligned}
\]

\begin{proposition}\label{prop:minimization}
For all sufficiently small \(\varepsilon\), there exist
\(\alpha_\varepsilon\in\mathbb R^k\), \(\widehat\xi_\varepsilon\in(\mathbb R^N)^k\), and
\(\omega_\varepsilon\in H^s(\mathbb R^N)\) such that
\[
u_\varepsilon=W_{\varepsilon,\alpha_\varepsilon,\widehat\xi_\varepsilon}+\omega_\varepsilon,
\qquad
\omega_\varepsilon\in E_{\varepsilon,\widehat\xi_\varepsilon},
\]
and
\begin{equation}\label{eq:min-realization-estimate}
|\alpha_\varepsilon-\one|+
\frac{|\widehat\xi_\varepsilon-\xi_\varepsilon|}{\varepsilon}+
\varepsilon^{-N/2}\|\omega_\varepsilon\|_{\varepsilon,V}
\le C\varepsilon^\tau,
\end{equation}
\begin{equation}\label{eq:min-realization-weighted-estimate}
\|\omega_\varepsilon\|_{\tau,\varepsilon,\widehat\xi_\varepsilon}
\le C\varepsilon^\tau.
\end{equation}
\end{proposition}

\begin{proof}
Set
\[
K_\varepsilon:=\{(\alpha,\xi):
|\alpha_j-1|\le\varepsilon^{\tau/2},
\ |\xi_j-\xi_{\varepsilon,j}|\le\varepsilon^{1+\tau/2},
\ j=1,\ldots,k\}.
\]
The map
\[
(\alpha,\xi)\longmapsto \norm{u_\varepsilon-W_{\varepsilon,\alpha,\xi}}_{\varepsilon,V}^2
\]
is continuous on \(K_\varepsilon\). Hence
\[
m_\varepsilon:=\inf_{(\alpha,\xi)\in K_\varepsilon}
\norm{u_\varepsilon-W_{\varepsilon,\alpha,\xi}}_{\varepsilon,V}^2
\]
is achieved. Since \((\one,\xi_\varepsilon)\in K_\varepsilon\), Lemma~\ref{lem:energy-approx} yields
\begin{equation}\label{eq:m-upper}
m_\varepsilon\le
\norm{u_\varepsilon-W_\varepsilon^*}_{\varepsilon,V}^2
\le C\varepsilon^{N+2\tau}.
\end{equation}
Let \((\alpha_\varepsilon,\widehat{\xi}_\varepsilon)\) be a minimizer. Then
\[
\norm{u_\varepsilon-W_{\varepsilon,\alpha_\varepsilon,\widehat{\xi}_\varepsilon}}_{\varepsilon,V}
\le C\varepsilon^{N/2+\tau},
\]
and the triangle inequality gives
\begin{equation}\label{eq:manifold-distance}
\norm{W_{\varepsilon,\alpha_\varepsilon,\widehat{\xi}_\varepsilon}-W_\varepsilon^*}_{\varepsilon,V}
\le C\varepsilon^{N/2+\tau}.
\end{equation}

For the rest of the proof, take
\[
b_j:=\alpha_j-1,
\qquad
\eta_j:=\xi_j-\xi_{\varepsilon,j},
\qquad
\eta_{j,h}:=\xi_{j,h}-\xi_{\varepsilon,j,h}.
\]
Uniformly for \((\alpha,\xi)\in K_\varepsilon\), the Taylor expansion of
\((\alpha_j,\xi_j)\mapsto \alpha_jW_{\varepsilon,j}(x;\xi_j)\) at
\((1,\xi_{\varepsilon,j})\) gives
\begin{equation}\label{eq:manifold-taylor}
W_{\varepsilon,\alpha,\xi}-W_\varepsilon^*
=
\sum_{j=1}^k b_jW_{\varepsilon,j}^*
+
\sum_{j=1}^k\sum_{h=1}^N\eta_{j,h}Y_{\varepsilon,j,h}(x;\xi_{\varepsilon,j})
+
\mathcal R_\varepsilon(\alpha,\xi).
\end{equation}
Thus,
\[
\begin{aligned}
\mathcal R_\varepsilon(\alpha,\xi)
=
&\sum_{j=1}^k
\left[
W_{\varepsilon,j}(x;\xi_j)
-
W_{\varepsilon,j}^*
-
\sum_{h=1}^N
\eta_{j,h}Y_{\varepsilon,j,h}(x;\xi_{\varepsilon,j})
\right]                                      \\
&+
\sum_{j=1}^k
b_j
\left[
W_{\varepsilon,j}(x;\xi_j)
-
W_{\varepsilon,j}^*
\right].
\end{aligned}
\]
We prove that
\begin{equation}\label{eq:manifold-taylor-remainder}
\|\mathcal R_\varepsilon(\alpha,\xi)\|_{\varepsilon,V}
=
o\left(
\varepsilon^{N/2}|\alpha-\mathbf 1|
+
\varepsilon^{N/2-1}|\xi-\xi_\varepsilon|
\right),
\end{equation}
uniformly for \((\alpha,\xi)\in K_\varepsilon\). 

For the first term,
\[
\begin{aligned}
&W_{\varepsilon,j}(x;\xi_j)
-
W_{\varepsilon,j}^*
-
\sum_{h=1}^N
\eta_{j,h}Y_{\varepsilon,j,h}(x;\xi_{\varepsilon,j})                         \\
&=
\int_0^1
\frac{d}{dt}
W_{\varepsilon,j}(x;\xi_{\varepsilon,j}+t\eta_j)\,dt
-
\sum_{h=1}^N
\eta_{j,h}Y_{\varepsilon,j,h}(x;\xi_{\varepsilon,j})   \\
&=
\sum_{h=1}^N
\eta_{j,h}
\int_0^1
\left[
Y_{\varepsilon,j,h}(x;\xi_{\varepsilon,j}+t\eta_j)
-
Y_{\varepsilon,j,h}(x;\xi_{\varepsilon,j})
\right]\,dt .
\end{aligned}
\]
Meanwhile, after the change of variables \(x=\xi_{\varepsilon,j}+\varepsilon y\), we have
\[
\begin{aligned}
&\left\|
Y_{\varepsilon,j,h}(x;\xi_{\varepsilon,j}+t\eta_j)
-
Y_{\varepsilon,j,h}(x;\xi_{\varepsilon,j})
\right\|_{\varepsilon,V}                                      \\
&\le
C\varepsilon^{N/2-1}
\left\|
\partial_h w_{V(\xi_{\varepsilon,j}+t\eta_j)}
\left(
y-\frac{t\eta_j}{\varepsilon}
\right)
-
\partial_h w_{V(\xi_{\varepsilon,j})}(y)
\right\|_{H^s(\mathbb R^N)}                                   \\
&\quad
+
C\varepsilon^{N/2}
\left\|
\left.
\partial_\lambda w_\lambda
\left(
y-\frac{t\eta_j}{\varepsilon}
\right)
\right|_{\lambda=V(\xi_{\varepsilon,j}+t\eta_j)}
\partial_hV(\xi_{\varepsilon,j}+t\eta_j)
-
\left.
\partial_\lambda w_\lambda(y)
\right|_{\lambda=V(\xi_{\varepsilon,j})}
\partial_hV(\xi_{\varepsilon,j})
\right\|_{H^s(\mathbb R^N)}                                   \\
&\quad +o(\varepsilon^{N/2-1})\\
&=
o(\varepsilon^{N/2-1}).
\end{aligned}
\]
Therefore, by the Minkowski inequality, we obtain
\begin{equation}\label{eq:estimate-for-R-1}
  \begin{aligned}
&\left\|
\sum_{j=1}^k
\left[
W_{\varepsilon,j}(x;\xi_j)
-
W_{\varepsilon,j}^*
-
\sum_{h=1}^N
\eta_{j,h}Y_{\varepsilon,j,h}(x;\xi_{\varepsilon,j})
\right]
\right\|_{\varepsilon,V}                                      \\
&\le
\sum_{j=1}^k\sum_{h=1}^N
|\eta_{j,h}|
\int_0^1
\left\|
Y_{\varepsilon,j,h}(x;\xi_{\varepsilon,j}+t\eta_j)
-
Y_{\varepsilon,j,h}(x;\xi_{\varepsilon,j})
\right\|_{\varepsilon,V}\,dt                                  \\
&=
o(\varepsilon^{N/2-1})
|\xi-\xi_\varepsilon|.
\end{aligned}
\end{equation}

For the second term, by a similar argument, we have
\begin{equation}\label{eq:estimate-for-R-2}
  \begin{aligned}
&\left\|
\sum_{j=1}^k
b_j
\left[
W_{\varepsilon,j}(x;\xi_j)-W_{\varepsilon,j}^*
\right]
\right\|_{\varepsilon,V}                                      \\
&\le
o(1)
\left(
\varepsilon^{N/2}|\alpha-\mathbf 1|
+
\varepsilon^{N/2-1}|\xi-\xi_\varepsilon|
\right).
\end{aligned}
\end{equation}
Combining \eqref{eq:estimate-for-R-1} and \eqref{eq:estimate-for-R-2} gives \eqref{eq:manifold-taylor-remainder}.

On the other hand, by \eqref{eq:app-same-ww}--\eqref{eq:app-diff-yy},  we have
\begin{equation}\label{eq:linear-gram-lower-bound}
  \begin{aligned}
&\|\sum_{j=1}^k b_jW_{\varepsilon,j}^*
+
\sum_{j=1}^k\sum_{h=1}^N\eta_{j,h}Y_{\varepsilon,j,h}(x;\xi_{\varepsilon,j})\|_{\varepsilon,V}^2\\
=&
\varepsilon^N
\sum_{j=1}^k
\bigl(a_0V(\xi_{\varepsilon,j})^\theta+o(1)\bigr)b_j^2
+
\varepsilon^{N-2}
\sum_{j=1}^k\sum_{h=1}^N
(K_j+o(1))\eta_{j,h}^2     \\
&
+
o(1)
\left(
\varepsilon^N |b|^2
+
\varepsilon^{N-2}|\eta|^2
\right),\\
 \ge&
c\varepsilon^N|\alpha-\one|^2
+c\varepsilon^{N-2}|\xi-\xi_\varepsilon|^2,
\end{aligned}
\end{equation}
for all \((\alpha,\xi)\in K_\varepsilon\) and sufficiently small \(\varepsilon\). Combining \eqref{eq:manifold-taylor-remainder}
with \eqref{eq:linear-gram-lower-bound} gives
\begin{equation}\label{eq:gram-lower-bound}
\norm{W_{\varepsilon,\alpha,\xi}-W_\varepsilon^*}_{\varepsilon,V}^2
\ge
c_1\varepsilon^N|\alpha-\one|^2
+c_2\varepsilon^{N-2}|\xi-\xi_\varepsilon|^2.
\end{equation}
Applying \eqref{eq:gram-lower-bound} to the minimizer and using \eqref{eq:manifold-distance}, we obtain
\[
\varepsilon^N|\alpha_\varepsilon-\one|^2+
\varepsilon^{N-2}|\widehat{\xi}_\varepsilon-\xi_\varepsilon|^2
\le C\varepsilon^{N+2\tau}.
\]
Therefore
\begin{equation}\label{eq:minimizer-parameter-estimates}
|\alpha_\varepsilon-\one|\le C\varepsilon^\tau,
\qquad
|\widehat{\xi}_\varepsilon-\xi_\varepsilon|\le C\varepsilon^{1+\tau}.
\end{equation}
This implies that
\((\alpha_\varepsilon,\widehat{\xi}_\varepsilon)\) lies in the interior of \(K_\varepsilon\) for all sufficiently small
\(\varepsilon\).

Moreover, $\max_j|\widehat\xi_{\varepsilon,j}-\xi_{\varepsilon,j}|/\varepsilon=o(1)$, and hence the triangle inequality yields
\begin{equation}\label{eq:equivalent-modulation-weights}
C^{-1}\rho_{\tau,\varepsilon,\xi_\varepsilon}
\le \rho_{\tau,\varepsilon,\widehat\xi_\varepsilon}
\le C\rho_{\tau,\varepsilon,\xi_\varepsilon}
\qquad\text{in }\mathbb R^N.
\end{equation}
Thus the corresponding weighted norms are uniformly equivalent.

The derivatives of
\[
F_\varepsilon(\alpha,\xi):=\norm{u_\varepsilon-W_{\varepsilon,\alpha,\xi}}_{\varepsilon,V}^2
\]
therefore vanish at \((\alpha_\varepsilon,\widehat{\xi}_\varepsilon)\). With
\[
\omega_\varepsilon:=u_\varepsilon-W_{\varepsilon,\alpha_\varepsilon,\widehat{\xi}_\varepsilon},
\]
we have
\[
0=\partial_{\alpha_j}F_\varepsilon(\alpha_\varepsilon,\widehat{\xi}_\varepsilon)
=-2\langle\omega_\varepsilon,
W_{\varepsilon,j}(x;\widehat{\xi}_{\varepsilon,j})\rangle_{\varepsilon,V},
\]
and
\[
0=\partial_{\xi_{j,h}}F_\varepsilon(\alpha_\varepsilon,\widehat{\xi}_\varepsilon)
=-2\alpha_{\varepsilon,j}\langle\omega_\varepsilon,
Y_{\varepsilon,j,h}(x;\widehat{\xi}_{\varepsilon,j})\rangle_{\varepsilon,V}.
\]
Since \(\alpha_{\varepsilon,j}=1+O(\varepsilon^\tau)\), these two identities imply
\[
\omega_\varepsilon\in E_{\varepsilon,\widehat{\xi}_\varepsilon}.
\]
Meanwhile,
\[
\norm{\omega_\varepsilon}_{\varepsilon,V}^2=m_\varepsilon\le C\varepsilon^{N+2\tau}
\]
by \eqref{eq:m-upper}. This estimate and \eqref{eq:minimizer-parameter-estimates} prove
\eqref{eq:min-realization-estimate}.

Moreover, we estimate the difference between
$W_{\varepsilon,\alpha_\varepsilon,\widehat{\xi}_\varepsilon}$ and $W_\varepsilon^*$ in the
weighted norm. We claim that
\[
\left\|
W_{\varepsilon,\alpha_\varepsilon,\widehat{\xi}_\varepsilon}
-
W_\varepsilon^*
\right\|_{\tau,\varepsilon,\widehat{\xi}_\varepsilon}
\le
C|\alpha_\varepsilon-\mathbf 1|
+
C\frac{|\widehat{\xi}_\varepsilon-\xi_\varepsilon|}{\varepsilon}
+
C|\widehat{\xi}_\varepsilon-\xi_\varepsilon|.
\]
Indeed, for every $x\in\mathbb R^N$ and every $j=1,\ldots,k$,
\[
\begin{aligned}
&\alpha_{\varepsilon,j}
w_{V(\widehat{\xi}_{\varepsilon,j})}
\left(
\frac{x-\widehat{\xi}_{\varepsilon,j}}{\varepsilon}
\right)
-
w_{V(\xi_{\varepsilon,j})}
\left(
\frac{x-\xi_{\varepsilon,j}}{\varepsilon}
\right)  \\
&=
(\alpha_{\varepsilon,j}-1)
w_{V(\widehat{\xi}_{\varepsilon,j})}
\left(
\frac{x-\widehat{\xi}_{\varepsilon,j}}{\varepsilon}
\right)  \\
&\quad+
\left[
w_{V(\widehat{\xi}_{\varepsilon,j})}
\left(
\frac{x-\widehat{\xi}_{\varepsilon,j}}{\varepsilon}
\right)
-
w_{V(\widehat{\xi}_{\varepsilon,j})}
\left(
\frac{x-\xi_{\varepsilon,j}}{\varepsilon}
\right)
\right]  \\
&\quad+
\left[
w_{V(\widehat{\xi}_{\varepsilon,j})}
\left(
\frac{x-\xi_{\varepsilon,j}}{\varepsilon}
\right)
-
w_{V(\xi_{\varepsilon,j})}
\left(
\frac{x-\xi_{\varepsilon,j}}{\varepsilon}
\right)
\right].
\end{aligned}
\]
By Lemma~\ref{lem:single-bubble-input} and \eqref{eq:minimizer-parameter-estimates}, a direct computation gives
\[
\begin{aligned}
\left\|
W_{\varepsilon,\alpha_\varepsilon,\widehat{\xi}_\varepsilon}
-
W_\varepsilon^*
\right\|_{\tau,\varepsilon,\widehat{\xi}_\varepsilon}
&\le
C|\alpha_\varepsilon-\mathbf 1|
+
C\frac{|\widehat{\xi}_\varepsilon-\xi_\varepsilon|}{\varepsilon}
+
C|\widehat{\xi}_\varepsilon-\xi_\varepsilon|.
\end{aligned}
\]
Thus,
\begin{equation}\label{eq:manifold-weighted-estimate}
\left\|
W_{\varepsilon,\alpha_\varepsilon,\widehat{\xi}_\varepsilon}
-
W_\varepsilon^*
\right\|_{\tau,\varepsilon,\widehat{\xi}_\varepsilon}
\le
C\varepsilon^\tau .
\end{equation}
Consequently, by Proposition~\ref{prop:local-expansion} and \eqref{eq:manifold-weighted-estimate}, we obtain
\[
\begin{aligned}
\|\omega_\varepsilon\|_{\tau,\varepsilon,\widehat{\xi}_\varepsilon}
&=
\left\|
u_\varepsilon
-
W_{\varepsilon,\alpha_\varepsilon,\widehat{\xi}_\varepsilon}
\right\|_{\tau,\varepsilon,\widehat{\xi}_\varepsilon}  \\
&\le
\|u_\varepsilon-W_\varepsilon^*\|_{\tau,\varepsilon,\widehat{\xi}_\varepsilon}
+
\left\|
W_{\varepsilon,\alpha_\varepsilon,\widehat{\xi}_\varepsilon}
-
W_\varepsilon^*
\right\|_{\tau,\varepsilon,\widehat{\xi}_\varepsilon}  \\
&\le
C\varepsilon^\tau,
\end{aligned}
\]
which proves \eqref{eq:min-realization-weighted-estimate}.
\end{proof}

We now prove that the decomposition is unique.
\begin{lemma}\label{lem:local-decomposition}
There exist $\rho_*,r_*,\varepsilon_*>0$, independent of the base
point, such that the following holds. If
$0<\varepsilon<\varepsilon_*$,
$(\alpha^0,\xi^*)\in\overline{\mathcal D}_\varepsilon$, and
\[
\varepsilon^{-N/2}
\|u-W_{\varepsilon,\alpha^0,\xi^*}\|_{\varepsilon,V}
<\rho_*,
\]
then there is a unique pair $(\alpha,\xi)$ satisfying
\[
|\alpha-\alpha^0|
+\varepsilon^{-1}|\xi-\xi^*|
<r_*,
\qquad
u-W_{\varepsilon,\alpha,\xi}\in E_{\varepsilon,\xi}.
\]
Moreover, $(\alpha,\xi)$ depends $C^1$-smoothly on $u$.
\end{lemma}

\begin{proof}
Set
\[
b=\alpha-\alpha^0,
\qquad
\eta=\frac{\xi-\xi^*}{\varepsilon},
\]
and write
\[
R(u,b,\eta)
:=
u-W_{\varepsilon,\alpha^0+b,\xi^*+\varepsilon\eta}.
\]
Define the scaled orthogonality map
\[
G(u,b,\eta)
:=
\left(
\varepsilon^{-N}
\langle R(u,b,\eta),W_{\varepsilon,j}\rangle_{\varepsilon,V},
\;
\varepsilon^{1-N}
\langle R(u,b,\eta),Y_{\varepsilon,j,h}\rangle_{\varepsilon,V}
\right)_{j,h},
\]
where the test functions are evaluated at
$\xi_j^*+\varepsilon\eta_j$. Then
\[
G(u,b,\eta)=0
\quad\Longleftrightarrow\quad
R(u,b,\eta)\in E_{\varepsilon,\xi^*+\varepsilon\eta}.
\]

At
\[
(u,b,\eta)
=
\bigl(W_{\varepsilon,\alpha^0,\xi^*},0,0\bigr),
\]
the derivative $D_{(b,\eta)} G$ is, up to signs and the
factors $\alpha_j^0=1+o(1)$, the $\varepsilon^{-N}$-normalized Gram
matrix of
\[
W_{\varepsilon,j}
\quad\text{and}\quad
\varepsilon Y_{\varepsilon,j,h}.
\]
By Lemma~C.1, this matrix converges uniformly, with respect to the
base point, to an invertible block-diagonal matrix. Hence it is
uniformly invertible for sufficiently small $\varepsilon$.
The implicit-function theorem, with constants uniform in the base
point, gives the required existence, uniqueness, and $C^1$
dependence.
\end{proof}

\begin{corollary}\label{cor:injectivity}
For all sufficiently small $\varepsilon$, the map
\[
(\alpha,\xi,\omega)
\longmapsto
W_{\varepsilon,\alpha,\xi}+\omega
\]
is injective on $S_\varepsilon$.
\end{corollary}

\begin{proof}
There exists $c>0$ such that, for all sufficiently small
$\varepsilon$ and all
\[
(\alpha,\xi),\,
(\widetilde\alpha,\widetilde\xi)
\in\overline{\mathcal D}_\varepsilon,
\]
one has the uniform estimate
\begin{equation}\label{eq:separation-estimate}
  \varepsilon^{-N}
\|W_{\varepsilon,\alpha,\xi}
-W_{\varepsilon,\widetilde\alpha,\widetilde\xi}\|_{\varepsilon,V}^2
\ge
c\left(
|\alpha-\widetilde\alpha|^2
+
\sum_{j=1}^k
\min\left\{
\frac{|\xi_j-\widetilde\xi_j|^2}{\varepsilon^2},1
\right\}
\right).
\end{equation}

Indeed, \eqref{eq:separation-estimate} follows from a Taylor expansion of the profile map and the uniform Gram estimates in Lemma~\ref{lem:app-gram-estimates}. 

Suppose that two triples in $S_\varepsilon$ represent the same
function:
\[
W_{\varepsilon,\alpha,\xi}+\omega
=
W_{\varepsilon,\widetilde\alpha,\widetilde\xi}
+\widetilde\omega.
\]
Then
\[
W_{\varepsilon,\alpha,\xi}
-W_{\varepsilon,\widetilde\alpha,\widetilde\xi}
=
\widetilde\omega-\omega.
\]
By the bounds defining $S_\varepsilon$ and \eqref{eq:separation-estimate},
\[
|\alpha-\widetilde\alpha|
+\varepsilon^{-1}|\xi-\widetilde\xi|
\le C\varepsilon^{\delta_0}.
\]
Thus, for sufficiently small $\varepsilon$, both parameter pairs
lie in the same uniqueness neighborhood in
Lemma~\ref{lem:local-decomposition}. Since both remainders satisfy
the corresponding orthogonality conditions, that gives
\[
\alpha=\widetilde\alpha,
\qquad
\xi=\widetilde\xi.
\]
The equality of the two representations then yields
$\omega=\widetilde\omega$.
\end{proof}

\begin{proof}[Proof of Theorem~\ref{th:parametrization-class}]

Proposition~\ref{prop:minimization} shows that $(\alpha_\varepsilon,\widehat\xi_\varepsilon,\omega_\varepsilon)\in S_\varepsilon$ for all sufficiently small $\varepsilon$.  The uniqueness follows from Lemma~\ref{lem:local-decomposition} and Corollary~\ref{cor:injectivity}.

\end{proof}

\section{Coercivity and reduction}\label{sec:approx-manifold-estimates}

In this section, we collect useful estimates for an approximate linearized problem and the approximating profile \(W_{\varepsilon,\alpha_\varepsilon,\widehat{\xi}_\varepsilon}\). 

We use the constants \(0<\delta_0<\tau\), \(\tau_0>0\), and the parameter domain
\(\mathcal D_\varepsilon\) fixed in \eqref{eq:parameter-domain}.
We denote its closure by \(\overline{\mathcal D}_\varepsilon\).

\subsection{Uniform coercivity on \texorpdfstring{$E_{\varepsilon,\xi}$}{Eepsxi}}\label{subsec:uniform-coercivity}

We first give a uniform spectral gap estimate for the linearized operator.
\begin{lemma}\label{lem:spectral-gap}
For $\lambda\in[V_{\min},V_{\max}]$, define
\[
\langle\phi,\psi\rangle_\lambda:=
\int_{\R^N}(-\Delta)^{s/2}\phi(-\Delta)^{s/2}\psi+
\lambda\int_{\R^N}\phi\psi,
\qquad
\norm{\phi}_\lambda^2:=\langle\phi,\phi\rangle_\lambda,
\]
and the symmetric bilinear form
\[
Q_\lambda(\phi,\psi):=
\langle\phi,\psi\rangle_\lambda
-p\int_{\R^N}w_\lambda^{p-1}\phi\psi,
\qquad
Q_\lambda(\phi):=Q_\lambda(\phi,\phi).
\]
There exists $\sigma_*>0$ such that
\begin{equation*}
Q_\lambda(\phi)\ge\sigma_*\norm{\phi}_\lambda^2
\end{equation*}
whenever
\[
\langle\phi,w_\lambda\rangle_\lambda=0,
\qquad
\langle\phi,\partial_hw_\lambda\rangle_\lambda=0,
\quad h=1,\ldots,N.
\]
\end{lemma}

\begin{proof}
For fixed \(\lambda\), set
\[
M_\lambda
=
\{\phi:\langle\phi,w_\lambda\rangle_\lambda=0,\
\langle\phi,\partial_hw_\lambda\rangle_\lambda=0,\ h=1,\ldots,N\}.
\]
For \(\phi\in M_\lambda\), using the equation satisfied by $w_\lambda$, we have 
\[
Q_\lambda(w_\lambda,\phi)
=
(1-p)\int_{\mathbb R^N}w_\lambda^p\phi
=
(1-p)\langle w_\lambda,\phi\rangle_\lambda
=
0,
\]
while
\[
Q_\lambda(w_\lambda)
=
(1-p)\int_{\mathbb R^N}w_\lambda^{p+1}<0.
\]

If \(Q_\lambda(\phi)<0\) for some \(\phi\in M_\lambda\), then
\(\operatorname{span}\{w_\lambda,\phi\}\) is a two-dimensional negative
subspace for \(Q_\lambda\), contradicting Lemma~\ref{lem:single-bubble-input}, which states that the Morse index is one. Hence
\(Q_\lambda\ge0\) on \(M_\lambda\).

If \(Q_\lambda(\phi)=0\) for some \(\phi\in M_\lambda\), then for every
\(\psi\in M_\lambda\) we have
 \[
0\le Q_\lambda(\phi+t\psi)
=
2tQ_\lambda(\phi,\psi)+t^2Q_\lambda(\psi),
\]
where we used \(Q_\lambda(\phi)=0\).  Taking \(t>0\), dividing by \(t\),
and letting \(t \to 0\) gives \(Q_\lambda(\phi,\psi)\ge0\).  Replacing
\(t\) by \(-t\) gives \(Q_\lambda(\phi,\psi)\le0\).  Hence
\[
Q_\lambda(\phi,\psi)=0,\qquad \psi\in M_\lambda.
\]
Together with
\[
Q_\lambda(\phi,w_\lambda)=0,
\qquad
Q_\lambda(\phi,\partial_hw_\lambda)=0,
\]
this gives \(L_\lambda\phi=0\). By nondegeneracy in Lemma~\ref{lem:single-bubble-input},
\[
\phi\in\operatorname{span}\{\partial_1w_\lambda,\ldots,\partial_Nw_\lambda\}.
\]
The orthogonality conditions then force \(\phi=0\).

It remains to make the lower bound uniform for $\lambda\in[V_{\min},V_{\max}]$.  Suppose not.  Then there are $\lambda_n\in[V_{\min},V_{\max}]$ and $\phi_n\in H^s(\R^N)$ such that
\[
\norm{\phi_n}_{\lambda_n}=1,
\qquad
\langle\phi_n,w_{\lambda_n}\rangle_{\lambda_n}=0,
\qquad
\langle\phi_n,\partial_hw_{\lambda_n}\rangle_{\lambda_n}=0,
\]
for $h=1,\ldots,N$, and
\[
Q_{\lambda_n}(\phi_n)\to0.
\]
After passing to a subsequence, $\lambda_n\to\lambda_*$ and $\phi_n \rightharpoonup  \phi$ in $H^s(\R^N)$.  By Lemma~\ref{lem:single-bubble-input}, 
\[
\int_{\R^N}w_{\lambda_n}^{p-1}\phi_n^2
\longrightarrow
\int_{\R^N}w_{\lambda_*}^{p-1}\phi^2.
\]
The orthogonality relations also pass to the limit:
\[
\langle\phi,w_{\lambda_*}\rangle_{\lambda_*}=0,
\qquad
\langle\phi,\partial_hw_{\lambda_*}\rangle_{\lambda_*}=0,
\quad h=1,\ldots,N.
\]
If $\phi=0$, we obtain 
\[
p\int_{\R^N}w_{\lambda_n}^{p-1}\phi_n^2=o(1),
\]
and hence
\[
Q_{\lambda_n}(\phi_n)=\norm{\phi_n}_{\lambda_n}^2+o(1)=1+o(1),
\]
contrary to $Q_{\lambda_n}(\phi_n)\to0$.  Thus $\phi\ne0$.

By weak lower semicontinuity,
\[
Q_{\lambda_*}(\phi)
\le \liminf_{n\to\infty}Q_{\lambda_n}(\phi_n)=0.
\]
The preceding argument gives $Q_{\lambda_*}(\phi)\ge0$.  Hence $Q_{\lambda_*}(\phi)=0$, so $\phi\in\ker L_{\lambda_*}$. Since $\phi \in M_{\lambda_*}$, the orthogonality conditions force  $\phi=0$, a contradiction. This completes the proof.
\end{proof}

For $(\alpha,\xi)\in \overline{\mathcal D}_\varepsilon$, define
\[
L_{\varepsilon,\alpha,\xi}:=\varepsilon^{2s}(-\Delta)^s+V(x)-pW_{\varepsilon,\alpha,\xi}^{p-1},
\]
and
\[
Q_{\varepsilon,\alpha,\xi}(\phi):=
\norm{\phi}_{\varepsilon,V}^2-p\int_{\R^N}W_{\varepsilon,\alpha,\xi}^{p-1}\phi^2.
\]
Since $|\alpha_j-1|\le\varepsilon^{\delta_0}$ and $p$ is fixed, one has
\begin{equation}\label{eq:alpha-power-close}
\alpha_j^{p-1}=1+O(\varepsilon^{\delta_0})
\qquad\hbox{uniformly in }\overline{\mathcal D}_\varepsilon.
\end{equation}

Next, we prove uniform coercivity of $Q_{\varepsilon,\alpha,\xi}$ on $E_{\varepsilon,\xi}$.
\begin{proposition}\label{prop:coercivity}
There exists $c_1>0$ such that, for all sufficiently small $\varepsilon$, all $(\alpha,\xi)\in \overline{\mathcal D}_\varepsilon$, and all $\phi\in E_{\varepsilon,\xi}$,
\begin{equation}\label{eq:coercivity}
Q_{\varepsilon,\alpha,\xi}(\phi)
\ge c_1\norm{\phi}_{\varepsilon,V}^2.
\end{equation}
\end{proposition}

\begin{proof}
Suppose that \eqref{eq:coercivity} is false.  Then there exist
\[
\varepsilon_n\to0,
\qquad
(\alpha^n,\xi^n)\in \overline{\mathcal D}_{\varepsilon_n},
\qquad
0\ne\phi_n\in E_{\varepsilon_n,\xi^n},
\]
such that, with $W_n:=W_{\varepsilon_n,\alpha^n,\xi^n}$,
\[
Q_{\varepsilon_n,\alpha^n,\xi^n}(\phi_n)
=o(1)\norm{\phi_n}_{\varepsilon_n,V}^2 .
\]
Since
\[
Q_{\varepsilon_n,\alpha^n,\xi^n}(\phi_n)
=\norm{\phi_n}_{\varepsilon_n,V}^2
-p\int_{\R^N}W_n^{p-1}\phi_n^2,
\]
we have
\[
p\int_{\R^N}W_n^{p-1}\phi_n^2
=(1+o(1))\norm{\phi_n}_{\varepsilon_n,V}^2>0
\]
for large \(n\). We may assume
\begin{equation}\label{eq:coercivity-normalization}
p\int_{\R^N}W_n^{p-1}\phi_n^2=1.
\end{equation}
Then the preceding relation gives
\begin{equation}\label{eq:coercivity-norm-bound}
\norm{\phi_n}_{\varepsilon_n,V}^2=1+o(1).
\end{equation}
Fix $d>0$ so small that the balls $B_{4d}(\xi_j^0)$ are disjoint and
\begin{equation}\label{eq:d-small-oscillation}
\omega(d):=
\sup_{\substack{|x-y|\le4d\\ x,y\in\cup_jB_{4\tau_0}(\xi_j^0)}}|V(x)-V(y)|
\le \frac{\sigma_*}{64}.
\end{equation}
Let $\chi_{n,j}\in C_c^\infty(B_{2d}(\xi_j^n))$, $j=1,\ldots,k$, satisfy $\chi_{n,j}=1$ in $B_d(\xi_j^n)$, and let
\[
\chi_{n,0}^2+\sum_{j=1}^k\chi_{n,j}^2=1,
\]
{as in Proposition~\ref{prop:mu-small}}. Set
\begin{equation*}
\psi_{n,j}(y):=\varepsilon_n^{N/2}(\chi_{n,j}\phi_n)(\xi_j^n+\varepsilon_n y).
\end{equation*}

By \cite[Lemma~3.5]{frank2008hardy}, we have
\begin{equation}\label{eq:coercivity-ims-localization}
\norm{\phi_n}_{\varepsilon_n,V}^2
\ge \sum_{j=1}^k\norm{\chi_{n,j}\phi_n}_{\varepsilon_n,V}^2
+\norm{\chi_{n,0}\phi_n}_{\varepsilon_n,V}^2-o(1).
\end{equation}
Indeed, arguing as in the proof of Proposition~\ref{prop:mu-small}, we have 
\[
\varepsilon_n^{2s}\sum_{j=0}^k\iint
\frac{(\chi_{n,j}(x)-\chi_{n,j}(y))^2|\phi_n(x)||\phi_n(y)|}{|x-y|^{N+2s}}\,dxdy
\le C\varepsilon_n^{2s}\norm{\phi_n}_{L^2}^2=o(1).
\]
By the change of variables $x=\xi_j^n+\varepsilon_n y$,
\begin{equation}\label{eq:localized-norm}
\begin{aligned}
\norm{\chi_{n,j}\phi_n}_{\varepsilon_n,V}^2
&=\int_{\R^N}|(-\Delta)^{s/2}\psi_{n,j}|^2
+\int_{\R^N}V(\xi_j^n+\varepsilon_n y)\psi_{n,j}^2\,dy  \\
&=\norm{\psi_{n,j}}_{V(\xi_j^n)}^2
+O(\omega(d))\norm{\psi_{n,j}}_{L^2}^2  \\
&\ge (1-C\omega(d))\norm{\psi_{n,j}}_{V(\xi_j^n)}^2.
\end{aligned}
\end{equation}
Similarly, on \(\operatorname{supp}\chi_{n,j}\subset B_{2d}(\xi_j^n)\),
\begin{equation}\label{eq:localized-denominator}
\begin{aligned}
&p\int_{\R^N}W_n^{p-1}(\chi_{n,j}\phi_n)^2  \\
&\quad=
p\int_{\R^N}
\left(
\alpha_j^n W_{\varepsilon_n,j}(x;\xi_j^n)
+
\sum_{i\ne j}\alpha_i^n W_{\varepsilon_n,i}(x;\xi_i^n)
\right)^{p-1}
(\chi_{n,j}\phi_n)^2\,dx                         \\
&\quad=
p(\alpha_j^n)^{p-1}
\int_{\R^N}
W_{\varepsilon_n,j}^{p-1}(x;\xi_j^n)
(\chi_{n,j}\phi_n)^2\,dx                         \\
&\qquad
+p\int_{\R^N}
\Bigg[
\left(
\alpha_j^n W_{\varepsilon_n,j}(x;\xi_j^n)
+
\sum_{i\ne j}\alpha_i^n W_{\varepsilon_n,i}(x;\xi_i^n)
\right)^{p-1}
-
\left(
\alpha_j^n W_{\varepsilon_n,j}(x;\xi_j^n)
\right)^{p-1}
\Bigg]
(\chi_{n,j}\phi_n)^2\,dx                         \\
&\quad=
p(\alpha_j^n)^{p-1}
\int_{\R^N}
w_{V(\xi_j^n)}^{p-1}(y)\psi_{n,j}^2(y)\,dy
 +o(1),
\end{aligned}
\end{equation}
where the last equality uses \eqref{eq:coercivity-normalization} and Lemma~\ref{lem:single-bubble-input}.
Consequently,
\begin{equation}\label{eq:global-denominator-localized}
1=p\int W_n^{p-1}\phi_n^2
=\sum_{j=1}^k p(\alpha_j^n)^{p-1}\int w_{V(\xi_j^n)}^{p-1}\psi_{n,j}^2+o(1).
\end{equation}
Moreover, since
\[
\operatorname{supp}\chi_{n,0}
\subset
\mathbb R^N\setminus\bigcup_{j=1}^k B_d(\xi_j^n),
\]
and
\[
\sup_{\mathbb R^N\setminus\cup_jB_d(\xi_j^n)}W_n
\le C\varepsilon_n^{N+2s},
\]
we have
\[
\begin{aligned}
p\int_{\mathbb R^N}W_n^{p-1}(\chi_{n,0}\phi_n)^2
&\le C\varepsilon_n^{(N+2s)(p-1)}\|\chi_{n,0}\phi_n\|_{L^2(\mathbb R^N)}^2 \\
&= o(1)\|\chi_{n,0}\phi_n\|_{\varepsilon_n,V}^2.
\end{aligned}
\]
Therefore,
\begin{equation}\label{eq:outer-region-positive}
\|\chi_{n,0}\phi_n\|_{\varepsilon_n,V}^2
-p\int_{\mathbb R^N}W_n^{p-1}(\chi_{n,0}\phi_n)^2
\ge (1-o(1))\|\chi_{n,0}\phi_n\|_{\varepsilon_n,V}^2.
\end{equation}

We now localize the orthogonality.  Since $\langle\phi_n,W_{\varepsilon_n,j}(x;\xi_j^n)\rangle_{\varepsilon_n,V}=0$,
\begin{equation*}
\begin{aligned}
0&=\langle\chi_{n,j}\phi_n,W_{\varepsilon_n,j}\rangle_{\varepsilon_n,V}
+\langle(1-\chi_{n,j})\phi_n,W_{\varepsilon_n,j}\rangle_{\varepsilon_n,V} \\
&=\varepsilon_n^{N/2}\langle\psi_{n,j},w_{V(\xi_j^n)}\rangle_{V(\xi_j^n)} \\
&\quad+\varepsilon_n^{N/2}\int_{\mathbb R^N}
\bigl(V(\xi_j^n+\varepsilon_n y)-V(\xi_j^n)\bigr)
\psi_{n,j}w_{V(\xi_j^n)}\,dy \\
&\quad+\langle(1-\chi_{n,j})\phi_n,W_{\varepsilon_n,j}\rangle_{\varepsilon_n,V} \\
&=\varepsilon_n^{N/2}\langle\psi_{n,j},w_{V(\xi_j^n)}\rangle_{V(\xi_j^n)}
+O(\omega(d))\varepsilon_n^{N/2}\norm{\psi_{n,j}}_{V(\xi_j^n)}
+o(\varepsilon_n^{N/2}).
\end{aligned}
\end{equation*}
Thus
\begin{equation}\label{eq:orthogonality-alpha}
|\langle\psi_{n,j},w_{V(\xi_j^n)}\rangle_{V(\xi_j^n)}|
\le C\omega(d)\norm{\psi_{n,j}}_{V(\xi_j^n)}+o(1).
\end{equation}
Using \eqref{eq:tangent-vector-explicit} and $\langle\phi_n,Y_{\varepsilon_n,j,h}(x;\xi_j^n)\rangle_{\varepsilon_n,V}=0$, one obtains
\begin{equation*}
\begin{aligned}
0&=-\varepsilon_n^{N/2-1}
\langle\psi_{n,j},\partial_hw_{V(\xi_j^n)}\rangle_{V(\xi_j^n)}
+O(\varepsilon_n^{N/2})\norm{\psi_{n,j}}_{V(\xi_j^n)} \\
&\quad+O(\omega(d))\varepsilon_n^{N/2-1}\norm{\psi_{n,j}}_{V(\xi_j^n)}
+o(\varepsilon_n^{N/2-1}),
\end{aligned}
\end{equation*}
and hence
\begin{equation}\label{eq:orthogonality-translation}
|\langle\psi_{n,j},\partial_hw_{V(\xi_j^n)}\rangle_{V(\xi_j^n)}|
\le C\omega(d)\norm{\psi_{n,j}}_{V(\xi_j^n)}+o(1).
\end{equation}
Set
\[
S_{n,j}:=\operatorname{span}\{w_{V(\xi_j^n)},\partial_1w_{V(\xi_j^n)},\ldots,\partial_Nw_{V(\xi_j^n)}\}.
\]
Let $\Pi_{n,j}\psi_{n,j}$ denote the orthogonal projection, in the $\langle\cdot,\cdot\rangle_{V(\xi_j^n)}$ inner product, onto
\[
S_{n,j}^{\perp}.
\]
Since $V(\xi_j^n)$ stays in the compact interval $[V_{\min},V_{\max}]$,
the Gram matrices of the basis
\[
w_{V(\xi_j^n)},\partial_1w_{V(\xi_j^n)},\ldots,\partial_Nw_{V(\xi_j^n)}
\]
are uniformly invertible.  Hence, if
\(r_{n,j}:=\psi_{n,j}-\Pi_{n,j}\psi_{n,j}\in S_{n,j}\), then
\[
\norm{r_{n,j}}_{V(\xi_j^n)}
\le C\left(
|\langle r_{n,j},w_{V(\xi_j^n)}\rangle_{V(\xi_j^n)}|
+\sum_{h=1}^N
|\langle r_{n,j},\partial_hw_{V(\xi_j^n)}\rangle_{V(\xi_j^n)}|
\right).
\]
Because $\Pi_{n,j}\psi_{n,j}\perp S_{n,j}$, the inner products of \(r_{n,j}\)
with the above basis are exactly the corresponding inner products of
\(\psi_{n,j}\).  Equations \eqref{eq:orthogonality-alpha} and \eqref{eq:orthogonality-translation} therefore give
\begin{equation}\label{eq:projection-error}
\norm{\psi_{n,j}-\Pi_{n,j}\psi_{n,j}}_{V(\xi_j^n)}
\le C\omega(d)\norm{\psi_{n,j}}_{V(\xi_j^n)}+o(1).
\end{equation}

Set \(v_{n,j}:=\Pi_{n,j}\psi_{n,j}\).  By Lemma~\ref{lem:spectral-gap},
\[
Q_{V(\xi_j^n)}(v_{n,j})
\ge \sigma_*\norm{v_{n,j}}_{V(\xi_j^n)}^2 .
\]
On the other hand, the quadratic form \(Q_\lambda\) is continuous uniformly for
\(\lambda\in[V_{\min},V_{\max}]\).  Using \eqref{eq:projection-error} and the
boundedness of \(\norm{\psi_{n,j}}_{V(\xi_j^n)}\), which follows from
\eqref{eq:localized-norm} and \eqref{eq:coercivity-norm-bound}, we obtain
\[
\norm{v_{n,j}}_{V(\xi_j^n)}^2
\ge \bigl(1-C\omega(d)\bigr)\norm{\psi_{n,j}}_{V(\xi_j^n)}^2-o(1),
\]
and
\[
\begin{aligned}
Q_{V(\xi_j^n)}(\psi_{n,j})
&\ge Q_{V(\xi_j^n)}(v_{n,j})
-C\norm{\psi_{n,j}-v_{n,j}}_{V(\xi_j^n)}
\norm{\psi_{n,j}}_{V(\xi_j^n)}
-o(1) \\
&\ge \bigl(\sigma_*-C\omega(d)\bigr)
\norm{\psi_{n,j}}_{V(\xi_j^n)}^2-o(1).
\end{aligned}
\]
Finally, by \eqref{eq:alpha-power-close},
\[
\left|(\alpha_j^n)^{p-1}-1\right|\int_{\mathbb R^N}
w_{V(\xi_j^n)}^{p-1}\psi_{n,j}^2
\le o(1)\norm{\psi_{n,j}}_{V(\xi_j^n)}^2=o(1).
\]
Thus, after decreasing \(d\) if necessary so that the \(C\omega(d)\) term is
absorbed into \(\sigma_*/2\), and then taking \(n\) sufficiently large,
\begin{equation}\label{eq:localized-spectral-gap}
\begin{aligned}
&\norm{\psi_{n,j}}_{V(\xi_j^n)}^2
-p(\alpha_j^n)^{p-1}\int w_{V(\xi_j^n)}^{p-1}\psi_{n,j}^2   \\
&\quad\ge \frac{\sigma_*}{2}\norm{\psi_{n,j}}_{V(\xi_j^n)}^2-o(1),
\end{aligned}
\end{equation}
provided $d$ is fixed so that \eqref{eq:d-small-oscillation} holds and $n$ is large.

Combining \eqref{eq:coercivity-ims-localization}, \eqref{eq:localized-norm}, \eqref{eq:localized-denominator}, \eqref{eq:localized-spectral-gap}, and \eqref{eq:outer-region-positive}, we obtain
\begin{equation*}
Q_{\varepsilon_n,\alpha^n,\xi^n}(\phi_n)
\ge c\sum_{j=1}^k\norm{\psi_{n,j}}_{V(\xi_j^n)}^2
+(1-o(1))\norm{\chi_{n,0}\phi_n}_{\varepsilon_n,V}^2-o(1).
\end{equation*}
On the other hand, \eqref{eq:global-denominator-localized} and the boundedness of $(\alpha_j^n)^{p-1}$ give
\begin{equation*}
1\le C\sum_{j=1}^k\norm{\psi_{n,j}}_{V(\xi_j^n)}^2+o(1).
\end{equation*}
Thus
\[
Q_{\varepsilon_n,\alpha^n,\xi^n}(\phi_n)\ge c_0+o(1)>0.
\]
But by \eqref{eq:coercivity-normalization} and \eqref{eq:coercivity-norm-bound},
\[
Q_{\varepsilon_n,\alpha^n,\xi^n}(\phi_n)
=\norm{\phi_n}_{\varepsilon_n,V}^2-1=o(1),
\]
a contradiction.
\end{proof}

As an immediate consequence, we obtain the following result.
\begin{corollary}\label{cor:homotopy-coercivity}
For every \(t\in[0,1]\), the bilinear form
\[
B_t(\omega,\psi)
:=
\langle \omega,\psi\rangle_{\varepsilon,V}
-
tp\int_{\mathbb R^N}
W_{\varepsilon,\alpha,\xi}^{p-1}\omega\psi
\]
is uniformly coercive on \(E_{\varepsilon,\xi}\). More precisely, there exists
\(c>0\), independent of \(\varepsilon\), \(t\), and \((\alpha,\xi)\in \overline{\mathcal D}_\varepsilon\), such that
\[
B_t(\omega,\omega)
\ge
c\|\omega\|_{\varepsilon,V}^2,
\qquad
\omega\in E_{\varepsilon,\xi}.
\]
\end{corollary}

\begin{proof}
For \(\omega\in E_{\varepsilon,\xi}\), Proposition~\ref{prop:coercivity} gives
\[
\begin{aligned}
B_t(\omega,\omega)
&=
\|\omega\|_{\varepsilon,V}^2
-
tp\int_{\mathbb R^N}
W_{\varepsilon,\alpha,\xi}^{p-1}\omega^2
\\
&=
\left(
\|\omega\|_{\varepsilon,V}^2
-
p\int_{\mathbb R^N}
W_{\varepsilon,\alpha,\xi}^{p-1}\omega^2
\right)
+
(1-t)p\int_{\mathbb R^N}
W_{\varepsilon,\alpha,\xi}^{p-1}\omega^2
\\
&=
Q_{\varepsilon,\alpha,\xi}(\omega)
+
(1-t)p\int_{\mathbb R^N}
W_{\varepsilon,\alpha,\xi}^{p-1}\omega^2
\\
&\ge
c_1\|\omega\|_{\varepsilon,V}^2 .
\end{aligned}
\]
Since \(t\in[0,1]\) and \(W_{\varepsilon,\alpha,\xi}\ge0\), the second term is nonnegative.
\end{proof}

\subsection{Some crucial estimates}

\begin{lemma}\label{lem:alpha-equation}
For each $j=1,\ldots,k$, uniformly for $(\alpha,\xi)\in \overline{\mathcal D}_\varepsilon$,
\begin{equation}\label{eq:alpha-equation}
\mathcal E_\varepsilon'(W_{\varepsilon,
\alpha,
\xi})[W_{\varepsilon,j}(x;\xi_j)]
=a_0\varepsilon^NV(\xi_j)^\theta(\alpha_j-\alpha_j^p)+O(\varepsilon^{N+\tau}).
\end{equation}
\end{lemma}

\begin{proof}
For each $i=1,\ldots,k$, the function $W_{\varepsilon,i}(x;\xi_i)$ satisfies
\[
\varepsilon^{2s}(-\Delta)^sW_{\varepsilon,i}(x;\xi_i)
+V(\xi_i)W_{\varepsilon,i}(x;\xi_i)
=W_{\varepsilon,i}(x;\xi_i)^p
\qquad\hbox{in }\mathbb R^N.
\]
Testing this equation against $W_{\varepsilon,j}(x;\xi_j)$, we obtain
\begin{equation}\label{eq:alpha-frozen-testing}
\begin{aligned}
\langle W_{\varepsilon,i}(x;\xi_i),W_{\varepsilon,j}(x;\xi_j)\rangle_{\varepsilon,V}
&=\int_{\mathbb R^N}W_{\varepsilon,i}(x;\xi_i)^pW_{\varepsilon,j}(x;\xi_j)\,dx\\
&\quad+\int_{\mathbb R^N}(V(x)-V(\xi_i))
W_{\varepsilon,i}(x;\xi_i)W_{\varepsilon,j}(x;\xi_j)\,dx.
\end{aligned}
\end{equation}
Consequently,
\begin{equation}\label{eq:alpha-direct-decomposition}
\begin{aligned}
&\mathcal E_\varepsilon'(W_{\varepsilon,\alpha,\xi})
[W_{\varepsilon,j}(x;\xi_j)]\\
&=\sum_{i=1}^k(\alpha_i-\alpha_i^p)
\int_{\mathbb R^N}W_{\varepsilon,i}(x;\xi_i)^p
W_{\varepsilon,j}(x;\xi_j)\,dx\\
&\quad+\sum_{i=1}^k\alpha_i
\int_{\mathbb R^N}(V(x)-V(\xi_i))
W_{\varepsilon,i}(x;\xi_i)W_{\varepsilon,j}(x;\xi_j)\,dx\\
&\quad-\int_{\mathbb R^N}
\left[
W_{\varepsilon,\alpha,\xi}(x)^p
-\sum_{i=1}^k\alpha_i^pW_{\varepsilon,i}(x;\xi_i)^p
\right]W_{\varepsilon,j}(x;\xi_j)\,dx.
\end{aligned}
\end{equation}

For the second term, \eqref{eq:sep-tail-potential-pointwise} gives
\begin{equation}\label{eq:alpha-potential-errors}
\begin{aligned}
\sum_{i=1}^k
\left|\int_{\mathbb R^N}(V(x)-V(\xi_i))
W_{\varepsilon,i}(x;\xi_i)W_{\varepsilon,j}(x;\xi_j)\,dx\right|
&\le C\varepsilon^\tau
\int_{\mathbb R^N}\rho_{\tau,\varepsilon,\xi}(x)
W_{\varepsilon,j}(x;\xi_j)\,dx\\
&\le C\varepsilon^{N+\tau}.
\end{aligned}
\end{equation}
The $i=j$ term in the first sum is
\begin{equation}\label{eq:alpha-diagonal-main}
\begin{aligned}
(\alpha_j-\alpha_j^p)
\int_{\mathbb R^N}W_{\varepsilon,j}(x;\xi_j)^{p+1}\,dx
&=a_0\varepsilon^NV(\xi_j)^\theta(\alpha_j-\alpha_j^p).
\end{aligned}
\end{equation}
For $i\ne j$, using \eqref{eq:alpha-frozen-testing}, \eqref{eq:alpha-potential-errors} and \eqref{eq:app-diff-ww}, we have
\begin{equation}\label{eq:alpha-off-diagonal-overlap}
\begin{aligned}
\left|\int_{\mathbb R^N}W_{\varepsilon,i}(x;\xi_i)^p
W_{\varepsilon,j}(x;\xi_j)\,dx\right|
&\le
\left|\langle W_{\varepsilon,i}(x;\xi_i),
W_{\varepsilon,j}(x;\xi_j)\rangle_{\varepsilon,V}\right|\\
&\quad+
\left|\int_{\mathbb R^N}(V(x)-V(\xi_i))
W_{\varepsilon,i}(x;\xi_i)W_{\varepsilon,j}(x;\xi_j)\,dx\right|\\
&\le C\varepsilon^{2N+2s}+C\varepsilon^{N+\tau}
\le C\varepsilon^{N+\tau}.
\end{aligned}
\end{equation}

Finally, by \eqref{eq:sep-tail-nonlinear-pointwise}, we have
\begin{equation}\label{eq:alpha-nonlinear-error}
\begin{aligned}
&\left|\int_{\mathbb R^N}
\left[
W_{\varepsilon,\alpha,\xi}(x)^p
-\sum_{i=1}^k\alpha_i^pW_{\varepsilon,i}(x;\xi_i)^p
\right]W_{\varepsilon,j}(x;\xi_j)\,dx\right|\\
&\qquad\le C\varepsilon^\tau
\int_{\mathbb R^N}\rho_{\tau,\varepsilon,\xi}(x)
W_{\varepsilon,j}(x;\xi_j)\,dx
\le C\varepsilon^{N+\tau}.
\end{aligned}
\end{equation}
Combining \eqref{eq:alpha-direct-decomposition}--\eqref{eq:alpha-nonlinear-error},
we conclude that
\[
\mathcal E_\varepsilon'(W_{\varepsilon,\alpha,\xi})
[W_{\varepsilon,j}(x;\xi_j)]
=a_0\varepsilon^NV(\xi_j)^\theta(\alpha_j-\alpha_j^p)
+O(\varepsilon^{N+\tau}).
\]
This completes the proof.
\end{proof}

Choose \(d>0\) so small that the balls \(B_{4d}(\xi_j)\) are pairwise disjoint for every \(\xi\in\overline{\mathcal D}_\varepsilon\).  Let
\[
\eta\in C_c^\infty(\overline{\mathbb R^{N+1}_+}),
\qquad
0\le\eta\le1,
\qquad
\eta\equiv1\quad\hbox{in }\overline{\mathcal B_d^+(0)},
\qquad
\operatorname{supp}\eta\subset \overline{\mathcal B_{2d}^+(0)}.
\]
Set
\[
\eta_j(x,t):=\eta(x-\xi_j,t),
\qquad
\eta_j(x):=\eta_j(x,0).
\]
For \(u\in H^s(\R^N)\cap L^\infty(\R^N)\), let \(\bar u=E(u)\) be its extension, and set
\[
{F(t):=\frac{|t|^{p+1}}{p+1}.}
\]
Define
\begin{equation}\label{eq:pohozaev-functional}
\begin{aligned}
\mathcal P_{j,h}(u)
&:=\frac{\varepsilon^{2s}}{2}
\int_{\R^{N+1}_+}t^{1-2s}(\partial_h\eta_j)|\nabla\bar u|^2
-\varepsilon^{2s}
\int_{\R^{N+1}_+}t^{1-2s}(\partial_h\bar u)(\nabla\bar u\cdot\nabla\eta_j)\\
&\quad
+\frac{1}{2}\int_{\R^N}\partial_h(V\eta_j)u^2
-
\int_{\R^N}(\partial_h\eta_j)F(u).
\end{aligned}
\end{equation}
For smooth \(u\), integration by parts gives
\begin{equation}\label{eq:pohozaev-as-test}
\mathcal P_{j,h}(u)=\mathcal E_\varepsilon'(u)[-\eta_j\partial_hu].
\end{equation}
Formula \eqref{eq:pohozaev-functional} is taken as the definition; it involves only \(\nabla\bar u\in L^2(t^{1-2s})\) and is therefore meaningful in the fractional energy space.

\begin{lemma}\label{lem:xi-equation}
Uniformly for \((\alpha,\xi)\in\overline{\mathcal D}_\varepsilon\),
\begin{equation}\label{eq:pohozaev-xi-main}
\mathcal P_{j,h}(W_{\varepsilon,\alpha,\xi})
=c_*\theta\varepsilon^N V(\xi_j)^{\theta-1}\partial_hV(\xi_j)
+O(\varepsilon^N|\alpha_j-1|)+O(\varepsilon^{N+\tau}).
\end{equation}
\end{lemma}

\begin{proof}
Let
\[
A_j^+:=\mathcal B_{2d}^+(\xi_j)\setminus\overline{\mathcal B_d^+(\xi_j)}.
\]
The extension terms involving \(\nabla\eta_j\) are supported in \(A_j^+\). By the Poisson representation and Lemma~\ref{lem:single-bubble-input}, an argument analogous to that in the proof of Lemma~\ref{lem:extension-u-bound}(i), we have
\begin{equation}\label{eq:extension-w-gradient-bound-xi}
\varepsilon^{2s}\int_{A_j^+}t^{1-2s}|\nabla\overline{W_{\varepsilon,\alpha,\xi}}|^2
\le C\varepsilon^{2N}.
\end{equation}
Since the trace of \(\eta_j\) equals \(1\) on \(B_d(\xi_j)\) and \(\nabla\eta_j=0\) in \(\mathcal B_d^+(\xi_j)\), we keep the annular and different peak terms explicit and obtain
\[
\begin{aligned}
\mathcal P_{j,h}(W_{\varepsilon,\alpha,\xi})
&= 
\frac{1}{2}\int_{B_d(\xi_j)}\partial_hV(x)W_{\varepsilon,\alpha,\xi}^2(x)\,dx
+\frac{\varepsilon^{2s}}{2}\int_{A_j^+}t^{1-2s}(\partial_h\eta_j)|\nabla\overline{W_{\varepsilon,\alpha,\xi}}|^2\\
&\quad
-\varepsilon^{2s}\int_{A_j^+}t^{1-2s}(\partial_h\overline{W_{\varepsilon,\alpha,\xi}})(\nabla\overline{W_{\varepsilon,\alpha,\xi}}\cdot\nabla\eta_j)
+\frac{1}{2}\int_{B_{2d}(\xi_j)\setminus\overline{B_d(\xi_j)}}\partial_h(V\eta_j)W_{\varepsilon,\alpha,\xi}^2\\
&\quad
-\int_{B_{2d}(\xi_j)\setminus\overline{B_d(\xi_j)}}(\partial_h\eta_j)F(W_{\varepsilon,\alpha,\xi})\\
&=
\frac{1}{2}\alpha_j^2\int_{B_d(\xi_j)}\partial_hV(x)W_{\varepsilon,j}^2\,dx
+\alpha_j\sum_{i\ne j}\alpha_i
\int_{B_d(\xi_j)}\partial_hV(x)W_{\varepsilon,j}W_{\varepsilon,i}\,dx\\
&\quad
+
\frac{1}{2}\int_{B_d(\xi_j)}\partial_hV(x)\left(\sum_{i\ne j}\alpha_iW_{\varepsilon,i}\right)^2\,dx
+\frac{\varepsilon^{2s}}{2}\int_{A_j^+}t^{1-2s}(\partial_h\eta_j)|\nabla\overline{W_{\varepsilon,\alpha,\xi}}|^2\\
&\quad
-\varepsilon^{2s}\int_{A_j^+}t^{1-2s}(\partial_h\overline{W_{\varepsilon,\alpha,\xi}})(\nabla\overline{W_{\varepsilon,\alpha,\xi}}\cdot\nabla\eta_j)
+\frac{1}{2}\int_{B_{2d}(\xi_j)\setminus\overline{B_d(\xi_j)}}\partial_h(V\eta_j)W_{\varepsilon,\alpha,\xi}^2\\
&\quad
-\int_{B_{2d}(\xi_j)\setminus\overline{B_d(\xi_j)}}(\partial_h\eta_j)F(W_{\varepsilon,\alpha,\xi})\\
&=\frac{1}{2}\alpha_j^2\int_{B_d(\xi_j)}\partial_hV(x)W_{\varepsilon,j}(x;\xi_j)^2\,dx+O(\varepsilon^{N+\tau}).
\end{aligned}
\]
The last equality follows from Lemma~\ref{lem:single-bubble-input} and \eqref{eq:extension-w-gradient-bound-xi}.
Changing variables \(x=\xi_j+\varepsilon y\), we obtain
\[
\begin{aligned}
\frac{1}{2}\alpha_j^2
\int_{B_d(\xi_j)}
\partial_hV(x)W_{\varepsilon,j}(x;\xi_j)^2\,dx
&=
\frac{1}{2}\alpha_j^2\varepsilon^N
\int_{B_{d/\varepsilon}(0)}
\partial_hV(\xi_j+\varepsilon y)
w_{V(\xi_j)}^2(y)\,dy                                      \\
&=
\frac{1}{2}\alpha_j^2\varepsilon^N
\partial_hV(\xi_j)
\int_{\mathbb R^N}w_{V(\xi_j)}^2(y)\,dy
+
O(\varepsilon^{N+1})                                  \\
&=
\alpha_j^2 c_*\theta\varepsilon^N
V(\xi_j)^{\theta-1}\partial_hV(\xi_j)
+
O(\varepsilon^{N+1}).
\end{aligned}
\]

Here we used the identity
\[
\frac12\int_{\mathbb R^N}w_\lambda^2
=c_*\theta\lambda^{\theta-1}.
\]

Therefore
\[
\begin{aligned}
\mathcal P_{j,h}(W_{\varepsilon,\alpha,\xi})
&=
\alpha_j^2 c_*\theta\varepsilon^N
V(\xi_j)^{\theta-1}\partial_hV(\xi_j)
+
O(\varepsilon^{N+\tau})                             \\
&=
c_*\theta\varepsilon^N
V(\xi_j)^{\theta-1}\partial_hV(\xi_j)
+
O(\varepsilon^N|\alpha_j-1|)
+
O(\varepsilon^{N+\tau}),
\end{aligned}
\]
which proves \eqref{eq:pohozaev-xi-main}.

\end{proof}

Lastly, we give the following weighted norm estimate.
\begin{lemma}\label{lem:weighted-norm-estimate}
Let \(0\le t\le 1\), \((\alpha,\xi)\in \overline{\mathcal D}_\varepsilon\), and
\(\varphi\in E_{\varepsilon,\xi}\) solve
\[
\varepsilon^{2s}(-\Delta)^s\varphi+V(x)\varphi
-tpW_{\varepsilon,\alpha,\xi}^{p-1}\varphi=H
\qquad\text{in }\mathbb R^N .
\]
Suppose that, for some \(M>0\),
\[
|H(x)|\le M\rho_{\tau,\varepsilon,\xi}(x),
\qquad
\|\varphi\|_{\tau,\varepsilon,\xi}<\infty.
\]
Then there exists \(C>0\), independent of \(\varepsilon\), \(t\), and
\((\alpha,\xi)\), such that
\[
\|\varphi\|_{\tau,\varepsilon,\xi}\le CM .
\]
\end{lemma}

\begin{proof}
We argue by contradiction. Suppose that the estimate is false. Then there exist
\[
\varepsilon_n\to 0,\qquad
t_n\in[0,1],\qquad
(\alpha^n,\xi^n)\in\overline{\mathcal D}_{\varepsilon_n},
\]
and functions
\[
\varphi_n\in E_{\varepsilon_n,\xi^n},\qquad H_n,
\]
such that, with
\[
W_n:=W_{\varepsilon_n,\alpha^n,\xi^n},
\]
one has
\[
\varepsilon_n^{2s}(-\Delta)^s\varphi_n+V(x)\varphi_n
-t_npW_n^{p-1}\varphi_n=H_n
\qquad\text{in }\mathbb R^N,
\]
and
\[
|H_n(x)|\le m_n\rho_{\tau,\varepsilon_n,\xi^n}(x),\qquad
\|\varphi_n\|_{\tau,\varepsilon_n,\xi^n}=1,
\qquad
m_n\to0.
\]

We first rescale the equation. Define
\[
\widetilde\varphi_n(y):=\varphi_n(\varepsilon_n y),
\qquad
q_{n,j}:=\frac{\xi_j^n}{\varepsilon_n},
\qquad
\widetilde W_n(y):=W_n(\varepsilon_n y),
\]
and
\[
\rho_n(y):=\rho_{\tau,\varepsilon_n,\xi^n}(\varepsilon_n y)
=\sum_{j=1}^k(1+|y-q_{n,j}|)^{-N-2s+\tau}.
\]
Then
\[
\|\rho_n^{-1}\widetilde\varphi_n\|_{L^\infty(\mathbb R^N)}=1
\]
and
\[
(-\Delta)^s\widetilde\varphi_n+
\bigl[
V(\varepsilon_n y)-t_np\widetilde W_n(y)^{p-1}
\bigr]\widetilde\varphi_n
=
\widetilde H_n(y),
\]
where
\[
\widetilde H_n(y):=H_n(\varepsilon_n y),
\qquad
|\widetilde H_n(y)|\le m_n\rho_n(y).
\]

Choose \(R\gg1\) large, independently of \(n\). Since, by
Lemma~\ref{lem:single-bubble-input},
\[
\widetilde W_n(y)
\le
C\sum_{j=1}^k(1+|y-q_{n,j}|)^{-N-2s},
\]
we have, for
\[
y\in\mathbb R^N\setminus\bigcup_{j=1}^kB_R(q_{n,j}),
\]
that
\[
\widetilde W_n(y)^{p-1}
\le
CR^{-(N+2s)(p-1)}.
\]
Taking \(R\) sufficiently large gives
\[
V(\varepsilon_n y)-t_np\widetilde W_n(y)^{p-1}
\ge
\frac{1}{2}V_{\min}
\]
outside \(\bigcup_{j=1}^kB_R(q_{n,j})\). The Lemma~2.5 of \cite{MR3121716} then yields
\[
1=\|\rho_n^{-1}\widetilde\varphi_n\|_{L^\infty(\mathbb R^N)}
\le
C\left(
\|\widetilde\varphi_n\|_{L^\infty(\cup_jB_R(q_{n,j}))}
+
m_n
\right).
\]
Hence there exist \(j_n\in\{1,\dots,k\}\) and
\(y_n\in B_R(q_{n,j_n})\) such that
\[
|\widetilde\varphi_n(y_n)|\ge c_0>0.
\]

We convert this pointwise lower bound into an \(L^2\) lower bound.  Fix
\(m>0\). The rescaled equation can be rewritten as
\[
\widetilde\varphi_n
=T_m\!\left[
\widetilde H_n+
\bigl(m-V(\varepsilon_n y)+t_np\widetilde W_n^{p-1}\bigr)
\widetilde\varphi_n
\right],
\]
where \(T_m=((-\Delta)^s+m)^{-1}\). The expression in brackets is uniformly
bounded in \(L^2(\mathbb R^N)\cap L^\infty(\mathbb R^N)\). Hence the
H\"older estimate for \(T_m\) in \cite[Lemma~2.3]{MR3121716} gives a radius
\(r>0\), independent of \(n\), such that
\[
|\widetilde\varphi_n(y)|\ge \frac{c_0}{2}
\qquad\text{for }y\in B_r(y_n).
\]
Consequently,
\begin{equation}\label{eq:weighted-inverse-l2-lower}
\begin{aligned}
\|\varphi_n\|_{L^2(\mathbb R^N)}^2
&=\varepsilon_n^N\|\widetilde\varphi_n\|_{L^2(\mathbb R^N)}^2  \\
&\ge
\varepsilon_n^N\int_{B_r(y_n)}|\widetilde\varphi_n(y)|^2\,dy
\ge c_1\varepsilon_n^N .
\end{aligned}
\end{equation}

On the other hand, the equation and Corollary~\ref{cor:homotopy-coercivity}
give
\[
\begin{aligned}
c\|\varphi_n\|_{\varepsilon_n,V}^2
&\le
\|\varphi_n\|_{\varepsilon_n,V}^2
-t_np\int_{\mathbb R^N}W_n^{p-1}\varphi_n^2                         \\
&=
\int_{\mathbb R^N}H_n\varphi_n                                      \\
&\le
\|H_n\|_{L^2(\mathbb R^N)}
\|\varphi_n\|_{L^2(\mathbb R^N)} .
\end{aligned}
\]
Since
\[
\|H_n\|_{L^2(\mathbb R^N)}
\le m_n\|\rho_{\tau,\varepsilon_n,\xi^n}\|_{L^2(\mathbb R^N)}
\le Cm_n\varepsilon_n^{N/2},
\]
and \(\|\varphi_n\|_{L^2}\le C\|\varphi_n\|_{\varepsilon_n,V}\), we obtain
\[
\|\varphi_n\|_{\varepsilon_n,V}\le Cm_n\varepsilon_n^{N/2}.
\]
Thus
\begin{equation}\label{eq:weighted-inverse-l2-upper}
\|\varphi_n\|_{L^2(\mathbb R^N)}^2
\le C\|\varphi_n\|_{\varepsilon_n,V}^2
=o(\varepsilon_n^N),
\end{equation}
which contradicts \eqref{eq:weighted-inverse-l2-lower}.  This proves the
desired weighted estimate.
\end{proof}

\section{Degree reduction and computation}\label{sec:degree-reduction}
In this section, we compute the Leray--Schauder degree. Recall that
\[
\mathcal V(\xi):=\sum_{j=1}^k V(\xi_j)^\theta,\qquad
\theta=\frac{p+1}{p-1}-\frac{N}{2s}>0.
\]
We use \(m(\xi^0,\mathcal V)\) for the Morse index of \(\mathcal V\) at \(\xi^0\).
Since \(\nabla V(\xi_j^0)=0\),
\[
 m(\xi^0,\mathcal V)=\sum_{j=1}^k m(\xi_j^0,V).
\]

For each fixed \(\varepsilon\) and \((\alpha,\xi)\in\mathcal D_\varepsilon\), set
\[
\mathcal W_{\varepsilon,\xi}:=
E_{\varepsilon,\xi}\cap
\{\omega:\|\omega\|_{\tau,\varepsilon,\xi}<\infty\}
\]
with the Banach norm
\[
\|\omega\|_{\mathcal W_{\varepsilon,\xi}}
:=\|\omega\|_{\varepsilon,V}+\|\omega\|_{\tau,\varepsilon,\xi}.
\]


The degree domain is the open set $S_\varepsilon$ introduced in \eqref{eq:local-degree-set}.

For \((\alpha,\xi,\omega)\in S_\varepsilon\), we define
\begin{equation*}
A(\alpha,\xi,\omega)=
\bigl(A_{1,j}(\alpha,\xi,\omega),A_{2,j,h}(\alpha,\xi,\omega),A_3(\alpha,\xi,\omega)\bigr),
\end{equation*}
where
\begin{equation*}
A_{1,j}(\alpha,\xi,\omega):=\mathcal E_\varepsilon'(W_{\varepsilon,\alpha,\xi}+\omega)[W_{\varepsilon,j}(x;\xi_j)],
\qquad
A_{2,j,h}(\alpha,\xi,\omega):=\mathcal P_{j,h}(W_{\varepsilon,\alpha,\xi}+\omega),
\end{equation*}
and \(A_3(\alpha,\xi,\omega)\in E_{\varepsilon,\xi}\) satisfies
\begin{equation*}
\langle A_3(\alpha,\xi,\omega),\widetilde\omega\rangle_{\varepsilon,V}
=\mathcal E_\varepsilon'(W_{\varepsilon,\alpha,\xi}+\omega)[\widetilde\omega]
\qquad
\forall\widetilde\omega\in E_{\varepsilon,\xi}.
\end{equation*}

For each fixed small $\varepsilon$, the weighted norms are uniformly
equivalent as $(\alpha,\xi)$ varies in $\mathcal{D}_\varepsilon$. Then by a local trivialization argument for the smoothly varying
finite-codimensional spaces $E_{\varepsilon,\xi}$, $S_\varepsilon$ may be regarded as an open subset of a fixed Banach space, see \cite{FinsterLottner2021Banach, kato1995perturbation, McDuffWehrheim2024Polyfold}. Moreover,
since $W_{\varepsilon,\alpha,\xi}$ is orthogonal to
$E_{\varepsilon,\xi}$, the definition of $A_3$ gives
\[
\langle A_3(\alpha,\xi,\omega),\widetilde\omega
\rangle_{\varepsilon,V}
 =
 \langle\omega,\widetilde\omega\rangle_{\varepsilon,V}
 -
 \int_{\mathbb R^N}
 |W_{\varepsilon,\alpha,\xi}+\omega|^{p-1}
 (W_{\varepsilon,\alpha,\xi}+\omega)\widetilde\omega,
 \qquad
 \widetilde\omega\in E_{\varepsilon,\xi}.
\]
Thus $A_3$ is the identity in $\omega$ minus a compact nonlinear
operator, while $A_1$ and $A_2$ are finite-dimensional; consequently,
$A$ is an admissible Leray--Schauder map on $S_\varepsilon$.

We next show that the zeros of $A$ are precisely the concentrating solutions.

\begin{proposition}\label{prop:fractional-a-equivalence}
For all sufficiently small \(\varepsilon\), the following assertions are equivalent.
\begin{enumerate}[label=(\roman*)]
\item \(u_\varepsilon\) is a positive solution of \eqref{eq:fractional-schrodinger-problem} concentrating at \(\xi_1^0,\ldots,\xi_k^0\).
\item There exists \((\alpha,\xi,\omega)\in S_\varepsilon\) such that
\[
A(\alpha,\xi,\omega)=0,
\qquad
u_\varepsilon= W_{\varepsilon,\alpha,\xi}+\omega.
\]
\end{enumerate}
\end{proposition}

\begin{proof}
If \(u_\varepsilon\) is a concentrating solution, Theorem~\ref{th:parametrization-class} gives
\[
u_\varepsilon=W_{\varepsilon,\alpha,\xi}+\omega,
\qquad
\omega\in E_{\varepsilon,\xi},
\qquad 
(\alpha,\xi,\omega)\in S_\varepsilon.
\]
Thus \(A(\alpha,\xi,\omega)=0\).

Conversely, suppose that \(A(\alpha,\xi,\omega)=0\), and set \(u_\varepsilon:=W_{\varepsilon,\alpha,\xi}+\omega\). Then \(A_3=0\) gives
\[
\mathcal E_\varepsilon'(u_\varepsilon)[\widetilde\omega]=0
\qquad
\forall\widetilde\omega\in E_{\varepsilon,\xi}.
\]
There are constants \(\beta_j\) and \(\gamma_{jh}\) such that, for every \(\psi\in H^s(\R^N)\),
\begin{equation}\label{eq:a-equivalence-multipliers-detailed}
\mathcal E_\varepsilon'(u_\varepsilon)[\psi]
=
\sum_{j=1}^k\beta_j\langle W_{\varepsilon,j}(x;\xi_j),\psi\rangle_{\varepsilon,V}
+
\sum_{j=1}^k\sum_{h=1}^N\gamma_{jh}\langle Y_{\varepsilon,j,h}(x;\xi_j),\psi\rangle_{\varepsilon,V}.
\end{equation}
Equivalently,
\begin{equation}\label{eq:a-equivalence-weak-residual}
{\varepsilon^{2s}(-\Delta)^su_\varepsilon+V(x)u_\varepsilon-|u_\varepsilon|^{p-1}u_\varepsilon}
=
\sum_{j=1}^k\beta_jR_j+
\sum_{j=1}^k\sum_{h=1}^N\gamma_{jh}S_{j,h},
\end{equation}
where
\[
\begin{aligned}
R_j&:=W_{\varepsilon,j}(x;\xi_j)^p
+(V(x)-V(\xi_j))W_{\varepsilon,j}(x;\xi_j),\\
S_{j,h}&:=pW_{\varepsilon,j}(x;\xi_j)^{p-1}Y_{\varepsilon,j,h}(x;\xi_j)
-\partial_hV(\xi_j)W_{\varepsilon,j}(x;\xi_j)\\
&\quad +(V(x)-V(\xi_j))Y_{\varepsilon,j,h}(x;\xi_j).
\end{aligned}
\]
Testing \eqref{eq:a-equivalence-multipliers-detailed} with \(W_{\varepsilon,i}(x;\xi_i)\) and using \(A_{1,i}=0\) gives
\begin{equation*}
0=
\sum_{j=1}^k\beta_j\langle W_{\varepsilon,j}(x;\xi_j),W_{\varepsilon,i}(x;\xi_i)\rangle_{\varepsilon,V}
+
\sum_{j=1}^k\sum_{h=1}^N\gamma_{jh}\langle Y_{\varepsilon,j,h}(x;\xi_j),W_{\varepsilon,i}(x;\xi_i)\rangle_{\varepsilon,V}.
\end{equation*}
By \eqref{eq:app-same-ww}, \eqref{eq:app-same-wy}, \eqref{eq:app-diff-ww} and \eqref{eq:app-diff-wy}, we have
\begin{equation}\label{eq:beta-system-from-a1}
\varepsilon^N\bigl(a_0V(\xi_i)^\theta+o(1)\bigr)\beta_i
+o(\varepsilon^N)\sum_{j\ne i}|\beta_j|
+o(\varepsilon^{N-1})\sum_{j,h}|\gamma_{jh}|=0.
\end{equation}

Set
\[
\begin{aligned}
\mathcal R_{ij\ell}
&:=
-\int_{\mathbb R^N}R_j\eta_i\partial_\ell W_{\varepsilon,\alpha,\xi}
+\int_{\mathbb R^N}\omega\,\partial_\ell(\eta_iR_j),\\
\mathcal S_{ijh\ell}
&:=
-\int_{\mathbb R^N}S_{j,h}\eta_i\partial_\ell W_{\varepsilon,\alpha,\xi}
+\int_{\mathbb R^N}\omega\,\partial_\ell(\eta_iS_{j,h}) .
\end{aligned}
\]
Since \(A_{2,i,\ell}=\mathcal P_{i,\ell}(u_\varepsilon)=0\), the localized Pohozaev identity applied to \eqref{eq:a-equivalence-weak-residual} gives
\begin{equation}\label{eq:forced-pohozaev-coefficients}
\begin{aligned}
0
&=
-\int_{\mathbb R^N}\left(\sum_{j=1}^k\beta_jR_j
+\sum_{j=1}^k\sum_{h=1}^N\gamma_{jh}S_{j,h}\right)\,\eta_i\partial_\ell W_{\varepsilon,\alpha,\xi}\\
&\quad +\int_{\mathbb R^N}\omega\,\partial_\ell\left(\left(\sum_{j=1}^k\beta_jR_j
+\sum_{j=1}^k\sum_{h=1}^N\gamma_{jh}S_{j,h}\right) \eta_i\right) \\
&=
\sum_{j=1}^k\beta_j\mathcal R_{ij\ell}
+\sum_{j=1}^k\sum_{h=1}^N\gamma_{jh}\mathcal S_{ijh\ell}.
\end{aligned}
\end{equation}

For the \(R_j\)-coefficients, by Lemma~\ref{lem:app-gram-estimates},
\begin{equation}\label{eq:r-coefficient-long}
\begin{aligned}
\mathcal R_{ij\ell}
&= 
-\sum_{r=1}^k\alpha_r
\int_{\mathbb R^N}
R_j\eta_i\partial_\ell W_{\varepsilon,r}(x;\xi_r)
+\int_{\mathbb R^N}\omega\,\partial_\ell(\eta_iR_j)\\
&=
\frac{\alpha_i\delta_{ij}}{p+1}
\int_{\mathbb R^N}
(\partial_\ell\eta_i)W_{\varepsilon,i}(x;\xi_i)^{p+1}
+
\frac{\alpha_i\delta_{ij}}{2}
\int_{\mathbb R^N}
\partial_\ell\!\bigl((V(x)-V(\xi_i))\eta_i\bigr)W_{\varepsilon,i}(x;\xi_i)^2 \\
&\qquad
+O(\varepsilon^{N+2s-1})
+\int_{\mathbb R^N}\omega\,\partial_\ell(\eta_iR_j)\\
&=
O(\varepsilon^{N})
+O(\varepsilon^{N+2s-1})
+
O\!\left(
\|\omega\|_{L^2(\mathbb R^N)}
\|\partial_\ell(\eta_iR_j)\|_{L^2(\mathbb R^N)}
\right)\\
&=
O(\varepsilon^{N})
+O(\varepsilon^{N+2s-1})
+O(\varepsilon^{N-1+\delta_0})
=
o(\varepsilon^{N-1}) .
\end{aligned}
\end{equation}

For the \(S_{j,h}\)-coefficients, by Lemma~\ref{lem:app-gram-estimates},
\begin{equation}\label{eq:s-coefficient-long}
\begin{aligned}
\mathcal S_{ijh\ell}
&=
-\sum_{r=1}^k\alpha_r
\int_{\mathbb R^N}
S_{j,h}\eta_i\partial_\ell W_{\varepsilon,r}(x;\xi_r) +\int_{\mathbb R^N}\omega\,\partial_\ell(\eta_iS_{j,h})\\
&=
-\alpha_i\delta_{ij}
\int_{B_d(\xi_i)}
pW_{\varepsilon,i}^{p-1}
\left[
-\varepsilon^{-1}\partial_h w_{V(\xi_i)}\left(\frac{x-\xi_i}{\varepsilon}\right)
+\left.\partial_\lambda w_\lambda\left(\frac{x-\xi_i}{\varepsilon}\right)\right|_{\lambda=V(\xi_i)}\partial_hV(\xi_i)
\right]\\
&\quad\times 
\left[
\varepsilon^{-1}\partial_\ell w_{V(\xi_i)}\left(\frac{x-\xi_i}{\varepsilon}\right)
\right]\,dx +O(\varepsilon^{N-1})
+O(\varepsilon^{N+2s-2})\\
&\quad
+
O\!\left(
\|\omega\|_{L^2(\mathbb R^N)}
\|\partial_\ell(\eta_iS_{j,h})\|_{L^2(\mathbb R^N)}
\right)\\
&=
\alpha_iK_i\varepsilon^{N-2}\delta_{ij}\delta_{h\ell}
+O(\varepsilon^{N-1})
+O(\varepsilon^{N+2s-2})
+O(\varepsilon^{N-2+\delta_0})\\
&=
\alpha_iK_i\varepsilon^{N-2}\delta_{ij}\delta_{h\ell}
+o(\varepsilon^{N-2}),
\end{aligned}
\end{equation}
where
\[
K_i:=p\int_{\mathbb R^N}w_{V(\xi_i)}^{p-1}
\bigl(\partial_1w_{V(\xi_i)}\bigr)^2>0 .
\]

Substituting \eqref{eq:r-coefficient-long} and \eqref{eq:s-coefficient-long} into
\eqref{eq:forced-pohozaev-coefficients}, we obtain
\begin{equation}\label{eq:gamma-system-from-pohozaev}
o(\varepsilon^{N-1})\sum_{j=1}^k|\beta_j|
+\alpha_iK_i\varepsilon^{N-2}\gamma_{i\ell}
+o(\varepsilon^{N-2})\sum_{j=1}^k\sum_{h=1}^N|\gamma_{jh}|=0 .
\end{equation}

Let
\[
\widetilde\gamma_{jh}:=\varepsilon^{-1}\gamma_{jh}.
\]
Dividing \eqref{eq:beta-system-from-a1} by \(\varepsilon^N\) and
\eqref{eq:gamma-system-from-pohozaev} by \(\varepsilon^{N-1}\), we obtain
\[
\begin{aligned}
\bigl(a_0V(\xi_i)^\theta+o(1)\bigr)\beta_i
+o(1)\sum_{j=1}^k|\beta_j|
+o(1)\sum_{j=1}^k\sum_{h=1}^N|\widetilde\gamma_{jh}|&=0,\\
o(1)\sum_{j=1}^k|\beta_j|
+(\alpha_iK_i+o(1))\widetilde\gamma_{i\ell}
+o(1)\sum_{j=1}^k\sum_{h=1}^N|\widetilde\gamma_{jh}|&=0 .
\end{aligned}
\]
The last two systems form a finite-dimensional linear system whose coefficient
matrix is a small perturbation of
\[
\operatorname{diag}\bigl(a_0V(\xi_1)^\theta,\ldots,a_0V(\xi_k)^\theta,\alpha_1K_1 I_N,\ldots,\alpha_kK_k I_N\bigr),
\]
which is uniformly invertible for sufficiently small \(\varepsilon\).  Hence
\[
\beta_i=0,\qquad \widetilde\gamma_{i\ell}=0,
\]
and consequently \(\gamma_{i\ell}=0\).  Thus \(\mathcal E_\varepsilon'(u_\varepsilon)=0\) on all of
\(H^s(\R^N)\), and \(u_\varepsilon\) solves the equation
\[
\varepsilon^{2s}(-\Delta)^su_\varepsilon+V(x)u_\varepsilon=|u_\varepsilon|^{p-1}u_\varepsilon\qquad\hbox{in }\R^N.
\]
We now close the positivity argument.  Testing the weak equation with \(-u_\varepsilon^-\), where \(u_\varepsilon^-:=\max\{-u_\varepsilon,0\}\), and using the elementary fractional inequality
\[
\int_{\mathbb R^N}(-\Delta)^{s/2}u_\varepsilon\,(-\Delta)^{s/2}(-u_\varepsilon^-)\,dx
\ge \int_{\mathbb R^N}|(-\Delta)^{s/2}u_\varepsilon^-|^2\,dx,
\]
 we obtain
\[
\|u_\varepsilon^-\|_{\varepsilon,V}^2\le \int_{\R^N}(u_\varepsilon^-)^{p+1}.
\]
Since \(u_\varepsilon=W_{\varepsilon,\alpha,\xi}+\omega\), \(W_{\varepsilon,\alpha,\xi}>0\), and \((\alpha,\xi,\omega)\in S_\varepsilon\), one has \(u_\varepsilon^-\le |\omega|\). Moreover, the weight is uniformly bounded above, and therefore
\[
\|\omega\|_{L^\infty(\R^N)}\le C\|\omega\|_{\tau,\varepsilon,\xi}\le C\varepsilon^{\delta_0}.
\]
Using also \(\int_{\mathbb R^N}(u_\varepsilon^-)^2\le V_{\min}^{-1}\|u_\varepsilon^-\|_{\varepsilon,V}^2\), we obtain
\[
\int_{\R^N}(u_\varepsilon^-)^{p+1}
\le \|u_\varepsilon^-\|_{L^\infty}^{p-1}\int_{\R^N}(u_\varepsilon^-)^2
\le C\varepsilon^{\delta_0(p-1)}\|u_\varepsilon^-\|_{\varepsilon,V}^2.
\]
Choosing \(\varepsilon\) so small that \(C\varepsilon^{\delta_0(p-1)}<1\), we obtain \(u_\varepsilon^-=0\). Thus \(u_\varepsilon\ge0\), and the strong maximum principle gives \(u_\varepsilon>0\). Consequently \(|u_\varepsilon|^{p-1}u_\varepsilon=u_\varepsilon^p\), so \(u_\varepsilon\) solves the original positive problem \eqref{eq:fractional-schrodinger-problem}. The concentration at \((\xi_1^0,\ldots,\xi_k^0)\) follows from \((\alpha,\xi,\omega)\in S_\varepsilon\).
\end{proof}

\begin{remark}\label{rem:pohozaev-xi-component}
We use the localized Pohozaev functional \(\mathcal P_{j,h}\), rather than
\[
\mathcal E_\varepsilon'(W_{\varepsilon,\alpha,\xi}+\omega)
[Y_{\varepsilon,j,h}(x;\xi_j)].
\]
The reason is that the latter term contains
\[
\int_{\mathbb R^N}
W_{\varepsilon,\alpha,\xi}^{p-2}\omega^2
Y_{\varepsilon,j,h}(x;\xi_j)\,dx,
\]
which is difficult to control directly. After integration by parts in the extension space, the derivatives of the cut-off functions are supported in annuli away from \(\xi_j\).
\end{remark}

We have the following degree formula for the map \(A\).
\begin{theorem}\label{thm:fractional-degree-formula}
Assume that \(\xi^0=(\xi_1^0,\ldots,\xi_k^0)\)  has different components and  is a nondegenerate critical point of the reduced potential \(\mathcal V\) defined in Theorem~\ref{th:degree-counting}.  Then, for all sufficiently small \(\varepsilon\),
\begin{equation}\label{eq:fractional-degree-formula}
\deg(A,S_\varepsilon,0)=(-1)^{k+m(\xi^0,\mathcal V)}.
\end{equation}
\end{theorem}

To prove Theorem~\ref{thm:fractional-degree-formula}, we first reduce the degree of \(A\) to the degree of the comparison map \(B\), and then compute the resulting finite-dimensional degree.

We now introduce the finite-dimensional comparison map
\begin{equation*}
B(\alpha,\xi,\omega)=
\bigl(B_{1,j}(\alpha,\xi),B_{2,j,h}(\alpha,\xi),B_3(\omega)\bigr),
\end{equation*}
where
\begin{equation*}
B_{1,j}(\alpha,\xi):=\mathcal E_\varepsilon'(W_{\varepsilon,\alpha,\xi})[W_{\varepsilon,j}(x;\xi_j)],
\qquad
B_{2,j,h}(\alpha,\xi):=\mathcal P_{j,h}(W_{\varepsilon,\alpha,\xi}),
\end{equation*}
and \(B_3(\omega)\in E_{\varepsilon,\xi}\) is the identity in the energy metric:
\begin{equation*}
\langle B_3(\omega),\widetilde\omega\rangle_{\varepsilon,V}
=\langle \omega,\widetilde\omega\rangle_{\varepsilon,V}
\qquad
\forall\widetilde\omega\in E_{\varepsilon,\xi}.
\end{equation*}

The following estimates are the basic stability results for the homotopy between
\(A\) and \(B\).

\begin{lemma}\label{lem:degree-stability-estimates}
Uniformly for \((\alpha,\xi)\in\overline{\mathcal D}_\varepsilon\),
\begin{equation}\label{eq:projected-residual-estimate}
\sup_{\substack{\psi\in E_{\varepsilon,\xi}\\
\norm{\psi}_{\varepsilon,V}=1}}
|\mathcal E_\varepsilon'(W_{\varepsilon,\alpha,\xi})[\psi]|
\le C\varepsilon^{N/2+\tau}.
\end{equation}
If \(\omega\in E_{\varepsilon,\xi}\) satisfies
\begin{equation*}
\norm{\omega}_{\varepsilon,V}\le C\varepsilon^{N/2+\tau},
\qquad
\norm{\omega}_{\tau,\varepsilon,\xi}\le \varepsilon^{\delta_0},
\end{equation*}
then
\begin{equation}\label{eq:a1-b1-stability}
A_{1,j}(\alpha,\xi,\omega)=B_{1,j}(\alpha,\xi)+O(\varepsilon^{N+\tau}),
\end{equation}
and
\begin{equation}\label{eq:a2-b2-stability}
A_{2,j,h}(\alpha,\xi,\omega)=B_{2,j,h}(\alpha,\xi)+O(\varepsilon^{N+\tau}).
\end{equation}

\end{lemma}

\begin{proof}
By Lemmas~\ref{lem:single-bubble-input} and \ref{lem:weighted-tail}, a direct computation gives
\begin{equation}\label{eq:l2-variable-potential-residual}
\|(V-V(\xi_i))W_{\varepsilon,i}\|_{L^2}
+\left\|W_{\varepsilon,\alpha,\xi}^p-
\sum_{i=1}^k\alpha_i^pW_{\varepsilon,i}^p\right\|_{L^2}
\le C\varepsilon^{N/2+\tau},
\end{equation}
and, for each fixed \(j\),
\begin{equation}\label{eq:l2-separated-nonlinear-residual}
\left\|W_{\varepsilon,\alpha,\xi}^{p-1}W_{\varepsilon,j}
-\alpha_j^{p-1}W_{\varepsilon,j}^p\right\|_{L^2}
\le C\varepsilon^{N/2+\tau}.
\end{equation}
Let \(\psi\in E_{\varepsilon,\xi}\), \(\|\psi\|_{\varepsilon,V}=1\).  Using
\(\langle W_{\varepsilon,i},\psi\rangle_{\varepsilon,V}=0\) and the equation of
\(W_{\varepsilon,i}\), we have
\[
\int_{\mathbb R^N}W_{\varepsilon,i}^p\psi
=-\int_{\mathbb R^N}(V-V(\xi_i))W_{\varepsilon,i}\psi .
\]
Therefore
\[
\begin{aligned}
\mathcal E_\varepsilon'(W_{\varepsilon,\alpha,\xi})[\psi]
&=
\langle W_{\varepsilon,\alpha,\xi},\psi\rangle_{\varepsilon,V}
-
\int_{\mathbb R^N}W_{\varepsilon,\alpha,\xi}^p\psi                                                                        \\
&=
-\int_{\mathbb R^N}
\left(
W_{\varepsilon,\alpha,\xi}^p-\sum_{i=1}^k\alpha_i^pW_{\varepsilon,i}(x;\xi_i)^p
\right)\psi
-
\sum_{i=1}^k\alpha_i^p
\int_{\mathbb R^N}W_{\varepsilon,i}(x;\xi_i)^p\psi                                    \\
&=
-\int_{\mathbb R^N}
\left(
W_{\varepsilon,\alpha,\xi}^p-\sum_{i=1}^k\alpha_i^pW_{\varepsilon,i}(x;\xi_i)^p
\right)\psi
+
\sum_{i=1}^k\alpha_i^p
\int_{\mathbb R^N}
(V(x)-V(\xi_i))W_{\varepsilon,i}(x;\xi_i)\psi .
\end{aligned}
\]
Since \(\|\psi\|_{L^2}\le C\|\psi\|_{\varepsilon,V}\), \eqref{eq:l2-variable-potential-residual} gives
\eqref{eq:projected-residual-estimate}.

We next prove the estimate for $A_{1,j} $. Set
\[
{N_{\varepsilon,\alpha,\xi}(\omega):=|W_{\varepsilon,\alpha,\xi}+\omega|^{p-1}(W_{\varepsilon,\alpha,\xi}+\omega)-W_{\varepsilon,\alpha,\xi}^p-pW_{\varepsilon,\alpha,\xi}^{p-1}\omega,}
\qquad
\kappa:=\min\{2,p\}>1.
\]
Since \(W_{\varepsilon,\alpha,\xi}\) is uniformly bounded, we have
\begin{equation}\label{eq:nonlinear-remainder-pointwise}
|N_{\varepsilon,\alpha,\xi}(\omega)|
\le
C\|\omega\|_{L^\infty}^{\kappa-1}|\omega|.
\end{equation}
Using \(\omega\in E_{\varepsilon,\xi}\), we have
\[
\begin{aligned}
A_{1,j}(\alpha,\xi,\omega)-B_{1,j}(\alpha,\xi)
&=
\mathcal E_\varepsilon'(W_{\varepsilon,\alpha,\xi}+\omega)
[W_{\varepsilon,j}(x;\xi_j)]
-
\mathcal E_\varepsilon'(W_{\varepsilon,\alpha,\xi})
[W_{\varepsilon,j}(x;\xi_j)]                                      \\
&=
\langle\omega,W_{\varepsilon,j}(x;\xi_j)\rangle_{\varepsilon,V}
-
\int_{\mathbb R^N}
{\bigl(|W_{\varepsilon,\alpha,\xi}+\omega|^{p-1}
(W_{\varepsilon,\alpha,\xi}+\omega)-W_{\varepsilon,\alpha,\xi}^p\bigr)}\\
&\quad \times W_{\varepsilon,j}(x;\xi_j)                            \\
&=
-p\int_{\mathbb R^N}
W_{\varepsilon,\alpha,\xi}^{p-1}\omega\,
W_{\varepsilon,j}(x;\xi_j)
-\int_{\mathbb R^N}
N_{\varepsilon,\alpha,\xi}(\omega)\,
W_{\varepsilon,j}(x;\xi_j) .
\end{aligned}
\]
For the first term, \eqref{eq:l2-variable-potential-residual} and
\eqref{eq:l2-separated-nonlinear-residual} give
\[
\begin{aligned}
\int_{\mathbb R^N}
W_{\varepsilon,\alpha,\xi}^{p-1}\omega W_{\varepsilon,j}(x;\xi_j)
&=
\alpha_j^{p-1}
\int_{\mathbb R^N}
W_{\varepsilon,j}(x;\xi_j)^p\omega                                      
+\int_{\mathbb R^N}
\Bigl[
W_{\varepsilon,\alpha,\xi}^{p-1}W_{\varepsilon,j}(x;\xi_j)
-\alpha_j^{p-1}W_{\varepsilon,j}(x;\xi_j)^p
\Bigr]\omega                                                 \\
&=
-\alpha_j^{p-1}
\int_{\mathbb R^N}
(V(x)-V(\xi_j))W_{\varepsilon,j}(x;\xi_j)\omega
+
\int_{\mathbb R^N}
\Bigl[
W_{\varepsilon,\alpha,\xi}^{p-1}W_{\varepsilon,j}(x;\xi_j)
\\
&\qquad
-\alpha_j^{p-1}W_{\varepsilon,j}(x;\xi_j)^p
\Bigr]\omega \\
&\le
C\varepsilon^{N/2+\tau}\|\omega\|_{L^2}
+
C\varepsilon^{N/2+\tau}\|\omega\|_{L^2}=
O(\varepsilon^{N+\tau}).
\end{aligned}
\]
Moreover, for the second term, using \eqref{eq:nonlinear-remainder-pointwise},
\[
\begin{aligned}
\left|
\int_{\mathbb R^N}N_{\varepsilon,\alpha,\xi}(\omega)W_{\varepsilon,j}(x;\xi_j)
\right|
&\le
C\|\omega\|_{L^\infty}^{\kappa-1}
\|\omega\|_{L^2}
\|W_{\varepsilon,j}(x;\xi_j)\|_{L^2}                                                   \\
&=
O(\varepsilon^{N+\tau+(\kappa-1)\delta_0})
=
O(\varepsilon^{N+\tau}).
\end{aligned}
\]
Consequently,
\[
A_{1,j}(\alpha,\xi,\omega)
=
B_{1,j}(\alpha,\xi)
+
O(\varepsilon^{N+\tau}),
\]
which implies \eqref{eq:a1-b1-stability}.

We now prove the estimate for $A_{2,j,h} $. We have
\[
\begin{aligned}
A_{2,j,h}&(\alpha,\xi,\omega)-B_{2,j,h}(\alpha,\xi)\\
&=
\mathcal P_{j,h}(W_{\varepsilon,\alpha,\xi}+\omega)
-\mathcal P_{j,h}(W_{\varepsilon,\alpha,\xi})                         \\
&=
\frac{\varepsilon^{2s}}{2}
\int_{\mathbb R^{N+1}_+}
t^{1-2s}(\partial_h\eta_j)
\left(
|\nabla(\overline W_{\varepsilon,\alpha,\xi}+\overline\omega)|^2
-
|\nabla\overline W_{\varepsilon,\alpha,\xi}|^2
\right)                                      \\
&\quad
-\varepsilon^{2s}
\int_{\mathbb R^{N+1}_+}
t^{1-2s}
\left[
\partial_h(\overline W_{\varepsilon,\alpha,\xi}+\overline\omega)
\nabla(\overline W_{\varepsilon,\alpha,\xi}+\overline\omega)-
\partial_h\overline W_{\varepsilon,\alpha,\xi}\nabla\overline W_{\varepsilon,\alpha,\xi}
\right]\cdot\nabla\eta_j                    \\
&\quad
+\frac{1}{2}
\int_{\mathbb R^N}
\partial_h(V\eta_j)
\left(
(W_{\varepsilon,\alpha,\xi}+\omega)^2-W_{\varepsilon,\alpha,\xi}^2
\right)                                      \\
&\quad
-\int_{\mathbb R^N}
(\partial_h\eta_j)
\left(
F(W_{\varepsilon,\alpha,\xi}+\omega)-F(W_{\varepsilon,\alpha,\xi})
\right).
\end{aligned}
\]
Let
\[
A_j^+
:=
\mathcal B_{2d}^+(\xi_j)\setminus\overline{\mathcal B_d^+(\xi_j)}.
\]
By the Poisson representation, Lemma~\ref{lem:single-bubble-input}, and an argument analogous to that in the proof of Lemma~\ref{lem:extension-u-bound}(i), we obtain
\begin{equation}\label{eq:extension-w-gradient-bound}
  \varepsilon^{2s}
\int_{A_j^+}
t^{1-2s}|\nabla\overline W_{\varepsilon,\alpha,\xi}|^2
\le
C\varepsilon^{2N}.
\end{equation}
Meanwhile, 
\begin{equation}\label{eq:extension-w-gradient-bound-omega}
  \varepsilon^{2s}
\int_{\mathbb R^{N+1}_+}
t^{1-2s}|\nabla\overline\omega|^2
\le
\|\omega\|_{\varepsilon,V}^2
\le
C\varepsilon^{N+2\tau},
\end{equation}
and
\begin{equation}\label{eq:extension-w-gradient-bound-cross}
\begin{aligned}
\left|
\int_{\mathbb R^N}
\partial_h(V\eta_j)
\left((W_{\varepsilon,\alpha,\xi}+\omega)^2-W_{\varepsilon,\alpha,\xi}^2\right)
\right|
&\le
C\int_{\mathbb R^N}|W_{\varepsilon,\alpha,\xi}\omega|
+
C\int_{\mathbb R^N}\omega^2                                \\
&\le
C\|W_{\varepsilon,\alpha,\xi}\|_{L^2}\|\omega\|_{L^2}
+
C\|\omega\|_{L^2}^2                                         \\
&=
O(\varepsilon^{N+\tau}).
\end{aligned}
\end{equation}
Finally, on \(\operatorname{supp}\partial_h\eta_j\), we have
\begin{equation}\label{eq:supp-eta}
  \begin{aligned}
\left|
\int_{\mathbb R^N}
(\partial_h\eta_j)
\left(
F(W_{\varepsilon,\alpha,\xi}+\omega)
-F(W_{\varepsilon,\alpha,\xi})
\right)
\right|
&\le
C
\int_{\operatorname{supp}\partial_h\eta_j}
\left(
|W_{\varepsilon,\alpha,\xi}|^p|\omega|+|\omega|^{p+1}
\right)                                      \\
&=
o(\varepsilon^{N+\tau}).
\end{aligned}
\end{equation}
Combining \eqref{eq:extension-w-gradient-bound}, \eqref{eq:extension-w-gradient-bound-omega}, \eqref{eq:extension-w-gradient-bound-cross}, and \eqref{eq:supp-eta} gives
\[
A_{2,j,h}(\alpha,\xi,\omega)
=
B_{2,j,h}(\alpha,\xi)
+
O(\varepsilon^{N+\tau}),
\]
which implies \eqref{eq:a2-b2-stability}.
\end{proof}
Now we prove that the maps $A$ and $B$ are homotopic on $S_\varepsilon $.
\begin{proposition}\label{prop:fractional-homotopy}
Let
\[
A_t:=tA+(1-t)B,
\qquad 0\le t\le1.
\]
Then, for all sufficiently small \(\varepsilon\),
\begin{equation}\label{eq:homotopy-no-boundary-zero}
A_t(\alpha,\xi,\omega)\ne0
\qquad
\hbox{for every }(\alpha,\xi,\omega)\in\partial S_\varepsilon
\hbox{ and every }t\in[0,1].
\end{equation}
Consequently,
\begin{equation}\label{eq:degree-homotopy}
\deg(A,S_\varepsilon,0)=\deg(B,S_\varepsilon,0).
\end{equation}
\end{proposition}

\begin{proof}
Suppose, for contradiction, that
\[
A_t(\alpha,\xi,\omega)=0
\]
for some \(t\in[0,1]\) and some \((\alpha,\xi,\omega)\in\partial S_\varepsilon\).  
Set \(u_\varepsilon:=W_{\varepsilon,\alpha,\xi}+\omega\).
We divide the proof into five steps.

\emph{Step 1: an energy estimate of \(\omega\).}
For every \(\psi\in E_{\varepsilon,\xi}\), the third component of \(A_t(\alpha,\xi,\omega)=0\) gives
\[
0=t\mathcal E_\varepsilon'(u_\varepsilon)[\psi]+(1-t)\langle\omega,\psi\rangle_{\varepsilon,V}.
\]
Since
\[
\langle W_{\varepsilon,\alpha,\xi},\psi\rangle_{\varepsilon,V}=\sum_{j=1}^k\alpha_j\langle W_{\varepsilon,j}(x;\xi_j),\psi\rangle_{\varepsilon,V}=0,
\]
we have
\begin{equation}\label{eq:projected-homotopy-equation-detailed}
\begin{aligned}
\langle\omega,\psi\rangle_{\varepsilon,V}
&-tp\int_{\R^N}W_{\varepsilon,\alpha,\xi}^{p-1}\omega\psi                                      \\
&=-t\mathcal E_\varepsilon'(W_{\varepsilon,\alpha,\xi})[\psi]
+t\int_{\R^N}N_{\varepsilon,\alpha,\xi}(\omega)\psi,
\qquad \psi\in E_{\varepsilon,\xi},
\end{aligned}
\end{equation}
where
\[
{N_{\varepsilon,\alpha,\xi}(\omega):=|u_\varepsilon|^{p-1}u_\varepsilon-W_{\varepsilon,\alpha,\xi}^p-pW_{\varepsilon,\alpha,\xi}^{p-1}\omega.}
\]
Taking \(\psi=\omega\) in \eqref{eq:projected-homotopy-equation-detailed}, using Corollary~\ref{cor:homotopy-coercivity}, Lemma~\ref{lem:degree-stability-estimates}, and
\[
|N_{\varepsilon,\alpha,\xi}(\omega)|
\le C|\omega|^\kappa
\le C\norm{\omega}_{L^\infty}^{\kappa-1}|\omega|,
\qquad
\kappa:=\min\{2,p\}>1,
\]
we obtain
\begin{equation}\label{eq:omega-energy-identity-before-absorption}
\begin{aligned}
c\norm{\omega}_{\varepsilon,V}^2
&\le
\langle\omega,\omega\rangle_{\varepsilon,V}
-tp\int_{\R^N}W_{\varepsilon,\alpha,\xi}^{p-1}\omega^2                                           \\
&=-t\mathcal E_\varepsilon'(W_{\varepsilon,\alpha,\xi})[\omega]
+t\int_{\R^N}N_{\varepsilon,\alpha,\xi}(\omega)\omega                    \\
&\le
C\varepsilon^{N/2+\tau}\norm{\omega}_{\varepsilon,V}
+C\norm{\omega}_{L^\infty}^{\kappa-1}
\norm{\omega}_{L^2}^2                                                     \\
&\le
C\varepsilon^{N/2+\tau}\norm{\omega}_{\varepsilon,V}
+C\norm{\omega}_{L^\infty}^{\kappa-1}
\norm{\omega}_{\varepsilon,V}^2 .
\end{aligned}
\end{equation}
Since $ (\alpha,\xi,\omega)\in\overline{S}_\varepsilon $,
\[
\norm{\omega}_{L^\infty(\R^N)}
\le C\norm{\omega}_{\tau,\varepsilon,\xi}
\le C\varepsilon^{\delta_0}=o(1).
\]
Absorbing the last term in \eqref{eq:omega-energy-identity-before-absorption} gives
\begin{equation}\label{eq:omega-energy-small-homotopy}
\norm{\omega}_{\varepsilon,V}
\le C\varepsilon^{N/2+\tau}.
\end{equation}

\emph{Step 2: the estimate of \(\alpha\).}
By \eqref{eq:omega-energy-small-homotopy} and the defining bound \(\norm{\omega}_{\tau,\varepsilon,\xi}\le\varepsilon^{\delta_0}\), Lemma~\ref{lem:degree-stability-estimates} gives
\[
A_{1,j}(\alpha,\xi,\omega)=B_{1,j}(\alpha,\xi)+O(\varepsilon^{N+\tau}).
\]
The first component of \(A_t(\alpha,\xi,\omega)=0\) gives
\begin{equation*}
\begin{aligned}
0
&=tA_{1,j}+(1-t)B_{1,j}                                                   \\
&=t\bigl(B_{1,j}+O(\varepsilon^{N+\tau})\bigr)+(1-t)B_{1,j}              \\
&=B_{1,j}+O(\varepsilon^{N+\tau}),
\end{aligned}
\end{equation*}
and therefore
\[
B_{1,j}(\alpha,\xi)=O(\varepsilon^{N+\tau}).
\]
By Lemma~\ref{lem:alpha-equation},
\begin{equation*}
\begin{aligned}
a_0\varepsilon^NV(\xi_j)^\theta(\alpha_j-\alpha_j^p)
&=B_{1,j}(\alpha,\xi)+O(\varepsilon^{N+\tau})\\
&=O(\varepsilon^{N+\tau}).
\end{aligned}
\end{equation*}
Since \(V\) is bounded away from zero and \(\alpha_j\to1\) in \(S_\varepsilon\),
\[
\begin{aligned}
|\alpha_j-\alpha_j^p|
&\le C\varepsilon^{\tau},                                                  \\
\alpha_j-\alpha_j^p
&=(1-p)(\alpha_j-1)+O(|\alpha_j-1|^2),                                      \\
|(1-p)+O(|\alpha_j-1|)|\,|\alpha_j-1|
&\le C\varepsilon^{\tau}.
\end{aligned}
\]
Thus, for sufficiently small \(\varepsilon\),
\begin{equation}\label{eq:alpha-homotopy-interior-detailed}
|\alpha_j-1|
\le C\varepsilon^{\tau},
\qquad j=1,\ldots,k.
\end{equation}
Because \(\delta_0<\tau\), this estimate places $\alpha$ strictly inside \(|\alpha_j-1|<\varepsilon^{\delta_0}\).

\emph{Step 3: a weighted estimate of \(\omega\).}
By the Lagrange multiplier rule applied to \eqref{eq:projected-homotopy-equation-detailed}, there exist \(\beta_j\) and \(\gamma_{jh}\) such that the following identity holds weakly in \(\R^N\),
\begin{equation}\label{eq:omega-full-equation-with-multipliers}
\begin{aligned}
\varepsilon^{2s}(-\Delta)^s\omega+V(x)\omega-tpW_{\varepsilon,\alpha,\xi}^{p-1}\omega
&=-tG_\varepsilon+tN_{\varepsilon,\alpha,\xi}(\omega)
+\sum_{j=1}^k\beta_jR_j
+\sum_{j=1}^k\sum_{h=1}^N\gamma_{jh}S_{j,h},
\end{aligned}
\end{equation}
where
\[
\begin{aligned}
G_\varepsilon&:=\varepsilon^{2s}(-\Delta)^sW_{\varepsilon,\alpha,\xi}
+V(x)W_{\varepsilon,\alpha,\xi}-W_{\varepsilon,\alpha,\xi}^p,\\
R_j&:=\varepsilon^{2s}(-\Delta)^sW_{\varepsilon,j}(x;\xi_j)
+V(x)W_{\varepsilon,j}(x;\xi_j)\\
&=W_{\varepsilon,j}(x;\xi_j)^p+(V(x)-V(\xi_j))W_{\varepsilon,j}(x;\xi_j),\\
S_{j,h}&:=\varepsilon^{2s}(-\Delta)^sY_{\varepsilon,j,h}(x;\xi_j)
+V(x)Y_{\varepsilon,j,h}(x;\xi_j)\\
&=pW_{\varepsilon,j}(x;\xi_j)^{p-1}Y_{\varepsilon,j,h}(x;\xi_j)
-\partial_hV(\xi_j)W_{\varepsilon,j}(x;\xi_j)\\
&\quad +(V(x)-V(\xi_j))Y_{\varepsilon,j,h}(x;\xi_j),
\end{aligned}
\]
By Lemma~\ref{lem:weighted-tail} and
\eqref{eq:alpha-homotopy-interior-detailed},
\begin{equation}\label{eq:residual-pointwise-homotopy}
|G_\varepsilon(x)|
\le C\varepsilon^\tau\rho_{\tau,\varepsilon,\xi}(x).
\end{equation}
The nonlinear remainder satisfies
\begin{equation}\label{eq:nonlinear-pointwise-homotopy}
\begin{aligned}
|N_{\varepsilon,\alpha,\xi}(\omega)(x)|
&\le C|\omega(x)|^\kappa                                            \\
&\le C\norm{\omega}_{\tau,\varepsilon,\xi}^{\kappa-1}
\norm{\omega}_{\tau,\varepsilon,\xi}\rho_{\tau,\varepsilon,\xi}(x)^\kappa \\
&\le C\varepsilon^{(\kappa-1)\delta_0}
\norm{\omega}_{\tau,\varepsilon,\xi}\rho_{\tau,\varepsilon,\xi}(x).
\end{aligned}
\end{equation}
Set
\[
M_\varepsilon:=\varepsilon^\tau
+\varepsilon^{(\kappa-1)\delta_0}\norm{\omega}_{\tau,\varepsilon,\xi}.
\]
Then \eqref{eq:residual-pointwise-homotopy} and \eqref{eq:nonlinear-pointwise-homotopy} imply
\begin{equation}\label{eq:combined-pointwise-homotopy}
|-tG_\varepsilon+tN_{\varepsilon,\alpha,\xi}(\omega)|
\le CM_\varepsilon\rho_{\tau,\varepsilon,\xi}(x).
\end{equation}

We next estimate the multipliers.  Testing \eqref{eq:omega-full-equation-with-multipliers}
against \(W_{\varepsilon,i}(x;\xi_i)\) and using
\[
\langle\omega,W_{\varepsilon,i}(x;\xi_i)\rangle_{\varepsilon,V}=0,
\]
we obtain
\[
\begin{aligned}
&\sum_{j=1}^k\beta_j\langle W_{\varepsilon,j}(x;\xi_j),W_{\varepsilon,i}(x;\xi_i)\rangle_{\varepsilon,V}
+\sum_{j=1}^k\sum_{h=1}^N\gamma_{jh}\langle Y_{\varepsilon,j,h}(x;\xi_j),W_{\varepsilon,i}(x;\xi_i)\rangle_{\varepsilon,V} \\
&\qquad
=-tp\int_{\R^N}W_{\varepsilon,\alpha,\xi}^{p-1}\omega W_{\varepsilon,i}(x;\xi_i)
+t\int_{\R^N}G_\varepsilon W_{\varepsilon,i}(x;\xi_i)
-t\int_{\R^N}N_{\varepsilon,\alpha,\xi}(\omega)W_{\varepsilon,i}(x;\xi_i) .
\end{aligned}
\]
Moreover,
\[
\begin{aligned}
\left|\int_{\R^N}W_{\varepsilon,\alpha,\xi}^{p-1}\omega W_{\varepsilon,i}(x;\xi_i)\right|
&\le \norm{W_{\varepsilon,\alpha,\xi}^{p-1}W_{\varepsilon,i}(x;\xi_i)}_{L^2}\norm{\omega}_{L^2}
\le C\varepsilon^{N/2}\varepsilon^{N/2+\tau}                             \\
&=C\varepsilon^{N+\tau}
\le C\varepsilon^NM_\varepsilon,
\end{aligned}
\]
and, since \(\int_{\R^N}\rho_{\tau,\varepsilon,\xi} W_{\varepsilon,i}(x;\xi_i)\le C\varepsilon^N\),
\[
\begin{aligned}
\left|\int_{\R^N}G_\varepsilon W_{\varepsilon,i}(x;\xi_i)\right|
+\left|\int_{\R^N}N_{\varepsilon,\alpha,\xi}(\omega)W_{\varepsilon,i}(x;\xi_i)\right|
&\le CM_\varepsilon\int_{\R^N}\rho_{\tau,\varepsilon,\xi} W_{\varepsilon,i}(x;\xi_i)                 \\
&\le C\varepsilon^NM_\varepsilon.
\end{aligned}
\]
Using \eqref{eq:app-same-ww}, \eqref{eq:app-same-wy}, and
\eqref{eq:app-diff-ww}--\eqref{eq:app-diff-wy}, this gives
\begin{equation}\label{eq:beta-system-homotopy}
\varepsilon^N\bigl(a_0V(\xi_i)^\theta+o(1)\bigr)\beta_i
+o(\varepsilon^N)\sum_{j=1}^k|\beta_j|
+o(\varepsilon^{N-1})\sum_{j=1}^k\sum_{h=1}^N|\gamma_{jh}|
=O(\varepsilon^NM_\varepsilon).
\end{equation}
Testing \eqref{eq:omega-full-equation-with-multipliers} against \(Y_{\varepsilon,i,\ell}(x;\xi_i)\) and using \(\langle\omega,Y_{\varepsilon,i,\ell}(x;\xi_i)\rangle_{\varepsilon,V}=0\), we similarly get
\[
\begin{aligned}
&\sum_{j=1}^k\beta_j\langle W_{\varepsilon,j}(x;\xi_j),Y_{\varepsilon,i,\ell}(x;\xi_i)\rangle_{\varepsilon,V}
+\sum_{j=1}^k\sum_{h=1}^N\gamma_{jh}\langle Y_{\varepsilon,j,h}(x;\xi_j),Y_{\varepsilon,i,\ell}(x;\xi_i)\rangle_{\varepsilon,V} \\
&
\qquad
=-tp\int_{\R^N}W_{\varepsilon,\alpha,\xi}^{p-1}\omega Y_{\varepsilon,i,\ell}(x;\xi_i)
+t\int_{\R^N}G_\varepsilon Y_{\varepsilon,i,\ell}(x;\xi_i)
-t\int_{\R^N}N_{\varepsilon,\alpha,\xi}(\omega)Y_{\varepsilon,i,\ell}(x;\xi_i).
\end{aligned}
\]
Since
\[
\norm{W_{\varepsilon,\alpha,\xi}^{p-1}Y_{\varepsilon,i,\ell}(x;\xi_i)}_{L^2}
\le C\varepsilon^{N/2-1},
\qquad
\int_{\R^N}\rho_{\tau,\varepsilon,\xi}|Y_{\varepsilon,i,\ell}(x;\xi_i)|
\le C\varepsilon^{N-1},
\]
we have
\[
\begin{aligned}
\left|\int_{\R^N}W_{\varepsilon,\alpha,\xi}^{p-1}\omega Y_{\varepsilon,i,\ell}(x;\xi_i)\right|
&\le C\varepsilon^{N/2-1}\varepsilon^{N/2+\tau}
=C\varepsilon^{N-1+\tau}
\le C\varepsilon^{N-1}M_\varepsilon,                                      \\
\left|\int_{\R^N}G_\varepsilon Y_{\varepsilon,i,\ell}(x;\xi_i)\right|
+\left|\int_{\R^N}N_{\varepsilon,\alpha,\xi}(\omega)Y_{\varepsilon,i,\ell}(x;\xi_i)\right|
&\le CM_\varepsilon\int_{\R^N}\rho_{\tau,\varepsilon,\xi}|Y_{\varepsilon,i,\ell}(x;\xi_i)|          \\
&\le C\varepsilon^{N-1}M_\varepsilon.
\end{aligned}
\]
Therefore, using \eqref{eq:app-same-wy}, \eqref{eq:app-same-yy}, and
\eqref{eq:app-diff-wy}--\eqref{eq:app-diff-yy},
\begin{equation}\label{eq:gamma-system-homotopy}
o(\varepsilon^{N-1})\sum_{j=1}^k|\beta_j|
+\varepsilon^{N-2}(K_i+o(1))\gamma_{i\ell}
+o(\varepsilon^{N-2})\sum_{j=1}^k\sum_{h=1}^N|\gamma_{jh}|
=O(\varepsilon^{N-1}M_\varepsilon).
\end{equation}
Let
\[
\widetilde\gamma_{jh}:=\varepsilon^{-1}\gamma_{jh}.
\]
Dividing \eqref{eq:beta-system-homotopy} by \(\varepsilon^N\) and \eqref{eq:gamma-system-homotopy} by \(\varepsilon^{N-1}\), we obtain
\[
\begin{aligned}
\bigl(a_0V(\xi_i)^\theta+o(1)\bigr)\beta_i
+o(1)\sum_{j=1}^k|\beta_j|
+o(1)\sum_{j=1}^k\sum_{h=1}^N|\widetilde\gamma_{jh}|
&=O(M_\varepsilon),                                                        \\
o(1)\sum_{j=1}^k|\beta_j|
+(K_i+o(1))\widetilde\gamma_{i\ell}
+o(1)\sum_{j=1}^k\sum_{h=1}^N|\widetilde\gamma_{jh}|
&=O(M_\varepsilon).
\end{aligned}
\]
Since \(a_0V(\xi_i)^\theta\) and \(K_i\) are uniformly positive, the finite-dimensional matrix is invertible for sufficiently small \(\varepsilon\). Hence
\begin{equation}\label{eq:multiplier-bound-homotopy}
|\beta_j|\le CM_\varepsilon,
\qquad
|\gamma_{jh}|\le C\varepsilon M_\varepsilon.
\end{equation}
On the other hand,
\[
|R_j(x)|\le C\rho_{\tau,\varepsilon,\xi}(x),
\qquad
|S_{j,h}(x)|\le C\varepsilon^{-1}\rho_{\tau,\varepsilon,\xi}(x).
\]
Combining these estimates with \eqref{eq:combined-pointwise-homotopy} and \eqref{eq:multiplier-bound-homotopy}, the right-hand side of \eqref{eq:omega-full-equation-with-multipliers} satisfies
\begin{equation}\label{eq:full-right-hand-side-pointwise-homotopy}
\left|-tG_\varepsilon+tN_{\varepsilon,\alpha,\xi}(\omega)
+\sum_{j=1}^k\beta_jR_j
+\sum_{j=1}^k\sum_{h=1}^N\gamma_{jh}S_{j,h}\right|
\le CM_\varepsilon\rho_{\tau,\varepsilon,\xi}(x).
\end{equation}

Applying Lemma~\ref{lem:weighted-norm-estimate} to \eqref{eq:omega-full-equation-with-multipliers} and using \eqref{eq:full-right-hand-side-pointwise-homotopy}, we obtain
\[
\norm{\omega}_{\tau,\varepsilon,\xi}
\le C\left(\varepsilon^\tau
+\varepsilon^{(\kappa-1)\delta_0}\norm{\omega}_{\tau,\varepsilon,\xi}\right).
\]
For sufficiently small \(\varepsilon\), the second term is absorbed, and hence
\begin{equation}\label{eq:omega-homotopy-weighted-interior}
\norm{\omega}_{\tau,\varepsilon,\xi}\le C\varepsilon^\tau.
\end{equation}

\emph{Step 4: the estimate of \(\xi\).}
By \eqref{eq:omega-energy-small-homotopy} and the defining bound \(\norm{\omega}_{\tau,\varepsilon,\xi}\le\varepsilon^{\delta_0}\), Lemma~\ref{lem:degree-stability-estimates} also gives
\[
A_{2,j,h}(\alpha,\xi,\omega)=B_{2,j,h}(\alpha,\xi)+O(\varepsilon^{N+\tau}).
\]
The second component of \(A_t(\alpha,\xi,\omega)=0\) gives
\begin{equation*}
\begin{aligned}
0
&=tA_{2,j,h}+(1-t)B_{2,j,h}                                                \\
&=t\bigl(B_{2,j,h}+O(\varepsilon^{N+\tau})\bigr)+(1-t)B_{2,j,h}           \\
&=B_{2,j,h}+O(\varepsilon^{N+\tau}),
\end{aligned}
\end{equation*}
so that
\[
B_{2,j,h}(\alpha,\xi)=O(\varepsilon^{N+\tau}).
\]
Using \eqref{eq:pohozaev-xi-main}, we obtain
\begin{equation}\label{eq:translation-small-homotopy-detailed}
\begin{aligned}
c_*\theta\varepsilon^NV(\xi_j)^{\theta-1}\partial_hV(\xi_j)
&=B_{2,j,h}(\alpha,\xi)
+O(\varepsilon^N|\alpha_j-1|)
+O(\varepsilon^{N+\tau})                                                   \\
&=O(\varepsilon^{N+\tau})
+O(\varepsilon^N|\alpha_j-1|)
+O(\varepsilon^{N+\tau})                                                   \\
&=O(\varepsilon^{N+\tau}),
\end{aligned}
\end{equation}
where the last equality uses \eqref{eq:alpha-homotopy-interior-detailed}.  Since \(c_*\theta V(\xi_j)^{\theta-1}\) is uniformly positive and bounded, \eqref{eq:translation-small-homotopy-detailed} gives
\begin{equation*}
|\nabla V(\xi_j)|
\le C\varepsilon^{\tau},
\qquad j=1,\ldots,k.
\end{equation*}
Since \(\nabla V(\xi_j^0)=0\) and \(D^2V(\xi_j^0)\) is invertible, after choosing \(\tau_0\) small there is \(c>0\) such that
\[
|\nabla V(\xi)|
\ge c|\xi-\xi_j^0|
\qquad\hbox{for }|\xi-\xi_j^0|<\tau_0.
\]
Therefore
\begin{equation}\label{eq:xi-homotopy-interior-detailed}
|\xi_j-\xi_j^0|
\le C|\nabla V(\xi_j)|
\le C\varepsilon^{\tau},
\qquad j=1,\ldots,k.
\end{equation}
For sufficiently small \(\varepsilon\), this gives \(|\xi_j-\xi_j^0|<\tau_0\) strictly.

\emph{Step 5: conclusion.}
The four boundary alternatives in \(\partial S_\varepsilon\) are
\[
|\alpha_j-1|=\varepsilon^{\delta_0},
\qquad
|\xi_j-\xi_j^0|=\tau_0,
\qquad
\norm{\omega}_{\varepsilon,V}=\varepsilon^{N/2+\delta_0},
\qquad
\norm{\omega}_{\tau,\varepsilon,\xi}=\varepsilon^{\delta_0}.
\]
But \eqref{eq:omega-energy-small-homotopy}, \eqref{eq:omega-homotopy-weighted-interior}, \eqref{eq:alpha-homotopy-interior-detailed}, and \eqref{eq:xi-homotopy-interior-detailed} give
\[
\begin{aligned}
\norm{\omega}_{\varepsilon,V}
&\le C\varepsilon^{N/2+\tau}
<\varepsilon^{N/2+\delta_0},                                                   \\
\norm{\omega}_{\tau,\varepsilon,\xi}
&\le C\varepsilon^\tau
<\varepsilon^{\delta_0},                                                        \\
|\alpha_j-1|
&\le C\varepsilon^{\tau}
<\varepsilon^{\delta_0},                                                        \\
|\xi_j-\xi_j^0|
&\le C\varepsilon^{\tau}
<\tau_0,
\end{aligned}
\]
for sufficiently small \(\varepsilon\), because
\[
0<\delta_0<\tau.
\]
Hence \((\alpha,\xi,\omega)\) is strictly inside \(S_\varepsilon\), contradicting \((\alpha,\xi,\omega)\in\partial S_\varepsilon\).  Therefore \eqref{eq:homotopy-no-boundary-zero} holds.  Homotopy invariance of the degree gives \eqref{eq:degree-homotopy}.
\end{proof}

\begin{proof}[Proof of Theorem~\ref{thm:fractional-degree-formula}]
By Proposition~\ref{prop:fractional-homotopy},
\begin{equation}\label{eq:degree-homotopy-reduction}
\deg(A,S_\varepsilon,0)=\deg(B,S_\varepsilon,0).
\end{equation}
The third component of \(B\) is the identity on
\(E_{\varepsilon, \xi }\), so it contributes degree one.  Hence
\begin{equation}\label{eq:b-map-degree-reduction}
\deg(B,S_\varepsilon,0)=\deg(\overline B,\mathcal D_\varepsilon,0),
\qquad
\overline B(\alpha,\xi):=\bigl(B_{1,j}(\alpha,\xi),B_{2,j,h}(\alpha,\xi)\bigr).
\end{equation}
By Lemma~\ref{lem:alpha-equation} and the expansion
\eqref{eq:pohozaev-xi-main} for
\(\mathcal P_{j,h}(W_{\varepsilon,\alpha,\xi})\), we can further deform
\(\overline B\) in \(S_\varepsilon \) to
\[
\widetilde B(\alpha,\xi):=
\left(
\bigl(a_0\varepsilon^NV(\xi_j)^\theta(\alpha_j-\alpha_j^p)\bigr)_{j=1}^k,
\bigl(c_*\theta\varepsilon^NV(\xi_j)^{\theta-1}\nabla V(\xi_j)\bigr)_{j=1}^k
\right),
\]
so that
\begin{equation}\label{eq:reduced-map-degree-homotopy}
\deg(\overline B,\mathcal D_\varepsilon,0)=\deg(\widetilde B,\mathcal D_\varepsilon,0).
\end{equation}
A direct computation gives
\begin{equation}\label{eq:leading-map-degree}
\deg(\widetilde B,\mathcal D_\varepsilon,0)
=(-1)^k(-1)^{m(\xi^0,\mathcal V)}
=(-1)^{k+m(\xi^0,\mathcal V)}.
\end{equation}
Combining \eqref{eq:degree-homotopy-reduction}, \eqref{eq:b-map-degree-reduction},
\eqref{eq:reduced-map-degree-homotopy}, and \eqref{eq:leading-map-degree} proves
\eqref{eq:fractional-degree-formula}. 
\end{proof}

\begin{proof}[Proof of Theorem~\ref{th:degree-counting}]
Theorem~\ref{th:degree-counting} follows directly from Proposition~\ref{prop:fractional-a-equivalence} and Theorem~\ref{thm:fractional-degree-formula}.
\end{proof}

\section{Proof of the local uniqueness}\label{sec:local-uniqueness}

We now derive local uniqueness from the Morse index formula and the degree formula obtained above.  

\begin{proof}[Proof of Theorem~\ref{th:local-uniqueness}]
By Theorem~\ref{th:nondegeneracy}, every positive \(k\)-peak solution satisfying Definition~\ref{def:kpeak-class} is nondegenerate for sufficiently small \(\varepsilon\). Hence the set of solutions in the class is finite; denote the number of such solutions by $q$.

By Theorem~\ref{th:morse-index}, every solution in the class has the same Morse index
\(
m(u_\varepsilon)=k+\sum_{j=1}^k m(\xi_j^0,V).
\)
Thus its degree is \( (-1)^{k+\sum_{j=1}^k m(\xi_j^0,V)} \). It follows that the total degree of all such solutions is
\(
q(-1)^{k+m(\xi^0,\mathcal V)}.
\)

On the other hand, the degree formula in Theorem~\ref{th:degree-counting} shows that the total degree in this class is
\(
(-1)^{k+m(\xi^0,\mathcal V)}.
\)
Therefore $q=1$. 
\end{proof}

\appendix
\section{Proof of Lemma~\ref{lem:single-bubble-input} and Lemma~\ref{lem:true-peak-profile-convergence}}\label{app:proof-profile-lemmas}
\begin{proof}[Proof of Lemma~\ref{lem:single-bubble-input}]
(i) Assertion (i) follows from Theorem~\ref{th:ground-state} and the scaling relation between $w_\lambda $ and $w$; see also \cite[Section 3]{frank2016uniqueness}.

(ii) The $C^1$ dependence of $\lambda\mapsto w_\lambda$ follows from the explicit scaling formula. By Theorem~\ref{th:ground-state} and the compactness of $[V_{\min},V_{\max}]$, we have 
\[
c(1+|y|)^{-N-2s}\le w_\lambda(y)\le C(1+|y|)^{-N-2s},
\qquad \lambda\in[V_{\min},V_{\max}].
\]

We first record a consequence of Appendix~C in \cite{frank2016uniqueness}.  Let $G_{s,\mu}$ be the kernel of $((-\Delta)^s+\mu)^{-1}$, $\mu>0$.  By \cite[Lemma~C.1]{frank2016uniqueness},
\[
G_{s,\mu}(x)\le C(1+|x|)^{-N-2s}\quad (|x|\ge1),
\]
and if $0\le f(x)\le C(1+|x|)^{-N-2s}$, then by an argument similar to that in the proof of \cite[Lemma~C.3]{frank2016uniqueness}, we have
\begin{equation}\label{eq:green-kernel-convolution-decay}
  G_{s,\mu}*f\le C(1+|x|)^{-N-2s}.
\end{equation}
Suppose that $z\in C(\R^N)\cap L^\infty(\R^N)$, $z(x)\to0$ as $|x|\to\infty$, and
\[
(-\Delta)^sz+z-a(x)z=f\quad\text{in }\R^N,
\]
where $a\ge0$, $a(x)\to0$ as $|x|\to\infty$, and
\[
|f(x)|\le C(1+|x|)^{-N-2s},
\]
we claim that 
\begin{equation}\label{eq:app-comparison-decay}
|z(x)|\le C(1+|x|)^{-N-2s}.
\end{equation}
Indeed, choose $R>0$ such that
\[
a(x)\le\frac{1}{2},\qquad |x|\ge R.
\]
The Kato inequality \[
(-\Delta)^s|z|\le(\operatorname{sgn}z)(-\Delta)^s z,
\]
gives, in $\R^N\setminus B_R$,
\[
(-\Delta)^s|z|+\frac{1}{2}|z|\le |f|.
\]
Let $h:=G_{s,1/2}*|f|$.  Then, by \eqref{eq:green-kernel-convolution-decay},
\[
(-\Delta)^s h+\frac{1}{2}h=|f|,
\qquad
h(x)\le C(1+|x|)^{-N-2s}.
\]

Choose $M>0$ so large that
\[
|z|\le h+MG_{s,1/2}\quad\text{in }B_R;
\]
this is possible because $G_{s,1/2}$ is positive and has a positive minimum on compact annuli, while it is singular at the origin.  Set
\[
U:=|z|-h-MG_{s,1/2}.
\]

Then
\begin{equation}\label{eq:comparison-function-nonpositive}
  (-\Delta)^sU+\frac{1}{2} U\le0
\quad\text{in }\mathbb R^N\setminus B_R,
\end{equation}
with \(U\le0\) in \(B_R\) and \(U(x)\to0\) as \(|x|\to\infty\).  If \(U\) were positive somewhere, then it would attain
a positive maximum at some point \(x_0\) in the exterior region.  At this point \(U(x_0)>0\),
\[
  (-\Delta)^sU(x_0)
  =C_{N,s}\,{\rm P.V.}\!\int_{\mathbb R^N}
  \frac{U(x_0)-U(y)}{|x_0-y|^{N+2s}}\,dy\ge0 .
  \]
Hence
\[
(-\Delta)^sU(x_0)+\frac{1}{2} U(x_0)>0,
\]
contradicting \eqref{eq:comparison-function-nonpositive}.  Therefore \(U\le0\) in
\(\mathbb R^N\). Using \eqref{eq:green-kernel-convolution-decay}, the decay of \(G_{s,1/2}\) and the boundedness of $z$, we obtain
\eqref{eq:app-comparison-decay}.

For the first derivatives, set
\[
z_i:=\partial_iw.
\]
Differentiating
\[
(-\Delta)^sw+w=w^p
\]
gives
\[
(-\Delta)^sz_i+z_i-pw^{p-1}z_i=0.
\]
Applying \eqref{eq:app-comparison-decay} with $a=pw^{p-1}$ and $f=0$ yields
\[
|\nabla w(y)|\le C(1+|y|)^{-N-2s}.
\]
By scaling,
\begin{equation}\label{eq:ground-state-gradient-decay}
  |\nabla w_\lambda(y)|\le C(1+|y|)^{-N-2s},
\qquad \lambda\in[V_{\min},V_{\max}].
\end{equation}

For the $\lambda$-derivative, define
\[
\eta(y):=\frac{1}{p-1}w(y)+\frac{1}{2s}y\cdot\nabla w(y)
=\partial_\lambda w_\lambda(y)\big|_{\lambda=1}.
\]
Differentiating
\[
(-\Delta)^sw_\lambda+\lambda w_\lambda-w_\lambda^p=0
\]
at $\lambda=1$ gives
\[
(-\Delta)^s\eta+\eta-pw^{p-1}\eta=-w.
\]
Since $w(y)\le C(1+|y|)^{-N-2s}$, \eqref{eq:app-comparison-decay} gives
\[
|\eta(y)|\le C(1+|y|)^{-N-2s}.
\]
Using
\[
\partial_\lambda w_\lambda(y)=\lambda^{1/(p-1)-1}\eta(\lambda^{1/(2s)}y),
\]
we obtain
\[
|\partial_\lambda w_\lambda(y)|\le C(1+|y|)^{-N-2s},
\qquad \lambda\in[V_{\min},V_{\max}].
\]

For the second derivatives, set
\[
z_{ij}:=\partial_{ij}w.
\]
Differentiating the equation twice yields
\[
(-\Delta)^sz_{ij}+z_{ij}-pw^{p-1}z_{ij}
=p(p-1)w^{p-2}\partial_iw\partial_jw.
\]
The two-sided decay of $w$ and the first-derivative estimate imply
\[
\left|w^{p-2}\partial_iw\partial_jw\right|
\le C(1+|y|)^{-p(N+2s)}
\le C(1+|y|)^{-N-2s},
\]
because $p>1$.  Applying \eqref{eq:app-comparison-decay} again gives
\[
|D^2w(y)|\le C(1+|y|)^{-N-2s}.
\]
Finally,
\[
D^2w_\lambda(y)=\lambda^{1/(p-1)+1/s}D^2w(\lambda^{1/(2s)}y),
\]
and therefore
\begin{equation}\label{eq:ground-state-hessian-decay}
  |D^2w_\lambda(y)|\le C(1+|y|)^{-N-2s},
\qquad \lambda\in[V_{\min},V_{\max}].
\end{equation}

It remains to prove that the maximum at the origin is nondegenerate.  We first prove that
\begin{equation}\label{eq:app-hessian-w1}
D^2w(0)=-\kappa_1I_N
\qquad\text{for some }\kappa_1>0.
\end{equation}
Let
\[
F(y):=w(y)^p.
\]
Then $F=F(|y|)$ is positive, continuous, strictly decreasing in $|y|$, and tends to zero at infinity.  Since
\[
(-\Delta)^sw+w=F,
\]
we have
\[
w=((-\Delta)^s+1)^{-1}F
=\int_0^\infty e^{-t}P_t^sF\,dt,
\qquad P_t^s:=e^{-t(-\Delta)^s}.
\]
By the subordination formula used in \cite[Lemma~C.1 and its proof]{frank2016uniqueness}, the fractional heat kernel $p_s(t,x)=p_s(t,|x|)$ of $P_t^s$ is positive, radial, smooth for $t>0$, and strictly decreasing in $|x|$.  Hence
\[
\partial_rp_s(t,r)<0,
\qquad t>0,\quad r>0.
\]

For every $t>0$ and every unit vector $e\in\mathbb S^{N-1}$, we claim that
\[
\partial_{ee}(P_t^sF)(0)<0.
\]
Indeed, \(P_t^sF(x)=\int_{\mathbb R^N}p_s(t,x-y)F(y)\,dy\) is radial and smooth. Thus
\[
\partial_{ee}(P_t^sF)(0)
=\frac{1}N\Delta(P_t^sF)(0)
=\frac{1}N\int_{\R^N}\Delta p_s(t,y)F(y)\,dy.
\]
Since \(F\) is positive, continuous, and strictly decreasing, for each $\tau\in(0,F(0))$ there exists $R_\tau>0$ such that
\[
\{y\in\R^N:F(y)>\tau\}=B_{R_\tau},
\qquad
F(y)=\int_0^{F(0)}{\chi }_{B_{R_\tau}}(y)\,d\tau.
\]
Using Fubini's theorem and the divergence theorem,
\[
\begin{aligned}
  \int_{\R^N}\Delta p_s(t,y)F(y)\,dy&=\int_0^{F(0)}\int_{B_{R_\tau}}\Delta p_s(t,y)\,dy\,d\tau\\
&=|\mathbb S^{N-1}|\int_0^{F(0)}R_\tau^{N-1}\partial_rp_s(t,R_\tau)\,d\tau<0.
\end{aligned}
\]

By Theorem~\ref{th:ground-state}, \(w>0\) and
\(w\in C^\infty(\mathbb R^N)\). Hence \(F\in C^\infty(\mathbb R^N)\).
Moreover, by \eqref{eq:ground-state-gradient-decay} and \eqref{eq:ground-state-hessian-decay}, we obtain 
\[
|D^2F(y)|\le C(1+|y|)^{-p(N+2s)}.
\]
In particular, \(D^2F\in L^\infty(\mathbb R^N)\).
Since
\[
w=((-\Delta)^s+1)^{-1}F
=
\int_0^\infty e^{-t}P_t^sF\,dt,
\]
we may differentiate under the integral sign:
\[
\partial_{ee}w(0)
=
\int_0^\infty e^{-t}\partial_{ee}(P_t^sF)(0)\,dt.
\]
Indeed,
\[
\partial_{ee}(P_t^sF)(0)=P_t^s(\partial_{ee}F)(0),
\qquad
|\partial_{ee}(P_t^sF)(0)|\le \|\partial_{ee}F\|_{L^\infty}.
\]
Therefore
\[
\partial_{ee}w(0)<0.
\]
By radiality, $D^2w(0)=\mu I_N$ for some $\mu\in\R$.  The previous inequality gives $\mu<0$, and \eqref{eq:app-hessian-w1} follows with $\kappa_1=-\mu$.

After scaling, this gives
\[
D^2w_\lambda(0)
=\lambda^{1/(p-1)+1/s}D^2w(0)
=-\kappa_\lambda I_N,
\qquad
\kappa_\lambda:=\lambda^{1/(p-1)+1/s}\kappa_1.
\]
Since $\lambda\in[V_{\min},V_{\max}]$,
\[
0<\kappa_*\le \kappa_\lambda\le \kappa^*
\qquad \text{for all }\lambda\in[V_{\min},V_{\max}].
\]
This proves (ii).
\end{proof}

\begin{proof}[Proof of Lemma~\ref{lem:true-peak-profile-convergence}]
Fix $j\in\{1,\ldots,k\}$. Definition~\ref{def:kpeak-class} together with the fractional Brezis--Kato type estimate in \cite[Proposition~4.5]{DuarteSouto2019} gives a uniform $L^\infty$ bound. The regularity estimates for fractional equations \cite{cabre2014nonlinear,silvestre2007regularity}, applied to
\[
(-\Delta)^s u_{\varepsilon,j}+V(\xi_{\varepsilon,j}+\varepsilon y)u_{\varepsilon,j}=u_{\varepsilon,j}^p,
\]
imply that, along a subsequence,
\[
u_{\varepsilon,j}\to u_j\quad\hbox{in }C^1_{\loc}(\R^N),
\]
and the limit satisfies
\begin{equation}\label{eq:appendix-limit-equation}
(-\Delta)^su_j+V(\xi_j^0)u_j=u_j^p
\qquad\hbox{in }\R^N.
\end{equation}
The nonvanishing condition in Definition~\ref{def:kpeak-class} gives $u_j(0)>0$ after passing to the limit, hence $u_j\not\equiv0$.

We now prove that $u_j$ is a ground state.  Testing \eqref{eq:fractional-schrodinger-problem} with $u_\varepsilon$ gives
\[
\|u_\varepsilon\|_{\varepsilon,V}^2=\int_{\R^N}u_\varepsilon^{p+1}\,dx,
\qquad
\mathcal E_\varepsilon(u_\varepsilon)=\frac{p-1}{2(p+1)}\int_{\R^N}u_\varepsilon^{p+1}\,dx.
\]
Fatou's lemma and \eqref{eq:appendix-limit-equation} yield
\[
\sum_{j=1}^k c_{V(\xi_j^0)}
=\lim_{\varepsilon\to0}\varepsilon^{-N}\mathcal E_\varepsilon(u_\varepsilon)
\ge \sum_{j=1}^k\mathcal I_{V(\xi_j^0)}(u_j)
\ge \sum_{j=1}^k c_{V(\xi_j^0)}.
\]
Therefore $\mathcal I_{V(\xi_j^0)}(u_j)=c_{V(\xi_j^0)}$ for every $j$, and $u_j$ is a ground state solution.  By Theorem~\ref{th:ground-state}, $u_j$ is a translate of $w_{V(\xi_j^0)}$.  Since $y=0$ is a local maximum of every $u_{\varepsilon,j}$, it is a local maximum of $u_j$; strict radial monotonicity therefore places the translation center at the origin.  Hence
\[
u_j=w_{V(\xi_j^0)}.
\]
Every convergent subsequence has this same limit, so the convergence holds for the full sequence.  This proves \eqref{eq:true-peak-profile-convergence}.
\end{proof}

\section{Auxiliary estimates}\label{sec:appendix}

\begin{lemma}\label{lem:appendix-energy-expansions}
The following expansions hold:
\[
\begin{aligned}
  \varepsilon^{2s}&\!\int_{\mathbb{R}^{N+1}_+}\! t^{1-2s}
\Big(|\nabla\bar f_\varepsilon|^2+|\nabla\bar g_\varepsilon|^2+|\nabla\bar h_\varepsilon|^2\Big)\,dX
+\int_{\mathbb{R}^N} V(x)\big(f_\varepsilon^2+g_\varepsilon^2+h_\varepsilon^2\big)\,dx\\={}&\sum_{j=1}^k\alpha_{\varepsilon,j}^2\Bigg[
\int_{\mathbb{R}^N}u_\varepsilon^{p+1}\phi_j^2\,dx
+\varepsilon^{2s}\!\int_{\mathbb{R}^{N+1}_+}\!t^{1-2s}|\nabla\Phi_j|^2\,\bar u_\varepsilon^{\,2}\,dX\Bigg]\\
&+\sum_{m=1}^k\Bigg[
p\int_{\mathbb{R}^N}u_\varepsilon^{p-1}\phi_m^2\,z_{\varepsilon,m}^2\,dx
+\varepsilon^{2s}\!\int_{\mathbb{R}^{N+1}_+}\!t^{1-2s}|\nabla\Phi_m|^2\,\bar z_{\varepsilon,m}^{\,2}\,dX\Bigg]\\
&+\sum_{r=1}^i\Bigg[
p\int_{\mathbb{R}^N}u_\varepsilon^{p-1}\phi_r^2\,\hat z_{\varepsilon,r}^2\,dx
+\varepsilon^{2s}\!\int_{\mathbb{R}^{N+1}_+}\!t^{1-2s}|\nabla\Phi_r|^2\,\overline{\hat z}_{\varepsilon,r}^{\,2}\,dX\Bigg]\\
&-\sum_{m=1}^k\sum_{q=1}^{j-1}\beta_{\varepsilon,m,q}\int_{\mathbb{R}^N}\frac{\partial V}{\partial x_q}(x)\,u_\varepsilon\,\phi_m^2\,z_{\varepsilon,m}\,dx
-\sum_{r=1}^i\beta_{\varepsilon,r,j}\int_{\mathbb{R}^N}\frac{\partial V}{\partial x_j}(x)\,u_\varepsilon\,\phi_r^2\,\hat z_{\varepsilon,r}\,dx.
\end{aligned}
\]
Moreover,
\[
\begin{aligned}
2\varepsilon^{2s}&\!\int_{\mathbb{R}^{N+1}_+}\! t^{1-2s}
\Big(\nabla\bar f_\varepsilon\!\cdot\!\nabla\bar g_\varepsilon
+\nabla\bar f_\varepsilon\!\cdot\!\nabla\bar h_\varepsilon
+\nabla\bar g_\varepsilon\!\cdot\!\nabla\bar h_\varepsilon\Big)\,dX\\
&\qquad +2\int_{\mathbb{R}^N} V(x)\big(f_\varepsilon g_\varepsilon+f_\varepsilon h_\varepsilon+g_\varepsilon h_\varepsilon\big)\,dx\\
=&
\sum_{m=1}^k\alpha_{\varepsilon,m}\Bigg[
(p+1)\int_{\mathbb{R}^N}u_\varepsilon^{p}\phi_m^2\,z_{\varepsilon,m}\,dx
+2\varepsilon^{2s}\!\int_{\mathbb{R}^{N+1}_+}\!t^{1-2s}|\nabla\Phi_m|^2\,\bar u_\varepsilon\,\bar z_{\varepsilon,m}\,dX\Bigg]\\
&+\sum_{r=1}^i\alpha_{\varepsilon,r}\Bigg[
(p+1)\int_{\mathbb{R}^N}u_\varepsilon^{p}\phi_r^2\,\hat z_{\varepsilon,r}\,dx
+2\varepsilon^{2s}\!\int_{\mathbb{R}^{N+1}_+}\!t^{1-2s}|\nabla\Phi_r|^2\,\bar u_\varepsilon\,\overline{\hat z}_{\varepsilon,r}\,dX\Bigg]\\
&+2\sum_{r=1}^i\Bigg[
p\int_{\mathbb{R}^N}u_\varepsilon^{p-1}\phi_r^2\,z_{\varepsilon,r}\hat z_{\varepsilon,r}\,dx
+\varepsilon^{2s}\!\int_{\mathbb{R}^{N+1}_+}\!t^{1-2s}|\nabla\Phi_r|^2\,\bar z_{\varepsilon,r}\,\overline{\hat z}_{\varepsilon,r}\,dX\Bigg]\\
&-\sum_{m=1}^k\sum_{q=1}^{j-1}\alpha_{\varepsilon,m}\beta_{\varepsilon,m,q}\int_{\mathbb{R}^N}\frac{\partial V}{\partial x_q}(x)\,u_\varepsilon^{2}\phi_m^2\,dx
-\sum_{r=1}^i\alpha_{\varepsilon,r}\beta_{\varepsilon,r,j}\int_{\mathbb{R}^N}\frac{\partial V}{\partial x_j}(x)\,u_\varepsilon^{2}\phi_r^2\,dx\\
&-\sum_{r=1}^i\sum_{q=1}^{j-1}\beta_{\varepsilon,r,q}\int_{\mathbb{R}^N}\frac{\partial V}{\partial x_q}(x)\,u_\varepsilon\,\phi_r^2\,\hat z_{\varepsilon,r}\,dx
-\sum_{r=1}^i\beta_{\varepsilon,r,j}\int_{\mathbb{R}^N}\frac{\partial V}{\partial x_j}(x)\,u_\varepsilon\,\phi_r^2\,z_{\varepsilon,r}\,dx.
\end{aligned}
\]
\end{lemma}

\begin{proof}
  The expansions follow by the same direct calculations as in the proof of Lemma~\ref{lem:local-identities}; we omit the details.
\end{proof}

\begin{lemma}\label{lem:fractional-test-space-estimate}
Without loss of generality, we may assume that the coefficient norm
\begin{equation}\label{eq:scaled-coeff-norm}
\varepsilon^N\sum_{s=1}^k\alpha_{\varepsilon,s}^2
+
\varepsilon^{N-2}
\left(
\sum_{m=1}^k\sum_{q=1}^{j-1}\beta_{\varepsilon,m,q}^2
+
\sum_{r=1}^i\beta_{\varepsilon,r,j}^2
\right)=1.
\end{equation}
Then 
\begin{equation}\label{eq:scaled-rayleigh-quotient-bound}
\frac{J_{1,\varepsilon}(v)+J_{2,\varepsilon}(v)}{J_{3,\varepsilon}(v)}
=o(\varepsilon ).
\end{equation}
\end{lemma}

\begin{proof}
We first estimate $J_{3,\varepsilon}$. Recall that
\[
J_{3,\varepsilon}(v)=\int_{\mathbb R^N}v^2\,dx.
\]
As in the definition of $J_{3,\varepsilon}$, write
\begin{align*}
J_{3,\varepsilon}(v)
&=
\underbrace{
\sum_{j=1}^k\alpha_{\varepsilon,j}^2
\int_{\mathbb R^N}\phi_j^2u_\varepsilon^2\,dx
+2\sum_{j=1}^k\alpha_{\varepsilon,j}
\int_{\mathbb R^N}\phi_j^2u_\varepsilon z_{\varepsilon,j}\,dx
+2\sum_{r=1}^i\alpha_{\varepsilon,r}
\int_{\mathbb R^N}\phi_r^2u_\varepsilon\hat z_{\varepsilon,r}\,dx
}_{J_{31, \varepsilon}}
\nonumber\\
&\quad+
\underbrace{
\sum_{m=1}^k\int_{\mathbb R^N}\phi_m^2z_{\varepsilon,m}^2\,dx
+\sum_{r=1}^i\int_{\mathbb R^N}\phi_r^2\hat z_{\varepsilon,r}^2\,dx
+2\sum_{r=1}^i\int_{\mathbb R^N}\phi_r^2z_{\varepsilon,r}\hat z_{\varepsilon,r}\,dx
}_{J_{32, \varepsilon}}.
\end{align*}

By Lemma~\ref{lem:true-peak-profile-convergence} and Lemma~\ref{lem:u-polynomial-decay}, we have:
\begin{equation}\label{eq:u-l2-diagonal-limit}
\int_{\mathbb R^N}\phi_m^2u_\varepsilon^2\,dx
=
\varepsilon^N(\int_{\mathbb R^N}w_{V(\xi_m^0)}^2+o(1)),
\end{equation}
and
\begin{equation}\label{eq:derivative-diagonal-limit}
\int_{\mathbb R^N}
\phi_m^2
\frac{\partial u_\varepsilon}{\partial x_q}
\frac{\partial u_\varepsilon}{\partial x_a}\,dx
=
\varepsilon^{N-2}
\left(
\frac{1}N\int_{\mathbb R^N}|\nabla w_{V(\xi_m^0)}|^2\delta_{qa}+o(1)
\right).
\end{equation}
Moreover,
\begin{equation}\label{eq:u-derivative-mixed-small}
\int_{\mathbb R^N}
\phi_m^2u_\varepsilon
\frac{\partial u_\varepsilon}{\partial x_q}\,dx
=-\frac{1}{2}
\int_{\mathbb R^N}
u_\varepsilon^2
\frac{\partial(\phi_m^2)}{\partial x_q}\,dx=
o(\varepsilon^{N-1}).
\end{equation}

Using \eqref{eq:u-l2-diagonal-limit} and \eqref{eq:u-derivative-mixed-small}, we obtain
\begin{equation}\label{eq:first-denominator-limit}
J_{31,\varepsilon}
=
\varepsilon^N\sum_{m=1}^k
\left(\int_{\mathbb R^N}w_{V(\xi_m^0)}^2+o(1)\right)\alpha_{\varepsilon,m}^2
+o\!\left(\varepsilon^N\sum_m\alpha_{\varepsilon,m}^2
+\varepsilon^{N-2}\sum_{m,q}\beta_{\varepsilon,m,q}^2\right).
\end{equation}
For $J_{32, \varepsilon}$, by \eqref{eq:derivative-diagonal-limit}, we have:
\begin{equation}\label{eq:second-denominator-limit}
\begin{aligned}[b]
J_{32, \varepsilon}
&=
\sum_{m=1}^{k}\int_{\mathbb{R}^{N}}\phi_{m}^{2}z_{\varepsilon,m}^{2}\,dx
+\sum_{r=1}^{i}\int_{\mathbb{R}^{N}}\phi_{r}^{2}\hat{z}_{\varepsilon,r}^{2}\,dx
+2\sum_{m=1}^{k}\sum_{r=1}^{i}
\int_{\mathbb{R}^{N}}\phi_{m}\phi_{r}z_{\varepsilon,m}\hat{z}_{\varepsilon,r}\,dx\\
&\quad
+\sum_{m\neq m^{\prime}}^k
\int_{\mathbb{R}^N}\phi_m\phi_{m^{\prime}}z_{\varepsilon,m}z_{\varepsilon,m^{\prime}}\,dx
+\sum_{r\neq r^{\prime}}^i
\int_{\mathbb{R}^N}\phi_r\phi_{r^{\prime}}\hat{z}_{\varepsilon,r}\hat{z}_{\varepsilon,r^{\prime}}\,dx\\
&=\varepsilon^{N-2}
\left[
\sum_{m=1}^k
\frac{1}N\int_{\mathbb R^N}|\nabla w_{V(\xi_m^0)}|^2
\sum_{q=1}^{j-1}
\beta_{\varepsilon,m,q}^2
+
\sum_{r=1}^i
\frac{1}N\int_{\mathbb R^N}|\nabla w_{V(\xi_r^0)}|^2\beta_{\varepsilon,r,j}^2
\right]\\
&\quad+{o\!\left(\varepsilon^{N-2}
\left[\sum_{m=1}^k\sum_{q=1}^{j-1}\beta_{\varepsilon,m,q}^2
+\sum_{r=1}^i\beta_{\varepsilon,r,j}^2\right]\right)}.
\end{aligned}
\end{equation}
Combining \eqref{eq:first-denominator-limit} and \eqref{eq:second-denominator-limit}, we obtain
\begin{equation}\label{eq:scaled-denominator-lower-bound}
  J_{3,\varepsilon}(v)
\ge
c.
\end{equation}

We next estimate $J_{1,\varepsilon}$. We write
\[
J_{1,\varepsilon}=J_{11,\varepsilon}+J_{12,\varepsilon},
\]
where 
\begin{align*}
J_{11,\varepsilon}:={}&
(1-p)\sum_{m=1}^k
\alpha_{\varepsilon,m}^2
\int_{\mathbb R^N}u_\varepsilon^{p+1}\phi_m^2\,dx
+(1-p)\sum_{m=1}^k
\alpha_{\varepsilon,m}
\int_{\mathbb R^N}u_\varepsilon^p\phi_m^2z_{\varepsilon,m}\,dx \\
&\quad
+(1-p)\sum_{r=1}^i
\alpha_{\varepsilon,r}
\int_{\mathbb R^N}u_\varepsilon^p\phi_r^2\hat z_{\varepsilon,r}\,dx,
\end{align*}
and 
\begin{align*}
J_{12,\varepsilon}:={}&
\varepsilon^{2s}\sum_{m=1}^k
\alpha_{\varepsilon,m}^2
\int_{\mathbb R^{N+1}_+}
t^{1-2s}|\nabla\Phi_m|^2\bar u_\varepsilon^{\,2}\,dX
+\varepsilon^{2s}\sum_{m=1}^k
\int_{\mathbb R^{N+1}_+}
t^{1-2s}|\nabla\Phi_m|^2\bar z_{\varepsilon,m}^{\,2}\,dX \\
&\quad
+\varepsilon^{2s}\sum_{r=1}^i
\int_{\mathbb R^{N+1}_+}
t^{1-2s}|\nabla\Phi_r|^2\overline{\hat z}_{\varepsilon,r}^{\,2}\,dX \\
&\quad
+2\varepsilon^{2s}\sum_{m=1}^k
\alpha_{\varepsilon,m}
\int_{\mathbb R^{N+1}_+}
t^{1-2s}|\nabla\Phi_m|^2\bar u_\varepsilon\bar z_{\varepsilon,m}\,dX \\
&\quad
+2\varepsilon^{2s}\sum_{r=1}^i
\alpha_{\varepsilon,r}
\int_{\mathbb R^{N+1}_+}
t^{1-2s}|\nabla\Phi_r|^2\bar u_\varepsilon\overline{\hat z}_{\varepsilon,r}\,dX \\
&\quad
+2\varepsilon^{2s}\sum_{r=1}^i
\int_{\mathbb R^{N+1}_+}
t^{1-2s}|\nabla\Phi_r|^2\bar z_{\varepsilon,r}\overline{\hat z}_{\varepsilon,r}\,dX .
\end{align*}

The first term in \(J_{11,\varepsilon}\) is non-positive. For each
$m$ and $q$, Lemma~\ref{lem:u-polynomial-decay} gives
\begin{align*}
\int_{\mathbb R^N}
u_\varepsilon^p\phi_m^2
\frac{\partial u_\varepsilon}{\partial x_q}\,dx
&=
\frac{1}{p+1}
\int_{\mathbb R^N}
\phi_m^2
\frac{\partial(u_\varepsilon^{p+1})}{\partial x_q}\,dx
=
-
\frac{1}{p+1}
\int_{\mathbb R^N}
u_\varepsilon^{p+1}
\frac{\partial(\phi_m^2)}{\partial x_q}\,dx=o(\varepsilon^N).
\end{align*}
Consequently,
\begin{align}
\left|
\alpha_{\varepsilon,m}
\beta_{\varepsilon,m,q}
\int_{\mathbb R^N}
u_\varepsilon^p\phi_m^2
\frac{\partial u_\varepsilon}{\partial x_q}\,dx
\right|
&\le
o(\varepsilon^N)
|\alpha_{\varepsilon,m}|
|\beta_{\varepsilon,m,q}| 
\nonumber\\
&\le
o(\varepsilon )
\left(
\varepsilon^N\alpha_{\varepsilon,m}^2
+
\varepsilon^{N-2}\beta_{\varepsilon,m,q}^2
\right)=o(\varepsilon).
\label{eq:mixed-nonlinear-scaled}
\end{align}
The terms involving $\hat z_{\varepsilon,r}$ are estimated in the same way. By \eqref{eq:scaled-coeff-norm}, summing \eqref{eq:mixed-nonlinear-scaled} over all indices gives
\[
J_{11,\varepsilon}\le o(\varepsilon).
\]

For $J_{12,\varepsilon}$, we use the annular estimates
\begin{equation}\label{eq:cutoff-u-scaled}
\varepsilon^{2s}
\int_{\operatorname{supp}\nabla\Phi_m}
t^{1-2s}
|\nabla\Phi_m|^2
\bar u_\varepsilon^2\,dX
=o(\varepsilon^{N+1}),
\end{equation}
\begin{equation}\label{eq:cutoff-u-du-scaled}
\varepsilon^{2s}
\int_{\operatorname{supp}\nabla\Phi_m}
t^{1-2s}
|\nabla\Phi_m|^2
\left|
\bar u_\varepsilon\,
\overline{\partial_{x_q}u_\varepsilon}
\right|\,dX
=
o(\varepsilon^{N}),
\end{equation}
and
\begin{equation}\label{eq:cutoff-du-du-scaled}
\varepsilon^{2s}
\int_{\operatorname{supp}\nabla\Phi_m}
t^{1-2s}
|\nabla\Phi_m|^2
\left|
\overline{\partial_{x_q}u_\varepsilon}\,
\overline{\partial_{x_a}u_\varepsilon}
\right|\,dX
=o(\varepsilon ^{N-1}).
\end{equation}
Indeed, these estimates follow from Lemma~\ref{lem:extension-u-bound} because \(\operatorname{supp}\nabla\Phi_m\subset A_{r,2r}^{+,m}\).

Expanding $z_{\varepsilon,m}$ and $\hat z_{\varepsilon,r}$ and applying Cauchy's inequality to
\eqref{eq:cutoff-u-scaled}--\eqref{eq:cutoff-du-du-scaled}, we obtain
\[
J_{12,\varepsilon}
=o(\varepsilon)
\left[
\varepsilon^N\sum_{m=1}^k\alpha_{\varepsilon,m}^2
+
\varepsilon^{N-2}
\left(
\sum_{m=1}^k\sum_{q=1}^{j-1}\beta_{\varepsilon,m,q}^2
+
\sum_{r=1}^i\beta_{\varepsilon,r,j}^2
\right)
\right]
=
o(\varepsilon).
\]
Together with the estimate for $J_{11,\varepsilon}$, this proves 
\begin{equation}\label{eq:scaled-first-numerator-upper-bound}
  J_{1,\varepsilon }=o(\varepsilon).
\end{equation}

It remains to estimate $J_{2,\varepsilon}$. By Lemmas~\ref{lem:u-polynomial-decay}, \ref{lem:extension-u-bound} and \ref{lem:pohozaev-center-estimate}, we have
\begin{equation}
  \begin{aligned}
      \int_{\mathbb R^N}
\frac{\partial V}{\partial x_q}(x)
u_\varepsilon^2\phi_m^2\,dx&=\varepsilon ^N\int_{B_{2r/\varepsilon}(0)}\frac{\partial V}{\partial x_q}
(\xi_{\varepsilon,m}+\varepsilon y)u_\varepsilon^2(\xi_{\varepsilon,m}+\varepsilon y)\phi^2(\varepsilon y)\,dy\\
=&\varepsilon ^N\frac{\partial V}{\partial x_q}(\xi_{\varepsilon,m})\int_{B_{2r/\varepsilon}(0)}u^2_\varepsilon(\xi_{\varepsilon,m}+\varepsilon y)\phi^2(\varepsilon y)\,dy+o(\varepsilon ^{N+1})\\
=&O(\varepsilon ^{N+1}).
  \end{aligned}
\end{equation}
Similarly,
\begin{equation}
  \begin{aligned}
      \int_{\mathbb R^N}
\frac{\partial V}{\partial x_q}(x)
u_\varepsilon \phi_m^2 z_{\varepsilon,m}\,dx&=\varepsilon ^N\int_{B_{2r/\varepsilon}(0)}\frac{\partial V}{\partial x_q}
(\xi_{\varepsilon,m}+\varepsilon y)u_\varepsilon(\xi_{\varepsilon,m}+\varepsilon y)\phi^2(\varepsilon y)z_{\varepsilon,m}(\xi_{\varepsilon,m}+\varepsilon y)\,dy\\
=&\varepsilon ^N\frac{\partial V}{\partial x_q}(\xi_{\varepsilon,m})\int_{B_{2r/\varepsilon}(0)}u_\varepsilon(\xi_{\varepsilon,m}+\varepsilon y)\phi^2(\varepsilon y)z_{\varepsilon,m}(\xi_{\varepsilon,m}+\varepsilon y)\,dy+o(\varepsilon ^{N})\\
=&O(\varepsilon ^{N})\left(\sum_{q=1}^{j-1}\beta_{\varepsilon,m,q}^2\right)^{1/2}.
  \end{aligned}
\end{equation}
The remaining terms in $J_{2,\varepsilon}$ are estimated in the same way. Therefore
\begin{equation}\label{eq:scaled-second-numerator-upper-bound}
\left|J_{2,\varepsilon}(v)\right|
\le
C\varepsilon^N|\beta_\varepsilon|^2
+
C\varepsilon^{N+1}
|\alpha_\varepsilon|\,|\beta_\varepsilon|
\le
C\varepsilon^2.
\end{equation}

Finally, \eqref{eq:scaled-denominator-lower-bound}, \eqref{eq:scaled-first-numerator-upper-bound} and
\eqref{eq:scaled-second-numerator-upper-bound} imply
\[
\frac{J_{1,\varepsilon}(v)+J_{2,\varepsilon}(v)}{J_{3,\varepsilon}(v)}
\le
\frac{o(\varepsilon )+C\varepsilon^2}{c}
=
o(\varepsilon ).
\]
This proves
\eqref{eq:scaled-rayleigh-quotient-bound}.
\end{proof}

We shall use the following elementary estimate for separated algebraic weights; see \cite{wei2010infinitely}.
\begin{lemma}\label{lem:separated-product-estimate}
For any $0<\delta\le \min\{\alpha,\beta\}$ and $x_1\neq x_2$, there holds that
\begin{equation*}
\frac{1}{(1+|y-x_1|)^\alpha(1+|y-x_2|)^\beta}
\le
\frac{C}{|x_1-x_2|^\delta}
\left(
\frac{1}{(1+|y-x_1|)^{\alpha+\beta-\delta}}
+
\frac{1}{(1+|y-x_2|)^{\alpha+\beta-\delta}}
\right),
\end{equation*}
where $\alpha,\beta$ are two constants.
\end{lemma}
\begin{lemma}\label{lem:weighted-convolution-estimate}

Let $\rho>0$, $\alpha>N$, $\beta>0$, and $\beta\ne N$.  There exists $C=C(N,\alpha,\beta,\rho)$ such that, for $0<\varepsilon\le1$, $t>0$, and $|y-x|^2+t^2\ge\rho^2$,
\[
\int_{\mathbb R^N}\frac{dz}{(t+|z|)^\alpha\left(1+\frac{|y-z-x|}{\varepsilon}\right)^\beta}
\le
\begin{cases}
C\varepsilon^\beta\bigl[t^{N-\alpha}(1+|y-x|)^{-\beta}+(1+|y-x|)^{-\alpha}\bigr],&0<\beta<N,\\[1mm]
C\varepsilon^N\bigl[t^{N-\alpha}(1+|y-x|)^{-\beta}+(1+|y-x|)^{-\alpha}\bigr],&\beta>N.
\end{cases}
\]
The borderline case $\beta=N$, which carries a logarithmic correction, is not used.

\end{lemma}

\begin{lemma}\label{lem:algebraic-nonlinear-estimates}
Let \(a,b\ge0\). Then
\[
|(a+b)^{p-1}-a^{p-1}|
\le C
\begin{cases}
b^{p-1},&1<p<2,\\
(a^{p-2}+b^{p-2})b,&p\ge2.
\end{cases}
\]
\end{lemma}

\section{Useful estimates}

We record the Gram estimates.
Throughout this appendix, \((\alpha,\xi)\in \overline{\mathcal D}_\varepsilon\).
\begin{lemma}\label{lem:app-gram-estimates}
The following estimates hold.

First, for the same peak \(i=j\),
\begin{equation}\label{eq:app-same-ww}
\langle W_{\varepsilon,j},W_{\varepsilon,j}\rangle_{\varepsilon,V}
=
\varepsilon^N\bigl(a_0V(\xi_j)^\theta+o(1)\bigr),
\end{equation}
Moreover,
\begin{equation}\label{eq:app-same-wy}
\langle W_{\varepsilon,j},Y_{\varepsilon,j,h}\rangle_{\varepsilon,V}
=
o(\varepsilon^{N-1}),
\end{equation}
and
\begin{equation}\label{eq:app-same-yy}
\langle Y_{\varepsilon,j,h},Y_{\varepsilon,j,m}\rangle_{\varepsilon,V}
=
\varepsilon^{N-2}(K_j\delta_{hm}+o(1)),
\end{equation}
where
\[
K_j
:=
p\int_{\mathbb R^N}
w_{V(\xi_j)}^{p-1}
\left(\partial_1w_{V(\xi_j)}\right)^2
>0.
\]

Second, for different peaks \(i\ne j\),
\begin{equation}\label{eq:app-diff-ww}
\langle W_{\varepsilon,i},W_{\varepsilon,j}\rangle_{\varepsilon,V}
=
O(\varepsilon^{2N+2s})
=
o(\varepsilon^N),
\end{equation}
\begin{equation}\label{eq:app-diff-wy}
\langle W_{\varepsilon,i},Y_{\varepsilon,j,h}\rangle_{\varepsilon,V}
=
O(\varepsilon^{2N+2s-1})
=
o(\varepsilon^{N-1}),
\end{equation}
and
\begin{equation}\label{eq:app-diff-yy}
\langle Y_{\varepsilon,i,h},Y_{\varepsilon,j,m}\rangle_{\varepsilon,V}
=
O(\varepsilon^{2N+2s-2})
=
o(\varepsilon^{N-2}).
\end{equation}

Finally, 
\begin{equation}\label{eq:app-weight-w}
p\int_{\mathbb R^N}
W_{\varepsilon,\alpha,\xi}^{p-1}
W_{\varepsilon,i}^2
=
O(\varepsilon^N),
\end{equation}
and
\begin{equation}\label{eq:app-weight-y}
p\int_{\mathbb R^N}
W_{\varepsilon,\alpha,\xi}^{p-1}
Y_{\varepsilon,i,\ell}^2
=
O(\varepsilon^{N-2}).
\end{equation}
\end{lemma}

\begin{proof}
We divide the proof into three parts.

\medskip
\noindent
\textbf{Step 1: Same peak.} Observe that
\begin{equation}\label{eq:app-wj-equation}
\varepsilon^{2s}(-\Delta)^sW_{\varepsilon,j}
+
V(\xi_j)W_{\varepsilon,j}
=
W_{\varepsilon,j}^p .
\end{equation}
Testing \eqref{eq:app-wj-equation} against \(W_{\varepsilon,j}\), we obtain
\[
\begin{aligned}
\langle W_{\varepsilon,j},W_{\varepsilon,j}\rangle_{\varepsilon,V}
&=
\int_{\mathbb R^N}W_{\varepsilon,j}^{p+1}
+
\int_{\mathbb R^N}
\bigl(V(x)-V(\xi_j)\bigr)W_{\varepsilon,j}^2  \\
&=
\varepsilon^Na_0V(\xi_j)^\theta+o(\varepsilon^N),
\end{aligned}
\]
which proves \eqref{eq:app-same-ww}.

Next, by the definition of \(Y_{\varepsilon,j,h}\), we have
\[
Y_{\varepsilon,j,h}
=
-\partial_{x_h}W_{\varepsilon,j}
+
\partial_hV(\xi_j)
\left.
\partial_\lambda w_\lambda
\left(\frac{x-\xi_j}{\varepsilon}\right)
\right|_{\lambda=V(\xi_j)}.
\]
Using \eqref{eq:app-wj-equation}, Lemma~\ref{lem:single-bubble-input}, and scaling,
\[
\langle W_{\varepsilon,j},Y_{\varepsilon,j,h}\rangle_{\varepsilon,V}
=
O(\varepsilon^N)
=
o(\varepsilon^{N-1}),
\]
which proves \eqref{eq:app-same-wy}.

For the following product, we obtain
\[
\begin{aligned}
\langle Y_{\varepsilon,j,h},Y_{\varepsilon,j,m}\rangle_{\varepsilon,V}
&=
\langle \partial_{x_h}W_{\varepsilon,j},
\partial_{x_m}W_{\varepsilon,j}\rangle_{\varepsilon,V}
+
O(\varepsilon^{N-1}).
\end{aligned}
\]
Differentiating \eqref{eq:app-wj-equation} with respect to \(x_h\) and testing against
\(\partial_{x_m}W_{\varepsilon,j}\), we obtain
\[
\begin{aligned}
\langle Y_{\varepsilon,j,h},Y_{\varepsilon,j,m}\rangle_{\varepsilon,V}
&=
p\int_{\mathbb R^N}
W_{\varepsilon,j}^{p-1}
\partial_{x_h}W_{\varepsilon,j}
\partial_{x_m}W_{\varepsilon,j}
+
O(\varepsilon^{N-1})  \\
&=
\varepsilon^{N-2}
\left(
p\int_{\mathbb R^N}
w_{V(\xi_j)}^{p-1}
\left(\partial_1w_{V(\xi_j)}\right)^2
\delta_{hm}
+
o(1)
\right),
\end{aligned}
\]
which proves \eqref{eq:app-same-yy}.

\medskip
\noindent
\textbf{Step 2: Different peaks.}
Let \(i\ne j\).
For every \(q,r\ge1\), Lemma~\ref{lem:separated-product-estimate} gives
\begin{equation}\label{eq:app-separated-tail}
\begin{aligned}
&\int_{\mathbb R^N}
\left(1+\frac{|x-\xi_i|}{\varepsilon}\right)^{-q(N+2s)}
\left(1+\frac{|x-\xi_j|}{\varepsilon}\right)^{-r(N+2s)}\,dx \\
&\qquad\le
C\varepsilon^N
\left(
1+\frac{|\xi_i-\xi_j|}{\varepsilon}
\right)^{-(N+2s)}
\le
C\varepsilon^{2N+2s}.
\end{aligned}
\end{equation}
Indeed, after the change of variables \(x=\xi_i+\varepsilon y\), this is Lemma~\ref{lem:separated-product-estimate} with exponents \(q(N+2s)\), \(r(N+2s)\), and \(\delta=N+2s\).

We now estimate the energy products. Observe that
\begin{equation}\label{eq:app-frozen-wi}
\varepsilon^{2s}(-\Delta)^sW_{\varepsilon,i}
+
V(\xi_i)W_{\varepsilon,i}
=
W_{\varepsilon,i}^p.
\end{equation}
Testing \eqref{eq:app-frozen-wi} against \(W_{\varepsilon,j}\), we obtain
\[
\begin{aligned}
\langle W_{\varepsilon,i},W_{\varepsilon,j}\rangle_{\varepsilon,V}
&=
\int_{\mathbb R^N}
W_{\varepsilon,i}^pW_{\varepsilon,j}
+
\int_{\mathbb R^N}
\bigl(V(x)-V(\xi_i)\bigr)
W_{\varepsilon,i}W_{\varepsilon,j}.
\end{aligned}
\]
By Lemma~\ref{lem:single-bubble-input} and \eqref{eq:app-separated-tail} with \((q,r)=(p,1)\) and \((q,r)=(1,1)\),
\[
\left|
\langle W_{\varepsilon,i},W_{\varepsilon,j}\rangle_{\varepsilon,V}
\right|
\le
C\varepsilon^{2N+2s}.
\]
This proves \eqref{eq:app-diff-ww}.

Testing \eqref{eq:app-frozen-wi} against \(Y_{\varepsilon,j,h}\), we obtain
\[
\begin{aligned}
\langle W_{\varepsilon,i},Y_{\varepsilon,j,h}\rangle_{\varepsilon,V}
&=
\int_{\mathbb R^N}
W_{\varepsilon,i}^pY_{\varepsilon,j,h}
+
\int_{\mathbb R^N}
\bigl(V(x)-V(\xi_i)\bigr)
W_{\varepsilon,i}Y_{\varepsilon,j,h}.
\end{aligned}
\]
By Lemma~\ref{lem:single-bubble-input} and \eqref{eq:app-separated-tail} with \((q,r)=(p,1)\) and \((q,r)=(1,1)\),
\[
\left|
\langle W_{\varepsilon,i},Y_{\varepsilon,j,h}\rangle_{\varepsilon,V}
\right|
\le
C\varepsilon^{2N+2s-1}.
\]
Thus \eqref{eq:app-diff-wy} follows.

Finally, differentiating \eqref{eq:app-frozen-wi} with respect to \(\xi_{i,h}\), we obtain
\begin{equation}\label{eq:app-frozen-yi}
\varepsilon^{2s}(-\Delta)^sY_{\varepsilon,i,h}
+
V(\xi_i)Y_{\varepsilon,i,h}
=
pW_{\varepsilon,i}^{p-1}Y_{\varepsilon,i,h}
-
\partial_hV(\xi_i)W_{\varepsilon,i}.
\end{equation}
Testing \eqref{eq:app-frozen-yi} against \(Y_{\varepsilon,j,m}\), we obtain
\[
\begin{aligned}
\langle Y_{\varepsilon,i,h},Y_{\varepsilon,j,m}\rangle_{\varepsilon,V}
&=
p\int_{\mathbb R^N}
W_{\varepsilon,i}^{p-1}
Y_{\varepsilon,i,h}
Y_{\varepsilon,j,m} \\
&\quad
-
\partial_hV(\xi_i)
\int_{\mathbb R^N}
W_{\varepsilon,i}Y_{\varepsilon,j,m} \\
&\quad
+
\int_{\mathbb R^N}
\bigl(V(x)-V(\xi_i)\bigr)
Y_{\varepsilon,i,h}Y_{\varepsilon,j,m}.
\end{aligned}
\]
By Lemma~\ref{lem:single-bubble-input} and \eqref{eq:app-separated-tail} with \((q,r)=(p,1)\) and \((q,r)=(1,1)\), 
\[
\left|
\langle Y_{\varepsilon,i,h},Y_{\varepsilon,j,m}\rangle_{\varepsilon,V}
\right|
\le
C\varepsilon^{2N+2s-2},
\]
which proves \eqref{eq:app-diff-yy}.

\medskip
\noindent
\textbf{Step 3: Integral estimates.} By Lemma~\ref{lem:single-bubble-input}, a direct computation gives
\[
p\int_{\mathbb R^N}
W_{\varepsilon,\alpha,\xi}^{p-1}
W_{\varepsilon,i}^2
\le
C\varepsilon^N
\int_{\mathbb R^N}w_{V(\xi_i)}^2 
=
O(\varepsilon^N).
\]
This proves \eqref{eq:app-weight-w}.

Similarly,
\[
\begin{aligned}
p\int_{\mathbb R^N}
W_{\varepsilon,\alpha,\xi}^{p-1}
Y_{\varepsilon,i,\ell}^2
&\le
C\int_{\mathbb R^N}
Y_{\varepsilon,i,\ell}^2 \\
&\le
C\varepsilon^{N-2}
\int_{\mathbb R^N}
|\partial_\ell w_{V(\xi_i)}|^2
+
C\varepsilon^N
\int_{\mathbb R^N}
\left|
\left.\partial_\lambda w_\lambda\right|_{\lambda=V(\xi_i)}
\right|^2 \\
&=
O(\varepsilon^{N-2}),
\end{aligned}
\]
which proves \eqref{eq:app-weight-y}, and the lemma follows.
\end{proof}
Moreover, we have the following pointwise estimates. 
\begin{lemma}\label{lem:weighted-tail}
\begin{equation}\label{eq:sep-tail-single-profile-pointwise}
W_{\varepsilon,j}(x)^p
\le
C\rho_{\tau,\varepsilon,\xi}(x),
\end{equation}
\begin{equation}\label{eq:sep-tail-potential-pointwise}
\bigl|(V(x)-V(\xi_j))W_{\varepsilon,j}(x)\bigr|
\le
C\varepsilon^\tau\rho_{\tau,\varepsilon,\xi}(x),
\end{equation}
and
\begin{equation}\label{eq:sep-tail-nonlinear-pointwise}
\left|
W_{\varepsilon,\alpha,\xi}(x)^p
-
\sum_{j=1}^k\alpha_j^pW_{\varepsilon,j}(x)^p
\right|
\le
C\varepsilon^\tau\rho_{\tau,\varepsilon,\xi}(x).
\end{equation}
For each fixed \(j\),
\begin{equation}\label{eq:sep-tail-linearized-pointwise}
\left|
W_{\varepsilon,\alpha,\xi}(x)^{p-1}W_{\varepsilon,j}(x)
-
\alpha_j^{p-1}W_{\varepsilon,j}(x)^p
\right|
\le
C\varepsilon^\tau\rho_{\tau,\varepsilon,\xi}(x).
\end{equation}
Moreover,
\begin{equation}\label{eq:sep-tail-frozen-residual-pointwise}
\left|
W_{\varepsilon,j}^p
+
(V(x)-V(\xi_j))W_{\varepsilon,j}
\right|
\le
C\rho_{\tau,\varepsilon,\xi}(x),
\end{equation}
and
\begin{equation}\label{eq:sep-tail-frozen-derivative-pointwise}
\left|
pW_{\varepsilon,j}^{p-1}Y_{\varepsilon,j,h}
-
\partial_hV(\xi_j)W_{\varepsilon,j}
+
(V(x)-V(\xi_j))Y_{\varepsilon,j,h}
\right|
\le
C\varepsilon^{-1}\rho_{\tau,\varepsilon,\xi}(x).
\end{equation}
\end{lemma}

\begin{proof}
Lemma~\ref{lem:single-bubble-input} gives both \eqref{eq:sep-tail-single-profile-pointwise} and \eqref{eq:sep-tail-potential-pointwise}.
Since the coefficients \(\alpha_i\) and the profiles \(W_{\varepsilon,i}\) are uniformly bounded, Lemma~\ref{lem:single-bubble-input} gives
\[
0
\le
W_{\varepsilon,\alpha,\xi}(x)^p
-
\sum_{j=1}^k\alpha_j^pW_{\varepsilon,j}(x)^p 
\le
C\varepsilon^\tau\rho_{\tau,\varepsilon,\xi}(x).
\]
This proves \eqref{eq:sep-tail-nonlinear-pointwise}. For \eqref{eq:sep-tail-linearized-pointwise}, apply Lemma~\ref{lem:algebraic-nonlinear-estimates}; this separates the H\"older case \(1<p<2\) from the Lipschitz case \(p\ge2\)

Estimate \eqref{eq:sep-tail-frozen-residual-pointwise} follows from \eqref{eq:sep-tail-single-profile-pointwise} and \eqref{eq:sep-tail-potential-pointwise}; Lemma~\ref{lem:single-bubble-input} gives \eqref{eq:sep-tail-frozen-derivative-pointwise} by the same argument.
\end{proof}

\allowdisplaybreaks


\begin{thebibliography}{99}


\bibitem{aftalion2004qualitative}
Amandine Aftalion and Filomena Pacella.
\newblock Qualitative properties of nodal solutions of semilinear elliptic
  equations in radially symmetric domains.
\newblock {\em Comptes Rendus Math\'ematique}, 339(5):339--344, 2004.

\bibitem{applebaum2009levy}
David Applebaum.
\newblock {\em L{\'e}vy Processes and Stochastic Calculus}.
\newblock Cambridge University Press, Cambridge, second edition, 2009.

\bibitem{bahri1992solutions}
Abbas Bahri and Pierre-Louis Lions.
\newblock Solutions of superlinear elliptic equations and their Morse indices.
\newblock {\em Communications on Pure and Applied Mathematics},
  45(9):1205--1215, 1992.

\bibitem{barrios2012some}
Bego{\~n}a Barrios, Eduardo Colorado, Arturo de~Pablo, and Urko S{\'a}nchez.
\newblock On some critical problems for the fractional Laplacian operator.
\newblock {\em Journal of Differential Equations}, 252(11):6133--6162, 2012.

\bibitem{2002Existence}
Peter~W. Bates and Junping Shi.
\newblock Existence and instability of spike layer solutions to singular
  perturbation problems.
\newblock {\em Journal of Functional Analysis}, 196(2):211--264, 2002.

\bibitem{brandle2013concave}
Christoph Br{\"a}ndle, Eduardo Colorado, Arturo de~Pablo, and Urko S{\'a}nchez.
\newblock A concave-convex elliptic problem involving the fractional
  Laplacian.
\newblock {\em Proceedings of the Royal Society of Edinburgh Section A:
  Mathematics}, 143(1):39--71, 2013.

\bibitem{cabre2014nonlinear}
Xavier Cabr{\'e} and Yannick Sire.
\newblock Nonlinear equations for fractional Laplacians {I}: Regularity,
  maximum principles, and Hamiltonian estimates.
\newblock {\em Annales de l'Institut Henri Poincar{\'e} C, Analyse Non
  Lin{\'e}aire}, 31(1):23--53, 2014.

\bibitem{Caffarelli08082007}
Luis Caffarelli and Luis Silvestre.
\newblock An extension problem related to the fractional Laplacian.
\newblock {\em Communications in Partial Differential Equations},
  32(7--9):1245--1260, 2007.

\bibitem{CaoHeinz2003}
Daomin Cao and Hans-Peter Heinz.
\newblock Uniqueness of positive multi-lump bound states of nonlinear
  Schr{\"o}dinger equations.
\newblock {\em Mathematische Zeitschrift}, 243:599--642, 2003.

\bibitem{CaoLiLuo2015}
Daomin Cao, Shuanglong Li, and Peng Luo.
\newblock Uniqueness of positive bound states with multi-bump for nonlinear
  Schr{\"o}dinger equations.
\newblock {\em Calculus of Variations and Partial Differential Equations},
  54(4):4037--4063, 2015.

\bibitem{CaoNoussairYan1998}
Daomin Cao, Ezzat~S. Noussair, and Shusen Yan.
\newblock Existence and uniqueness results on single-peaked solutions of a
  semilinear problem.
\newblock {\em Annales de l'Institut Henri Poincar{\'e} C, Analyse Non
  Lin{\'e}aire}, 15(1):73--111, 1998.

\bibitem{CassaniWang2023}
Daniele Cassani and Youjun Wang.
\newblock Asymptotic behavior of ground states and local uniqueness for fractional
  {S}chr\"odinger equations with nearly critical growth.
\newblock {\em Potential Analysis}, 59:1--39, 2023.

\bibitem{chen2020fractional}
Wenxiong Chen, Yan Li, and Pei Ma.
\newblock {\em The Fractional Laplacian}.
\newblock World Scientific, Hackensack, NJ, 2020.

\bibitem{choi2016qualitative}
Woocheol Choi, Seunghyeok Kim, and Ki-Ahm Lee.
\newblock Qualitative properties of multi-bubble solutions for nonlinear
  elliptic equations involving critical exponents.
\newblock {\em Advances in Mathematics}, 298:484--533, 2016.

\bibitem{dancer2004stableII}
E.~N. Dancer.
\newblock Stable and finite Morse index solutions on $\mathbb{R}^{n}$ or on
  bounded domains with small diffusion. II.
\newblock {\em Indiana University Mathematics Journal}, 53(1):97--108, 2004.

\bibitem{dancer2005stable}
E.~N. Dancer.
\newblock Stable and finite Morse index solutions on $\mathbb{R}^{n}$ or on
  bounded domains with small diffusion.
\newblock {\em Transactions of the American Mathematical Society},
  357(3):1225--1243, 2005.

\bibitem{MR3121716}
Juan D{\'a}vila, Manuel del Pino, and Juncheng Wei.
\newblock Concentrating standing waves for the fractional nonlinear
  Schr{\"o}dinger equation.
\newblock {\em Journal of Differential Equations}, 256(2):858--892, 2014.

\bibitem{DeMarchisGrossiIanniPacella2019JMPA}
Francesca De~Marchis, Massimo Grossi, Isabella Ianni, and Filomena Pacella.
\newblock Morse index and uniqueness of positive solutions of the
  Lane--Emden problem in planar domains.
\newblock {\em Journal de Math{\'e}matiques Pures et Appliqu{\'e}es},
  128:339--378, 2019.

\bibitem{DLY}
Yinbin Deng, Chang-Shou Lin, and Shusen Yan.
\newblock On the prescribed scalar curvature problem in $\mathbb{R}^N$: local
  uniqueness and periodicity.
\newblock {\em Journal de Math{\'e}matiques Pures et Appliqu{\'e}es},
  104(6):1013--1044, 2015.

\bibitem{YSX}
Yinbin Deng, Shuangjie Peng, and Xian Yang.
\newblock Existence and decays of solutions for fractional Schr{\"o}dinger
  equations with general potentials.
\newblock {\em Calculus of Variations and Partial Differential Equations},
  63:Article No.~128, 2024.

\bibitem{MR2944369}
Eleonora Di~Nezza, Giampiero Palatucci, and Enrico Valdinoci.
\newblock Hitchhiker's guide to the fractional Sobolev spaces.
\newblock {\em Bulletin des Sciences Math{\'e}matiques}, 136(5):521--573, 2012.

\bibitem{dipierro2012existence}
Serena Dipierro, Giampiero Palatucci, and Enrico Valdinoci.
\newblock Existence and symmetry results for a Schr{\"o}dinger type problem
  involving the fractional Laplacian.
\newblock {\em Le Matematiche}, 68(1):201--216, 2013.

\bibitem{DuarteSouto2019}
Ronaldo~C. Duarte and Marco A.~S. Souto.
\newblock Nonlocal Schr{\"o}dinger equations for integro-differential
  operators with measurable kernels.
\newblock {\em Topological Methods in Nonlinear Analysis}, 54(1):383--406,
  2019.

\bibitem{farina2007classification}
Alberto Farina.
\newblock On the classification of solutions of the Lane--Emden equation on
  unbounded domains of $\mathbb{R}^{N}$.
\newblock {\em Journal de Math\'ematiques Pures et Appliqu\'ees},
  87(5):537--561, 2007.

\bibitem{FinsterLottner2021Banach}
Felix Finster and Magdalena Lottner.
\newblock Banach manifold structure and infinite-dimensional analysis for
  causal fermion systems.
\newblock {\em Annals of Global Analysis and Geometry}, 60:313--354, 2021.

\bibitem{MR3070568}
Rupert~L. Frank and Enno Lenzmann.
\newblock Uniqueness of non-linear ground states for fractional Laplacians in
  $\mathbb{R}$.
\newblock {\em Acta Mathematica}, 210(2):261--318, 2013.

\bibitem{frank2016uniqueness}
Rupert~L. Frank, Enno Lenzmann, and Luis Silvestre.
\newblock Uniqueness of radial solutions for the fractional Laplacian.
\newblock {\em Communications on Pure and Applied Mathematics},
  69(9):1671--1726, 2016.

\bibitem{frank2008hardy}
Rupert~L. Frank, Elliott~H. Lieb, and Robert Seiringer.
\newblock Hardy--Lieb--Thirring inequalities for fractional Schr{\"o}dinger
  operators.
\newblock {\em Journal of the American Mathematical Society}, 21(4):925--950,
  2008.

\bibitem{gladiali2014morse}
Francesca Gladiali, Massimo Grossi, Hiroshi Ohtsuka, and Takashi Suzuki.
\newblock Morse indices of multiple blow-up solutions to the two-dimensional
  Gel'fand problem.
\newblock {\em Communications in Partial Differential Equations},
  39(11):2028--2063, 2014.

\bibitem{grossi2005eigenvalue}
Massimo Grossi and Filomena Pacella.
\newblock On an eigenvalue problem related to the critical exponent.
\newblock {\em Mathematische Zeitschrift}, 250(1):225--256, 2005.

\bibitem{grossi2007morse}
Massimo Grossi and Raffaella Servadei.
\newblock Morse index for solutions of the nonlinear Schr{\"o}dinger equation
  in a degenerate setting.
\newblock {\em Annali di Matematica Pura ed Applicata}, 186(3):433--453, 2007.

\bibitem{GuoLiLiuYang2026}
Yanyan Guo, Ying Li, Zhongyuan Liu, and Pingping Yang.
\newblock Existence, non-degeneracy and local uniqueness of multi-peak solutions
  to the fractional {S}chr\"odinger equation with nearly critical exponent in
  $\mathbb R^N$.
\newblock arXiv preprint arXiv:2603.19868, 2026.

\bibitem{guo2017local}
Yuxia Guo, Shuangjie Peng, and Shusen Yan.
\newblock Local uniqueness and periodicity induced by concentration.
\newblock {\em Proceedings of the London Mathematical Society},
  114(6):1005--1043, 2017.

\bibitem{ianni2025morse}
Isabella Ianni, Peng Luo, and Shusen Yan.
\newblock Morse index, topological degree and local uniqueness of multi-spikes
  solutions to the Lane--Emden problem in dimension two.
\newblock arXiv preprint arXiv:2506.12905, 2025.

\bibitem{kato1995perturbation}
Tosio Kato.
\newblock {\em Perturbation Theory for Linear Operators}.
\newblock Springer-Verlag, Berlin, second edition, 1995.

\bibitem{kwasnicki2017ten}
Mateusz Kwa{\'s}nicki.
\newblock Ten equivalent definitions of the fractional Laplace operator.
\newblock {\em Fractional Calculus and Applied Analysis}, 20(1):7--51, 2017.

\bibitem{laskin2002fractional}
Nick Laskin.
\newblock Fractional Schr{\"o}dinger equation.
\newblock {\em Physical Review E}, 66(5):056108, 2002.

\bibitem{LuoPanPeng2024}
Peng Luo, Kefan Pan, and Shuangjie Peng.
\newblock Morse index of concentrated solutions for the nonlinear
  Schr{\"o}dinger equation with a very degenerate potential.
\newblock {\em Calculus of Variations and Partial Differential Equations},
  63:Article No.~164, 2024.

\bibitem{McDuffWehrheim2024Polyfold}
Dusa McDuff and Katrin Wehrheim.
\newblock Polyfold fundamental classes and globally structured multivalued
  perturbations.
\newblock arXiv preprint arXiv:2406.15176, 2024.

\bibitem{radulescu2016variational}
Giovanni Molica~Bisci, Vicen{\c{t}}iu~D. R{\u{a}}dulescu, and Raffaella
  Servadei.
\newblock {\em Variational Methods for Nonlocal Fractional Problems}.
\newblock Cambridge University Press, Cambridge, 2016.

\bibitem{pacella2002symmetry}
Filomena Pacella.
\newblock Symmetry results for solutions of semilinear elliptic equations with
  convex nonlinearities.
\newblock {\em Journal of Functional Analysis}, 192(1):271--282, 2002.

\bibitem{pacella2007symmetry}
Filomena Pacella and Tobias Weth.
\newblock Symmetry of solutions to semilinear elliptic equations via Morse
  index.
\newblock {\em Proceedings of the American Mathematical Society},
  135(6):1753--1762, 2007.

\bibitem{reed1978iv}
Michael Reed and Barry Simon.
\newblock {\em Methods of Modern Mathematical Physics. IV: Analysis of
  Operators}.
\newblock Academic Press, New York, 1978.

\bibitem{reed1980methods}
Michael Reed and Barry Simon.
\newblock {\em Methods of Modern Mathematical Physics. I: Functional
  Analysis}.
\newblock Academic Press, New York, revised and enlarged edition, 1980.

\bibitem{RosOtonSerra2016Stable}
Xavier Ros-Oton and Joaquim Serra.
\newblock Regularity theory for general stable operators.
\newblock {\em Journal of Differential Equations}, 260(12):8675--8715, 2016.

\bibitem{silvestre2007regularity}
Luis Silvestre.
\newblock Regularity of the obstacle problem for a fractional power of the
  Laplace operator.
\newblock {\em Communications on Pure and Applied Mathematics}, 60(1):67--112,
  2007.

\bibitem{wei2010infinitely}
Juncheng Wei and Shusen Yan.
\newblock Infinitely many positive solutions for the nonlinear
  Schr{\"o}dinger equations in $\mathbb{R}^N$.
\newblock {\em Calculus of Variations and Partial Differential Equations},
  37(3--4):423--439, 2010.

\bibitem{WuZhang2024}
Yuanda Wu and Yimin Zhang.
\newblock Uniqueness and non-degeneracy of solutions for nonlinear fractional
  Schr{\"o}dinger equation with perturbation.
\newblock {\em Journal of Mathematical Physics}, 65(8):081512, 2024.

\bibitem{yang1998nodal}
Xue-Feng Yang.
\newblock Nodal sets and Morse indices of solutions of super-linear elliptic
  PDEs.
\newblock {\em Journal of Functional Analysis}, 160(1):223--253, 1998.

\end{thebibliography}
\end{document}